\documentclass[12pt,titlepage]{article} 
\usepackage{amsmath}                                                                         
\usepackage{amsfonts,amssymb,amsthm}                                                         
\usepackage[all]{xy}

\newtheorem{theorem}{Theorem}[section]
\newtheorem{lemma}[theorem]{Lemma}          
\newtheorem{proposition}[theorem]{Proposition}

\theoremstyle{definition}
\newtheorem{definition}[theorem]{Definition}
\newtheorem{example}[theorem]{Example}

\theoremstyle{remark}                                                                        
\newtheorem{remark}[theorem]{Remark}    

\numberwithin{equation}{section}                                                           

\newcommand{\com}{\circ} 
\newcommand{\iso}{\cong} 
\newcommand{\heq}{\simeq} 
\newcommand{\dsm}{\oplus} 
\newcommand{\sma}{\wedge_{\mu}} 
\newcommand{\esm}{\wedge_e} 
\newcommand{\wed}{\vee} 
\newcommand{\wdX}{\vee_{X}} 
\newcommand{\btn}{\boxtimes} 
\newcommand{\dop}{\diamond} 
\newcommand{\gro}{ \mspace{-2mu} \textstyle{\int} \mspace{-2mu} }
\newcommand{\id}{\ensuremath{{\rm id}}} 
                                                     
\newcommand{\Loops}{\Omega}

\newcommand{\Cat}{{\rm Cat}}                                                        

\newcommand{\bN}{{\mathbf N}}                                                                  
\newcommand{\bZ}{{\mathbf Z}}
\newcommand{\bR}{{\mathbf R}}
\newcommand{\cC}{{\mathcal C}} 
\newcommand{\cD}{{\mathcal D}}
\newcommand{\cF}{{\mathcal F}}
\newcommand{\cR}{{\mathcal R}}
\newcommand{\cT}{{\mathcal T}}
   
\newcommand{\fset}[1]{\ensuremath{\mathbf{#1}}}                   
\providecommand{\abs}[1]{\lvert#1\rvert}

\DeclareMathOperator{\Iso}{Iso}
\DeclareMathOperator*{\colim}{colim}                                                         
                                                         
\DeclareMathOperator*{\hocolim}{hocolim}                                                     
                                                          
\DeclareMathOperator{\fib}{fiber}

\DeclareMathOperator{\im}{Image}
\DeclareMathOperator{\sub}{sub}
\DeclareMathOperator{\Ind}{Ind}

\DeclareMathOperator{\Exact}{Exact}
\DeclareMathOperator{\Ob}{Ob}
\DeclareMathOperator{\rel}{rel}
\newcommand{\la}{\leftarrow}                                                                 
\newcommand{\lla}{\longleftarrow}                                                            
\newcommand{\ra}{\rightarrow}                                                                
\newcommand{\lra}{\longrightarrow}                                                           
\newdir{ >}{{}*!/-1ex/\dir{>}}    
\newcommand{\UCMT}{\SelectTips{cm}{}}
\begin{document}                                                                             
\title{Segal Operations in the Algebraic $K$-theory\\of Topological Spaces}
\author{
Thomas Gunnarsson 
\thanks{We thank Friedhelm Waldhausen for suggesting this problem, and we thank
the Institut f\"{u}r Mathematik of the University of Osnabr\"{u}ck for its hospitality
and Institute faculty members for their interest and encouragement.}
\\
Department of Mathematics
\\
Lule\aa\ University of Technology
\\
97187 Lule\aa, Sweden
\and 
Ross Staffeldt 
\\Department of Mathematical Sciences\\
New Mexico State University\\
Las Cruces, NM 88003 USA}
\maketitle
\begin{abstract}
We extend earlier work of Waldhausen which defines operations 
on the algebraic $K$-theory of the one-point space.
For a connected simplicial abelian group $X$ and symmetric groups $\Sigma_n$, we  define
operations $\theta^n \colon A(X) \ra A(X{\times}B\Sigma_n)$ in the algebraic $K$-theory
of spaces.
We show that our operations can be given the structure of $E_{\infty}$-maps. 
Let $\phi_n \colon  A(X{\times}B\Sigma_n) \ra A(X{\times}E\Sigma_n) \simeq A(X)$
be the $\Sigma_n$-transfer.
We also develop an inductive procedure to compute the compositions
$\phi_n \com \theta^n$, and outline some applications. 
\end{abstract}
\section{Introduction} 
  \label{Introduction}
Let $X$ be a connected simplicial abelian group, let $\Sigma_n$ be the symmetric group on $n$ letters, 
and let $B\Sigma_n$ be the classifying space.  
Our goal is to define a family of Segal operations
\begin{equation*}
\theta^n \colon A(X) \lra A(X \times B\Sigma_n)
\end{equation*}
satisfying the properties listed in theorems \ref{mainthm1} and \ref{gencalcintro} below. 
We follow Waldhausen \cite{Waldhausen82} in our naming convention,
which can be explained as follows. 
Around 1972 Graeme Segal \cite{Segalops} defined a set of operations
in stable homotopy theory
$\theta^n \colon \pi^s_i(S^0) \lra \pi^s_i\bigl( (B{\Sigma_n})_+ \bigr)$, verified certain
properties and used the
information to give a proof of the Kahn-Priddy theorem.  
The key to the Kahn-Priddy proof is a certain relation 
satisfied by the composition of an operation followed by a transfer homomorphim. 

Waldhausen \cite{Waldhausen82} adapted the construction in 
\cite{Segalops} to define operations 
 $\theta^n \colon A(*) \lra A(B\Sigma_n)$ and
 proved these new operations have properties precisely analogous to fundamental properties
of Segal's original operations.  Consequently, Waldhausen used the same
notation and called the operations ``Segal operations.'' 

We obtain the following result.
\begin{theorem}\label{mainthm1}
For a connected simplicial abelian group $X$, there are maps $\theta^n \colon A(X) \ra A(X \times B\Sigma_n)$ which have the
following properties.
\begin{enumerate}
\item The map $\theta^1$ is the identity.
\item The combined map
\begin{equation*}
\theta = \prod_{n \geq 1} \theta^n \colon A(X) \lra \{ 1 \} \times \prod_{n \geq 1} A(X \times B\Sigma_n)
\end{equation*}
has the structure of an $E_{\infty}$-map if the target is equipped with the 
$E_{\infty}$-structure arising from certain pairings
$A(X \times B\Sigma_m) \times A(X \times B\Sigma_n) \ra A(X \times B\Sigma_{n+m})$
derived from the box-tensor operation of definition \ref{boxtensor}.
\end{enumerate}
\end{theorem}
The first property is a normalization condition, as satisfied by the
constructions of Segal and Waldhausen.  The second property implies that
for every $j> 0$ the operations induce homomorphisms 
$\pi_jA(X) \ra \{1\} \times \prod_{n \geq 1} \pi_jA(X{\times}B\Sigma_n)$,
when the target is given a particular algebraic structure. 
A third, algebraic, property
of Waldhausen's operations  is recalled in 
proposition \ref{modelrecursion}.   This third property is crucial in the applications
made by Segal and Waldhausen.  
Our extended operations exhibit a more technical algebraic property stated in 
 theorem \ref{gencalcintro} and theorem \ref{generalcalculation}.  

A large part of our work follows
\cite{GSAthyops}
in which many results are developed for quite general situations, 
satisfying certain technical conditions.  Part of this paper
verifies these conditions.  In order to explain 
the necessity of this technical work, we repeat several definitions from
\cite{GSAthyops} 
and quote many results.  

In section \ref{General} the main results are
proposition \ref{externalbiexactness}
and theorem \ref{symmetricbimonoid}.  For the purposes of
algebraic $K$-theory we verify exactness properties of certain
constructions; to prepare for the $E_{\infty}$-structure we 
verify coherence properties.

In section \ref{Extension} 
we recall the $G_{\bullet}$-construction
for algebraic $K$-theory
\cite{GSVWKthy}, \cite{GraysonOps}
and prepare the constructions underlying the definition of the operations in
definition \ref{omegafunctors}.

In section \ref{EinftyStructure} we 
set up to apply  general machinery, 
taking the first step toward a main result: 
For $X$ a connected simplicial abelian group, there is an operation
\begin{equation} \label{operation}
\omega = \prod_{n \geq 1} \omega^n \colon A(X) \lra \{ 1 \} \times \prod_{n \geq 1} A_{\Sigma_n , \{ {\rm all}\}} (X)
\end{equation}
which is a map of $E_{\infty}$-spaces with respect to specific algebraic structures described
in  section \ref{EinftyStructure}. The target of $\omega$ is the algebraic $K$-theory of $\Sigma_n$-spaces
retracting to $X$ (with the trivial $\Sigma_n$-action) and relatively finite with respect to $X$.
See definition \ref{equivwcofandwe}.  In the first step the $E_{\infty}$-structure is only visible
if we restrict to spherical objects.  The next section addresses this problem.

In section \ref{Suspension} we study how the functors from definition 
\ref{omegafunctors}
interact with suspension operators.  At the end of the section we complete the construction of 
the operation displayed in equation \eqref{operation}.

In  theorem \ref{splitting}, we split  $A_{\Sigma_n, \{ {\rm all}\}} (X)$ as a product of
the algebraic $K$-theory of other spaces, one of which is $A(X{\times}B\Sigma_n)$.  This corresponds to 
the subcategory of $\Sigma_n$-spaces retracting to $X$ (with the trivial $\Sigma_n$ action), relatively 
finite with respect to $X$, and with  $\Sigma_n$ acting freely outside of $X$.
We also obtain an expression for the composite functors ``projecting to the free part''
\begin{equation*}
\theta^n \colon  A(X) \stackrel{\omega_n}{\lra} A_{\Sigma_n, \{ {\rm all}\}} (X) \lra A(X{\times}B\Sigma_n).  
\end{equation*}
This expression is used in section \ref{Formulas}.

In section \ref{Transfer} we establish equivalences among various models for
equivariant $K$-theory and discuss the functors that induce transfer operations.

In section \ref{Formulas} 
our main computational result evaluates the composition 
\begin{equation*}
  A(X) \stackrel{\theta^n}{\lra} A(X{\times}B\Sigma_n) \stackrel{\phi_n}{\lra} A(X),
\end{equation*}
where $\phi_n$ is the transfer map. 
\begin{theorem}[Theorem \ref{generalcalculation}]\label{gencalcintro}
Let $X$ be a connected simplicial abelian group, thinking of the group
operation as a multiplication, 
and let $\tau^n \colon X \ra X$ be the homomorphism that raises elements to the
$n$th power.  Then
\begin{equation*}
  \phi_n \theta^n_* = (-1)^{n-1} \cdot (n{-}1)! \cdot \tau^n_*
    \colon \pi_jA(X) \ra \pi_jA(X)
\end{equation*}
for $j >0$. 
\end{theorem}

We conclude this introduction with some comments on applications. 
First, we recall one formulation of the Kahn-Priddy theorem in stable homotopy theory.
Let $Q(X) = \Loops^{\infty}S^{\infty}(X_+)$ denote unreduced stable homotopy theory and
define reduced stable homotopy theory
$\widetilde{Q}(X) = \fib(Q(X) \ra Q(*))$, the homotopy fibre.  For each $n$ there is a
transfer map $Q(B\Sigma_n )\ra Q(E\Sigma_n) \heq Q(*)$, and, by composition, there
results a map $\widetilde{Q}(B\Sigma_n) \ra Q(*)$. The formulation of Kahn-Priddy 
theorem that we prefer is that the map
\begin{equation*}
  \pi_j(\widetilde{Q}B\Sigma_p)_{(p)} \lra \pi_j(Q(*))_{(p)}
\end{equation*}
 of homotopy groups localized at a prime $p$ is surjective for $j > 0$.

Walhausen's analogue of this result applies to the algebraic $K$-theory of the one-point 
space.  For the formulation we let $A(X)$ denote the algebraic $K$-theory of the space $X$
and let $\widetilde{A}(X) = \fib(A(X) \ra A(*))$ be the  algebraic $K$-theory of $X$ reduced relative to a point.
Manipulations formally similar to those above yield a map 
$\widetilde{A}(B\Sigma_n ) \ra A(*)$ and the analogue of the Kahn-Priddy theorem is that
the induced map
\begin{equation*}
  \pi_j(\widetilde{A}(B\Sigma_p)_{(p)} \lra \pi_j(A(*))_{(p)}
\end{equation*}
of homotopy groups localized at $p$ is surjective for $j > 0$.
In
\cite{Waldhausen87}
these operations are further developed and used to prove that the third factor $\mu(X)$ in the splitting
\begin{equation*}
  A(X) \simeq Q(X_+) \times Wh^{{\rm Diff}}(X) \times \mu(X)
\end{equation*}
is contractible, yielding the final result
  $A(X) \simeq Q(X_+) \times Wh^{{\rm Diff}}(X)$.
The significance of this fact is developed in \cite{PLmfds}.

In our situation we fix as base space a connected simplicial abelian group $X$ and define reduced 
algebraic $K$-theory relative to $X$ as 
$\widetilde{A}(X{\times}B\Sigma_n \rel X) = \fib(A(X{\times}B\Sigma_n) \ra A(X))$. 
The inclusion of a point into $B\Sigma_n$ combined with the definition of the
algebraic $K$-theory of $X{\times}B\Sigma_n$ reduced relative to $X$ yields a splitting
\begin{equation} \label{basicsplitting}
  \pi_jA(X{\times}B\Sigma_n) \iso \pi_j\widetilde{A}(X{\times}B\Sigma_n \rel X) \dsm \pi_jA(X)
\end{equation}
for any $j \geq 0$.
We have transfer maps $\phi_n \colon A(X{\times}B\Sigma_n) \ra A(X{\times}E\Sigma_n) \heq A(X)$ 
and a basic calculation in lemma \ref{transferproperty} that the composition
\begin{equation*}
  A(X) \ra A(X{\times}B\Sigma_n) \stackrel{\phi_n}{\lra} A(X)
\end{equation*}
is multiplication by $n! = \abs{\Sigma_n}$, 
where the first map is induced by inclusion of a point into $B\Sigma_n$. 

When we specialize $n$ to a prime number $p$, we have the following observations.
Make the following diagram of homotopy groups reduced mod $p$, where the
splitting \eqref{basicsplitting} appears as the middle column.
The diagonal arrow from the bottom row is multiplication by $p! = \abs{\Sigma_p}$, which is $0$ modulo $p$. 
Thus, in terms of the splitting of $\pi_jA(X{\times}B\Sigma_p)/p\bZ$ given above, 
on the second component of the image of $\theta^p_*$, the map $\phi_{p*}$ is zero.  
\begin{equation*}
  \xy \UCMT \xymatrix{
   &    \pi_j\widetilde{A}(X{\times}B\Sigma_p \rel X)/p\bZ \ar[d] \ar[dr]^{\phi_{p*}} &  
\\
  \pi_jA(X)/p\bZ   \ar[r]^(0.45){\theta^p_*} &   \pi_jA(X{\times}B\Sigma_p)/p \bZ   \ar[r]^(0.6){\phi_{p*} }    &   \pi_jA(X)/p\bZ .
\\
  &  \pi_jA(X)/p\bZ \ar[u]_{i_*} \ar[ur]_{0 } }
\endxy
\end{equation*}
Applying theorem \ref{generalcalculation},  $\phi_{p*}$ applied to the first component 
$\pi_j\widetilde{A}(X{\times}B\Sigma_p \rel X)/p$ 
of the splitting contains the image of
 $\phi_{p*}\theta^p_* = (-1)^{p-1} \cdot (p{-}1)! \cdot \tau^p_*$,
where $\tau^p \colon X \ra X$ raises elements to the $p$th power.
The numerical factors are invertible mod $p$ so that
\begin{equation*}
 \phi_{p*}\bigl( \pi_j\widetilde{A}(X{\times}B\Sigma_p \rel X)/p\bZ \bigr) \supset \im \tau^p_*,
\end{equation*}
viewing $\tau^p_*$ as an endomorphism of $\pi_jA(X)/p\bZ$.

From these calculations one extracts various additional observations.  It may happen that
the $p$th power homomorphism $\tau^p$ is an isomorphism, as in the case 
when $X$ is a connected simplicial abelian group, finite in each simplicial dimension
and $p$ is relatively prime to the order of $X_n$  for each $n$.
Then for $j > 0$, 
\begin{equation*}
  \phi_{p*} \colon  \pi_j\widetilde{A}(X{\times}B\Sigma_p \rel X)/p\bZ \lra \pi_jA(X)/p\bZ
\end{equation*}
is surjective.  The next input is the following theorem. 
\begin{theorem}
  [Theorem, \cite{Betley86}] 
If $\pi_1(X)$ is a finite group, and $\pi_i(X)$ is finitely generated for all $i \geq 2$,
then $\pi_j\bigl( A(X) \bigr) $ is finitely generated for all $j$.
\end{theorem}
Then Nakayama's lemma applies as in 
\cite{Waldhausen82} to lift the  result on  mod $p$ homotopy to a result on $p$-localized
homotopy. We obtain the following theorem of Kahn-Priddy type. 
\begin{theorem} \label{KPgeneral}
Let $X$ be a connected simplicial abelian group, finite in each dimension, such that the order
of $X_n$ is prime to $p$. 
  For $j > 0$ and $p$ an odd prime, the transfer induces surjections
  \begin{equation*}
    \pi_j\widetilde{A}(X{\times}B\Sigma_p \rel X)_{(p)} \lra \pi_jA(X)_{(p)}.
  \end{equation*}
on homotopy groups localized at $p$. \qed
\end{theorem}
 In particular, take $X = BC_2= \bR P^{\infty}$
and $p$  an odd prime.
There are similar statements for all the lens spaces $BC_q$, $q$ prime to $p$.

A very interesting case is $X = BC_{\infty}$, the classifying space of the infinite cyclic group $C_{\infty}$. 
Of course $X \heq S^1$, and there are splittings-up-to-homotopy of infinite loop spaces
\begin{align*}
  A^{fd}(S^1) & \heq A^{fd}(*) \times BA^{fd}(*) \times N_-A^{fd}(*) \times N_+A^{fd}(*)
\intertext{and}
 A^{fd}(S^1{\times}B\Sigma_n) &
 \heq A^{fd}(B\Sigma_n) \times BA^{fd}(B\Sigma_n) \times N_-A^{fd}(B\Sigma_n) \times N_+A^{fd}(B\Sigma_n).
\end{align*}
These are studied in 
\cite{KWFT}
and the first is examined in great detail in
\cite{GKMNilOps}.
In future work we would like to understand the  operations we have constructed in terms of these splittings. 
As a first step in this direction we have shown in section \ref{EinftyStructure} that the
operations we construct are morphisms of infinite loop spaces.  
Should the $\theta$ operations be compatible with the splitting, one must then
investigate whether or not the $\theta$ operations commute with the Frobenius and Verschiebung 
operations on the Nil-terms defined in \cite{GKMNilOps}.

Our work also admits a generalization where $X$ may be any connected space.  This result is a total operation
\begin{equation*}
  \tilde{\omega} \colon  A(X) \lra \{ 1 \} \times \prod_{n \geq 1} A_{\Sigma_n, \{ {\rm all} \}} (X^n),
\end{equation*}
about which we know little at this point. Our experiments have also lead to the observation that if $G$ is a 
simplicial group, not necessarily abelian, whose realization is homotopy equivalent to a finite $CW$-complex
then there is a product structure on $A(BG)$.   This will be the subject of a later paper.
Finally, reversing the progression from Segal's original idea to Waldhausen's generalization, we can develop
operations $\theta^n \colon \pi^s_* (X_+) \ra \pi^s_*(X{\times}B\Sigma_n)$, where $X$ is again a connected 
simplicial abelian group.
\section{The symmetric bimonoidal category of retractive \\
spaces  over a connected simplicial  abelian group}
 \label{General} 
The category $\cR(X)$ is the category of retractive simplicial sets 
$(Y, r, s)$ over the simplicial set $X$, where
$r \colon Y \ra X$ is a retraction,  $s \colon X \ra Y$
is a section for $r$ and morphisms $(Y_1, r_1, s_1) \ra (Y_2, r_2, s_2)$
respect all the data.  
 A cofibration
$\xy \UCMT \xymatrix@1{(Y_1, r_1, s_1) \ar@{ >->}[r] & (Y_2, r_2, s_2)} \endxy$ 
in $\cR_f(X)$ is a map such that $Y_1 \ra Y_2$ is injective.   A weak
equivalence is a map 
$(Y_1, r_1, s_1) \ra (Y_2, r_2, s_2)$ whose realization $\abs{Y_1} \ra \abs{Y_2}$
is a homotopy equivalence of spaces. 
For algebraic $K$-theory we use the full subcategory
$\cR_f(X)$ of relatively finite retractive simplicial sets with 
cofibrations and weak equivalences. ``Relatively finite''  means 
that there are only finitely many non-degenerate simplices in 
$Y{-}X$.  For background on the terminology, see \cite[Section 1.1]{Waldhausen85}.

We aim to construct a total operation
\begin{equation*}
\theta \colon  A(X) \lra \{1\} \times \prod_{n \geq 1} A(X{\times}B\Sigma_n)
\end{equation*}
for $X$ a connected simplicial abelian group with multiplication 
$\mu \colon X{\times}X \ra X$
and to prove the operation has an $E_{\infty}$-structure. 
In order to achieve this, the elements from which the construction 
is developed must be of high quality.  
The necessary qualities are recorded in the first part of theorem \ref{symmetricbimonoid};
the second part of the theorem records algebraic properties of the 
product operation $\sma$.  We discuss first the definition of the 
product operation, prove the second part of the theorem, 
and finish this section with the proof of the first part of the theorem.

Concerning the first part of the theorem, our constructions require a coherence result for diagrams
involving sum and product operations, as provided by LaPlaza
\cite[Proposition 10]{LaPlazaCoh1}.
His coherence theorem takes as input the commutativity of
24 diagrams, reducible to a smaller, but still relatively large, subset
\cite[pp.~40--41]{LaPlazaCoh1}.  We will see that the coherence properties we need 
rest on the well-understood coherence properties of the one-point union and smash product
of pointed sets.  On the other hand, the second part of the theorem involves properties
of the operations not reducible to dimension-wise considerations. 
\begin{theorem} \label{symmetricbimonoid}  
Let $X$ be a connected simplicial abelian group. 
\begin{enumerate}
\item The triple  $\bigl( \cR(X) , \wdX, \sma \bigr)$,
where $\wdX$ denotes the operation of union along the common subspace $X$
and $\sma$ denotes the pairing \eqref{abelianpairing}, 
is a symmetric bimonoidal category. 
\item The pairing $\sma$ restricts to $\cR_f(X)$, where it is biexact, meaning exact in each variable separately.  
Explicitly, the functors defined by $-\sma Y$ and $Y \sma -$ preserve cofibrations, pushouts along cofibrations, 
 and weak equivalences.
\end{enumerate}
\end{theorem}
Our product operation $\sma$ derives from an exterior smash product $\esm$ of retractive simplicial sets,
following the exterior smash product of retractive spaces as described in 
\cite{MaySigurdssonPHT}.  Since we are working with simplicial sets, our version of the exterior
smash product has a description in terms of operations on discrete sets, applied dimensionwise. 
See the discussion at the start of the proof of part one of theorem \ref{symmetricbimonoid}.
\begin{definition}
  \label{externalsmash}
Let $(Y_i, r_i, s_i)$ be  objects of $\cR(X_i)$, for $i =1, 2$. The
exterior smash product of $(Y_1, r_1, s_1)$ with $(Y_2, r_2, s_2)$ is in $\cR(X_1{\times}X_2)$, and the underlying space 
$Y_1 \esm Y_2$ completes the following square to a pushout. 
\begin{equation}
  \label{extsmashprod}
  \xy \UCMT \xymatrix{
Y_1{\times}X_2 \cup_{X_1{\times}X_2} X_1{\times}Y_2   \ar@{ >->}[r] \ar[d]_{r_1{\times}r_2} & Y_1 \times Y_2 \ar[d]
\\
X_1 \times X_2    \ar@{ >->}[r]^{s_1\esm s_2}                                      & Y_1 \esm Y_2}
\endxy
\end{equation}
The square displays the section $s_1\esm s_2$; the retraction $r_1 \esm r_2$ arises from the 
universal property of the pushout.
\end{definition}
  Note that if both $X_1$ and $X_2$ are the one-point space, then this 
is the smash product in the category of pointed spaces.  Extending this idea, if 
$x_1 \colon \{*\} \ra X_1$ and $x_2 \colon \{*\} \ra X_2$ are two maps of the one-point space into
$X_1$ and $X_2$, and we take preimages $r_1^{-1}(x_1)$ and $r_2^{-1}(x_2)$, then these are pointed spaces,
and there is an injective map 
$r_1^{-1}(x_1) \wedge r_2^{-1}(x_2) \ra Y_1\esm Y_2$ 
over the point $(x_1, x_2) \in X_1{\times}X_2$.  This observation helps
explain the ``fiber-wise smash product'' terminology and indicates how the coherence issues for products may be
resolved at the level of pointed sets.  The examples
\ref{zeroobject} and \ref{esmaction} here play roles in the proof of part one of theorem \ref{symmetricbimonoid}.  
Also, since we work with simplicial sets, underlying the symmetric monoidal structure 
 $\bigl( \cR(X) , \wdX, \sma \bigr)$ is the symmetric monoidal structure on the category of sets.
\begin{example} \label{zeroobject}
  For any $Y_2 \in R(X_2)$, note that $X_1 \esm Y_2 \iso X_1 \times X_2$, the ``zero'' object in $\cR(X_1 \times X_2)$.
Colloquially, the exterior smash product of a terminal object with any object yields a terminal object.  Explicitly,
a natural isomorphism $\lambda^*_{Y_2} \colon X_1 \esm Y_2 \ra X_1 \times X_2$ arises from the following diagram
by mapping the pushout of the top row to the pushout of the bottom row.
\begin{equation*}
  \xy \UCMT \xymatrix{
X_1{\times}X_2 \ar[d]_{=} &   \ar[l]   X_1{\times}X_2\cup_{X_1{\times}X_2} X_1{\times}Y_2 \ar@{ >->}[r] \ar[d]_{\iso} & X_1{\times}Y_2 \ar[d]_{=}
\\
X_1{\times}X_2 &   \ar[l]    X_1{\times}Y_2 \ar@{ >->}[r] & X_1{\times}Y_2}
\endxy
\end{equation*}
\end{example}
\begin{example}\label{esmaction}
The bifunctor $\cR(*)\times \cR(X) \ra \cR(X)$ given on objects by 
$(Y_1, Y_2)\mapsto Y_1{\esm}Y_2$ defines an action of
$\cR(*)$ on $\cR(X)$ after identifying $\{*\}{\times}X$ with $X$ in the canonical way.
 This bifunctor also restricts to 
an action $\cR_f(*) \times \cR_f(X) \ra \cR_f(X)$.

This action has an identity element.  Indeed, for $S^0=\{*,*'\}$ in $\cR_f(*)$, with $r$ the constant map to the basepoint $*$,
$s$ the inclusion, and $Y \in \cR(X)$, the function
$S^0{\times}Y\ra Y$ defined by $(*, y) \mapsto sr(y)$ and $(*',y) \mapsto y$ induces an isomorphism
$S^0{\esm}Y \stackrel{\iso}{\ra}  Y$ of retractive spaces over $X$.  An inverse to this isomorphism is provided by
$y \mapsto [(*', y)] \in S^0{\esm}Y$.
\end{example}
\begin{definition}
  \label{internalsmashproduct}
Let  $X$ be a space with a multiplication $\mu \colon X^2 \ra X$.
We operate on the category $\cR(X)$,
using the pairing
\begin{equation}  \label{abelianpairing}
\sma = \mu_*\com\esm \colon \cR( X ) \times \cR (X) \stackrel{\esm}{\lra}
               \cR(X \times X) \stackrel{\mu_*}{\lra} 
                     \cR(X)
\end{equation}
where $\esm$ is the external smash product pairing defined in \eqref{extsmashprod},
and $\mu_*$ is the functor induced by the multiplication 
$\mu \colon X^2 \ra X$. 
Explicitly, $(Y_1, r_1, s_1) \sma (Y_2, r_2, s_2)$ completes the following diagram to a pushout.
\begin{equation}
  \label{intsmashprod}
  \xy \UCMT \xymatrix{
 \ar[d]_{\mu  (r_1, \id)\cup \mu (\id, r_2)}  Y_1{\times}X \cup_{X{\times}X} X{\times}Y_2   \ar@{ >->}[r] & Y_1 \times Y_2 \ar[d]
\\
 X   \ar@{ >->}[r]^s    &     Y_1 \sma Y_2 }
\endxy
\end{equation}
\end{definition}
 We use these notations to bring this section close to 
conformity with 
\cite{GSAthyops}.
Perfect conformity is not possible, for we must use both the
one-point union of pointed spaces $\wed$ and the union of two spaces along a common subspace $X$, 
denoted $\wdX$.  We also point out that the usual  notation $\wedge$ has been used in 
\cite{MaySigurdssonPHT}
for a product defined by restricting the external smash product of two spaces over $X$ to
the diagonal of $X{\times}X$.

The following lemma is used to develop properties of the smash products; the proof will be given
 after demonstrating applications in propositions \ref{extprodnaturality} and \ref{externalbiexactness}.
\begin{lemma} \label{iteratedcolimlemma}
Let  $\cC$ be a category  with cofibrations and let 
  \begin{equation}
    \label{iteratedcolim}
    \xy \UCMT \xymatrix{  
  A_2  &  \ar[l]  B_2  \ar@{ >->}[r] & C_2
\\
  A_1   \ar@{ >->}[u] \ar[d] 
            &  \ar[l]  B_1  \ar@{ >->}[r] \ar@{ >->}[u] \ar[d]  
                     &   C_1 \ar@{ >->}[u] \ar[d] 
\\
  A_0      &  \ar[l]   B_0   \ar@{ >->}[r] &  C_0 
}
\endxy
  \end{equation}
be a commutative diagram in which the canonical map from $B_2\cup_{B_1} C_1$ to $C_2$ is a cofibration.  

Passing to pushouts by columns results in a diagram in which the right-pointing arrow is a cofibration:
\begin{equation}
  \label{iteratedcolim1}
  \xy \UCMT \xymatrix{
A_0 \cup_{A_1} A_2 &  \ar[l] B_0 \cup_{B_1} B_2  \ar@{ >->}[r] &  C_0\cup_{C_1} C_2.}
\endxy
\end{equation}
The diagram
\begin{equation}
  \label{iteratedcolim2}
  \xy \UCMT \xymatrix{
A_0 \cup_{B_0} C_0 &  \ar[l] A_1 \cup_{B_1} C_1  \ar@{ >->}[r] &  A_2\cup_{B_2} C_2}
\endxy
\end{equation}
 obtained by passing to pushouts by rows has a similar property. 
\end{lemma}
\begin{proposition} \label{extprodnaturality}
The exterior smash product $\esm$ is functorial for pairs of maps. 
That is, given $f_1 \colon X_1 \ra X_1'$ and $f_2 \colon X_2 \ra X_2'$, the diagram
\begin{equation}
  \label{functoriality}
  \xy \UCMT \xymatrix{
\cR_f(X_1) \times \cR_f(X_2) \ar[r]^(0.57){\esm} \ar[d]_{f_{1*} \times f_{2*}}   & \cR_f(X_1{\times}X_2)  \ar[d]_{(f_1 \times f_2)_*}
\\
\cR_f(X_1') \times \cR_f(X_2') \ar[r]^(0.57){\esm} & \cR_f(X_1'{\times}X_2')
}
\endxy
\end{equation}
commutes up to natural isomorphism.
\end{proposition}
\begin{proof}
For the naturality property of the external smash product, consider the diagram
\begin{equation}
  \label{naturality}
  \xy \UCMT \xymatrix{
X_1{\times}X_2  
           &    \ar[l]  Y_1{\times}X_2 \cup_{X_1\times X_2} X_1{\times}Y_2 \ar@{ >->}[r]   
                                & Y_1{\times}Y_2
\\
X_1{\times}X_2   \ar@{ >->}[u] \ar[d]_{f_1\times f_2} 
           &  \ar[l]  X_1{\times}X_2     \ar@{ >->}[u]  \ar@{ >->}[r] \ar[d]_{f_1\times f_2}  
                               & X_1{\times}X_2  \ar@{ >->}[u] \ar[d]_{f_1\times f_2}
\\
X_1'{\times}X_2'        
            &   \ar[l]  X_1'{\times}X_2'    \ar@{ >->}[r]
                                & X_1'{\times}X_2'
}
\endxy
\end{equation}  
which fulfills the hypotheses of lemma \ref{iteratedcolimlemma}.  Computing the colimits of the columns in this diagram
yields the diagram
\begin{equation*}
  \xy \UCMT \xymatrix{
  X_1'{\times}X_2' & \ar[l]_(0.72){r_1'\times r_2'}  (f_{1*}Y_1) {\times} X_2' \cup_{X_1'\times X_2'}X_1'{\times} (f_{2*}Y_2) \ar@{ >->}[r] & f_{1*}Y_1 \times f_{2*}Y_2,
}
\endxy
\end{equation*}
whose pushout is by definition $ f_{1*}Y_1 \esm f_{2*}Y_2 $.

On the other hand computing the colimits of the rows in the diagram yields the diagram
\begin{equation*}
  \xy \UCMT \xymatrix{
  X_1'{\times}X_2' &  \ar[l]_{f_1 \times f_2} X_1{\times}X_2  \ar@{ >->}[r] &   Y_1{\esm}Y_2,
}
\endxy
\end{equation*}
whose pushout is $(f_1{\times}f_2)_*(Y_1\esm Y_2)$. Since both iterative procedures compute the colimit of  diagram \eqref{naturality},
they are canonically isomorphic:
\begin{equation*}
   f_{1*}Y_1 \esm f_{2*}Y_2 \iso (f_1{\times}f_2)_*(Y_1\esm Y_2). \qedhere
\end{equation*}
\end{proof}
As a consequence, we have the following result.
\begin{proposition} \label{internalizing}
Let $X$ be a monoid with unit.  
The action of $\cR(*)$ on $\cR(X)$ set up in example \ref{esmaction} may be 
made internal to $\cR(X)$. Diagramatically, the diagram 
\begin{equation*}
  \xy \UCMT \xymatrix{
\cR(*) \times \cR(X) \ar[r]^{i_{e*}\times \id}  \ar[dr]^{\esm}  & \cR(X) \times \cR(X) \ar[d]^{\sma}
\\
  &  \cR(X) }
\endxy
\end{equation*}
commutes up to natural isomorphism. 
\end{proposition}
\begin{proof}
  Let $i_e \colon \{*\} \ra X$ be the inclusion of the one point space to the identity element of the monoid $X$.
The functor $i_{e*} \colon \cR(*) \ra \cR(X)$ sends a pointed retractive space $Y$ to $X \wed Y$, where the
base point of $Y$ is identified to the unit element of $X$.  The new retraction collapses $Y \subset X \wed Y$ to the identity 
$\{e\}$ in $X$.  We have the diagram
\begin{equation*}
  \xy \UCMT \xymatrix{
\cR(*) \times \cR(X) \ar[r]^(0.55){\esm} \ar[d]_{i_{e*}\times \id}& \cR(\{*\}{\times} X) \ar[d]_{(i_e{\times} \id)_*} \ar[dr]^{p_{2*}}
\\
\cR(X) \times \cR(X) \ar[r]^(0.55){\esm} &  \cR(X{\times} X) \ar[r]^(0.55){\mu_*}  & \cR(X).
}
\endxy
\end{equation*}
The lefthand square commutes by proposition \ref{extprodnaturality}, and the righthand triangle commutes because $e$ is the
monoid identity.
The bottom row defines $\sma$ and the trip across the top defines the action of $\cR(*)$ on $\cR(X)$. 
\end{proof}
For example, this result has the consequence that coherent associativity for $\sma$ 
implies corresponding coherent associativity for the $\esm$ action of $\cR(*)$ on $\cR(X)$. 

Next we record the biexactness property of the external smash product as defined
in the statement of theorem \ref{symmetricbimonoid}. 
\begin{proposition}
  \label{externalbiexactness}
The external smash product functor 
\begin{equation*}
\esm \colon  \cR_f(X_1) \times \cR_f(X_2) \lra \cR_f(X_1{\times}X_2)
\end{equation*}
is biexact.   
\end{proposition}
\begin{remark}
  In the approach of \cite{MaySigurdssonPHT} the external smash product is shown to preserve all colimits 
by exhibiting a left adjoint functor.  Their approach uses properties of convenient categories of topological spaces. 
\end{remark}
  For our applications in algebraic $K$-theory it seems more 
reasonable to give arguments modeled on those of
\cite[Lemma 1.1.1]{Waldhausen85},
which serve to illuminate other constructions we make.
\begin{proof}[Proof of proposition \ref{externalbiexactness}]
 For simplicial sets, cofibrations are precisely the injections.  Given a pair of cofibrations
\begin{equation*}
  \xy \UCMT \xymatrix{ (W_1, r_1, s_1) \ar@{ >->}[r] & (W_1', r'_1, s'_1)} \endxy
 \quad \text{and} \quad
  \xy \UCMT \xymatrix{ (W_2, r_2, s_2) \ar@{ >->}[r] & (W'_2, r'_2, s'_2)} \endxy 
\end{equation*}
in $\cR_f(X_1)$ and $\cR_f(X_2)$, respectively, 
the maps of differences of simplicial sets $W_1{-}X_1 \ra W_1'{-}X_1$ and $W_2{-}X_2 \ra W'_2{-}X_2$ are injective maps of 
sets in each simplicial dimension.   The product of these maps is also injective.
Since $(W_1{\esm}W_2){-}X_1{\times}X_2 = (W_1{-}X_1) \times (W_2{-}X_2)$, it follows
that 
$\xy \UCMT \xymatrix{W_1 \esm W_2 \ar@{ >->}[r]  &  W_1' \esm W'_2} \endxy$
is also a cofibration. Finally, if  $W_1{-}X_1$ and $W_2{-}X_2$ contain only finitely many 
non-degenerate simplices, then the same is true of their product.  Thus, 
the pairing $\esm$ restricts to a pairing of $\cR_f(X_1) \times \cR_f(X_2)$ to $\cR_f(X_1{\times}X_2)$. 

To prove that the functor $Z\esm (-) \colon R_f(X_2) \ra R_f(X_1{\times}X_2)$ preserves pushouts of cofibrations, 
start by considering the diagram
\begin{equation} \label{biexactnessmaindiagram}
  \xy \UCMT \xymatrix{
 X_1{\times} X_2  \ar@{=}[d]    &  \ar[l]  Z{\times}X_2 \cup_{X_1{\times}X_2} X_1{\times}Y_2   \ar@{ >->}[r] &  Z \times Y_2 
\\
 X_1{\times}X_2    \ar@{=}[d]   
            &  \ar[l]  Z{\times}X_2 \cup_{X_1{\times}X_2} X_1{\times}Y_1   \ar@{ >->}[r] \ar@{ >->}[u] \ar[d]  
                     &  Z \times Y_1 \ar@{ >->}[u] \ar[d] 
\\
 X_1{\times}X_2      &  \ar[l]  Z{\times}X_2 \cup_{X_1{\times}X_2} X_1{\times}Y_0   \ar@{ >->}[r] &  Z \times Y_0 
}
\endxy
\end{equation}
where the right-pointing arrows are induced from the retractions and the left-pointing arrows are induced by inclusions.
We verify the cofibration hypothesis of lemma \ref{iteratedcolimlemma} using the following diagram to analyze the upper
righthand corner of diagram \eqref{biexactnessmaindiagram}.
\begin{equation*}
  \xy \UCMT \xymatrix{ 
& Z{\times}X_2\cup_{X_1{\times}X_2} X_1{\times}Y_2   \ar@{ >->}[dl]  
        & \ar[l]_(.6){\iso} \bigl( Z{\times}X_2 \cup_{X_1{\times}X_2}X_1{\times}Y_1\bigr) \cup_{X_1{\times}Y_1} X_1{\times}Y_2 
                & \ar[l] X_1{\times}Y_2 
\\
Z{\times}Y_2 & Z{\times}X_2 \cup_{X_1{\times}X_2} X_1{\times}Y_1     \ar@{ >->}[u]  \ar@{ >->}[d]
              &  \ar[l]_{=} Z{\times}X_2 \cup_{X_1{\times}X_2} X_1{\times}Y_1    \ar@{ >->}[u] \ar@{ >->}[d]
                     & \ar[l]  X_1{\times}Y_1 \ar@{ >->}[u] \ar@{ >->}[d] 
\\
& Z{\times}Y_1  \ar@{ >->}[ul]      &  \ar[l]_{=}  Z{\times}Y_1   &  \ar[l]_{=} Z{\times}Y_1 
}
\endxy
\end{equation*}
Pass to pushouts in the columns, apply the universal mapping properties of the pushouts, and use isomorphism \eqref{pushoutpushout}
to simplify the pushout of the middle column to obtain the following commuting diagram. 
\begin{equation*}
  \xy \UCMT \xymatrix{
   Z{\times}Y_1 \cup_{X_1{\times}Y_1} X_1{\times}Y_2  \ar[d]_{\iso}
      &   \ar[l]_{\iso}   Z{\times}Y_1 \cup_{X_1{\times}Y_1} X_1{\times}Y_2  \ar@{ >->}[d]  
\\
Z{\times}Y_1  \cup_{(Z{\times}X_2 \cup_{X_1{\times}X_2} X_1{\times}Y_1)} (Z{\times}X_2\cup_{X_1{\times}X_2} X_1{\times}Y_2) \ar[r]
  &  Z \times Y_2  
}
\endxy
\end{equation*}
The space $Z{\times}Y_1 \cup_{X_1{\times}Y_1} X_1{\times}Y_2$ is a subspace of $Z \times Y_2$, so the downward arrow on the
right is a cofibration.  Since isomorphisms are cofibrations, it follows that the lower arrow is also a cofibration.  
Thus, we have verified the cofibration condition of lemma \ref{iteratedcolimlemma} for diagram \eqref{biexactnessmaindiagram}.

We may now calculate the colimit of diagram  \eqref{biexactnessmaindiagram} by two different iterative procedures.
Computing the pushouts of the rows first and applying  lemma \ref{iteratedcolimlemma}  gives a diagram
\begin{equation} \label{rowsfirst}
  \xy \UCMT \xymatrix{ Z \esm Y_2 &  \ar@{ >->}[l] Z \esm Y_1  \ar[r] &  Z \esm Y_0} \endxy 
\end{equation}
and calculating the pushouts of the columns first and applying lemma
\ref{iteratedcolimlemma} again gives a another diagram
\begin{equation}  \label{colsfirst}
  \xy \UCMT \xymatrix{ X_1{\times}X_2 
           &   \ar[l] Z{\times}X_2 \cup_{X_1{\times}X_2} X_1{\times} \bigl( Y_0 \cup_{Y_1} Y_2 \bigr)  \ar@{ >->}[r] 
                        &  Z \times \bigl( Y_0 \cup_{Y_1} Y_2 \bigr).} 
\endxy
\end{equation}
To see this formula for the middle object in diagram \eqref{colsfirst}, make the following considerations. 
We have the  diagram 
\begin{equation} \label{itcolimaux}
 \xy \UCMT \xymatrix{
Z_2{\times}X_2 & \ar[l] X_1{\times}X_2 \ar@{ >->}[r] &  X_1{\times}Y_2
\\
Z_2{\times}X_2 \ar@{ >->}[u]^{=} \ar[d]_{=} 
            & \ar[l]  X_1{\times}X_2 \ar@{ >->}[r] \ar@{ >->}[u]^{=} \ar[d]_{=}  
                       & X_1{\times}Y_1 \ar@{ >->}[u] \ar[d]
\\
Z_2{\times}X_2 & \ar[l]  X_1{\times}X_2 \ar@{ >->}[r] & X_1{\times}Y_0
}
\endxy
\end{equation} 
meeting the conditions of lemma \ref{iteratedcolimlemma}, whose colimit we also compute iteratively. 
Computing the pushouts of the rows first gives precisely the middle column in diagram \eqref{biexactnessmaindiagram},
whose pushout we are now evaluating. 
On the other hand, computing the pushouts along the columns first gives a diagram
\begin{equation*}
  \xy \UCMT \xymatrix{
Z_2{\times}X_2 &  \ar[l]  X_1{\times}X_2 \ar@{ >->}[r] & X_1{\times}(Y_2 \cup_{Y_1} Y_0) 
}
\endxy
\end{equation*}
whose pushout is the middle term displayed in diagram \eqref{colsfirst}.  As the iterated pushouts of diagram \eqref{itcolimaux}
are isomorphic to the colimit of the entire diagram, the iterated pushouts are isomorphic.  This justifies diagram \eqref{colsfirst}.  
 
Completing the analysis of diagram \eqref{biexactnessmaindiagram}, the pushouts of these diagrams 
\eqref{rowsfirst} and \eqref{colsfirst} are isomorphic, because they both represent the colimit of the original diagram 
\eqref{biexactnessmaindiagram}.  Interpreting this statement, we have the result that $Z \esm -$ preserves
pushouts of cofibrations.

  Suppose $f_1 \colon Y_1 \ra Y_1'$ and $f_2 \colon Y_2 \ra Y_2'$ are weak equivalences in $\cR_f(X_1)$ and $\cR_f(X_2)$, respectively.  
That is, the geometric realizations $|f_1|$ and $|f_2|$ are homotopy equivalences. Then $|f_1|{\times}\id_{|X_2|}$ and $\id_{|X_1|}{\times}|f_2|$
are homotopy equivalences.  By the ordinary gluing lemma for homotopy equivalences applied to the diagram
\begin{equation*}
  \xy \UCMT \xymatrix{
|Y_1| \times |X_2| \ar[d]_{|f_1|\times \id_{|X_2|}} 
             &   \ar@{ >->}[l]  {|X_1| \times |X_2|}  \ar@{ >->}[r] \ar[d]_{\id_{|X_1|} \times \id_{|X_2|}}  
                                              &   |X_1| \times |Y_2| \ar[d]_{\id_{|X_1|} \times |f_2|}
\\
|Y'_1| \times |X_2|  &   \ar@{ >->}[l] {|X_1| \times |X_2|}  \ar@{ >->}[r] &   |X_1| \times |Y'_2|
}
\endxy
\end{equation*}
the central arrow in the diagram below is also a homotopy equivalence.
\begin{equation*}
  \xy  \UCMT \xymatrix{
{|X_1| \times |X_2|}  \ar[d]_{\id_{|X_1|} \times \id_{|X_2|}} 
                 &   \ar[l]  {|Y_1|{\times}|X_2| \cup_{|X_1|{\times}|X_2|} |X_1|{\times}|Y_2|}  \ar@{ >->}[r] \ar[d]_{\heq}
                            &  {|Y_1| \times |Y_2|}  \ar[d]_{|f_1|\times|f_2|}
\\
{|X_1| \times |X_2|}
                 &   \ar[l]   {|Y'_1|{\times}|X_2| \cup_{|X_1|{\times}|X_2|} |X_1|{\times}|Y'_2|}  \ar@{ >->}[r]
                            &   {|Y'_1| \times |Y'_2|}
}
\endxy
\end{equation*}
Since the pushout of the last diagram is homeomorphic to $|Y_1 {\esm} Y_2| \ra \abs{Y_1' {\esm} Y_2'}$    (``colimits commute''),
$Y_1 \esm Y_2 \ra  Y_1' \esm Y_2'$ is a weak equivalence.
\end{proof}
\begin{remark}
  The external smash product will also preserve many colimits.  However, our applications principally involve the
special colimits that are pushouts of cofibration squares. 
\end{remark}
Here is the postponed proof of lemma \ref{iteratedcolimlemma}.
\begin{proof}[Proof of lemma \ref{iteratedcolimlemma}]
We make frequent use of the isomorphism
\begin{equation}  \label{pushoutpushout}
  ( A \cup_B C ) \cup_C D \iso A \cup_B D.
\end{equation}
  The canonical arrow 
$\xy \UCMT \xymatrix{ B_2 \cup_{B_1} B_0  \ar[r] & C_2\cup_{C_1} C_0 } \endxy $
factors into the composition of canonical arrows induced by passing to pushouts of the columns in the map of diagrams
\begin{equation*}
  \xy \UCMT \xymatrix{
  B_2  \ar[r]^{=} &    B_2  \ar@{ >->}[r] & C_2
\\
  B_1   \ar@{ >->}[u]  \ar[r]^=  \ar[d] 
            &   B_1  \ar@{ >->}[r] \ar@{ >->}[u] \ar[d]  
                     &   C_1 \ar@{ >->}[u] \ar[d] 
\\
  B_0    \ar@{ >->}[r]   &    C_0   \ar[r]^= &  C_0 
}
\endxy
\end{equation*}
and we show each arrow in the factorization is a cofibration, as follows. The first arrow in the factorization 
appears as the lower row in the completed pushout diagram
\begin{equation*}
  \xy \UCMT \xymatrix{
B_0 \ar@{ >->}[r] \ar[d] & C_0 \ar[d]
\\
B_2 \cup_{B_1} B_0 \ar@{ >->}[r] & (B_2 \cup_{B_1} B_0)\cup_{B_0} C_0 \ar[r]^(0.6){\iso} & B_2 \cup_{B_1} C_0}
\endxy
\end{equation*}
augmented by an isomorphism, so the first arrow is a cofibration, as claimed.  
From the hypothesis on the canonical map from $B_2\cup_{B_1} C_1$ to $C_2$, the upper arrow in the next diagram 
is a cofibration, so the lower arrow in the completed pushout diagram is as well.  
\begin{equation*}
  \xy \UCMT  \xymatrix{
 & B_2 \cup_{B_1} C_1 \ar@{ >->}[r] \ar[d]  &   C_2   \ar[d]   &
\\
B_2 \cup_{B_1} C_0 & \ar[l]_(0.6){\iso} \bigl( B_2 \cup_{B_1} C_1 \bigr)  \cup_{C_1} C_0   \ar@{ >->}[r] 
              &  C_2 \cup_{B_2 \cup_{B_1} C_1} \bigl(( B_2 \cup_{B_1} C_1) \cup_{C_1} C_0\bigr)  \ar[r]^(0.7){\iso}  & C_2 \cup_{C_1} C_0
}
\endxy
\end{equation*}
Augmenting the completed pushout diagram by the two isomorphisms, the second arrow 
$\xy \UCMT  \xymatrix{ B_2 \cup_{B_1} C_0 \ar[r] & C_2 \cup_{C_1} C_0 } \endxy $
in the factorization is also a cofibration.  
Then the composition 
$ \xy \UCMT \xymatrix{ B_2 \cup_{B_1} B_0 \ar@{ >->}[r] & B_2 \cup_{B_0} C_0 \ar@{ >->}[r] &   C_2 \cup_{C_1} C_0 } \endxy$
is a cofibration and this arrow is isomorphic to the arrow in diagram \eqref{iteratedcolim1}. 

To obtain the result for the row-wise pushouts from the result for column-wise pushouts, observe that 
the properties of the arrows in the diagram are symmetric with respect to reflection in the diagonal $A_0B_1C_2$.
Therefore, it suffices to reflect the diagram in this diagonal and apply the column-wise result.
\end{proof}
\begin{proof}[Proof of the second part of theorem \ref{symmetricbimonoid}]
Since the functor $\mu_* \colon \cR_f(X{\times}X) \ra \cR_f(X)$ is exact
\cite[Lemma 2.1.6]{Waldhausen85}, and we have seen that $\esm$ is biexact in proposition \ref{externalbiexactness},
the composite $\sma = \mu_* \com \esm$ is biexact. 
\end{proof}
Now we take up coherence properties.
\begin{proof}[Proof of the first part of theorem \ref{symmetricbimonoid}]
It is well-known that disjoint union of sets and the one point union $\wed$ of pointed sets are categorical sum
operations, so that all coherence conditions for these operations are automatically met.  For the category 
of sets containing a fixed set $S$ the union $\wed_S$ of two sets along the common subset is also the categorical sum, so
$\wed_S$ fulfills all coherence conditions.  Concerning products, the cartesian product of sets and the smash product of
pointed sets are operations also meeting coherence conditions.  When these operations of sum and product are considered
together, they are related by distributivity isomorphisms, and the combined systems exhibit the coherence properties
discussed in \cite{LaPlazaCoh1}.  It is possible to develop the coherence properties we need for operations on retractive spaces
from these basic elements by developing the operations $\wdX$ and $\esm$ dimensionwise and pointwise over 
$X$, resp., $X_1{\times}X_2$ in the case of products, from one point union and smash product of pointed sets. 
Compare the remark following definition \ref{externalsmash}.  We take a different approach here.

For the sum $\wdX$, we need a slight extension of the union of sets along a common subset 
to cover the case of the disjoint union of two simplicial sets along a common simplicial subset. 
Let $\cT$ be the category of triples $(T, r \colon T \ra S, s \colon S \ra T)$, where $S$ and $T$ are sets
and the functions satisfy $r\com s = \id_S$.  Occasionally, it is convenient to view $S$ as a subset of $T$.  A morphism 
\begin{equation*}
 (f, f') \colon (T_1, r_1 \colon T_1 \ra S_1, s_1 \colon S_1 \ra T_1) \ra (T_0, r_0 \colon T_0 \ra S_0, s_0 \colon S_0 \ra T_0)
\end{equation*}
is a pair of maps $f \colon T_1 \ra T_0$ and $f' \colon S_1 \ra S_0$ such that $s_0 f' =  f s_1$ and $ r_0f = f'r_1$.
A object $(Y,r,s)$ of $\cR(X)$ can be viewed as a functor $\Delta^{\rm op} \ra \cT$, and conversely.
There is a functor $u \colon \cT \ra \text{{\bf Set}}$ that selects the subset $S$ and morphisms $f' \colon S_1 \ra S_0$.
On the pullback category
\begin{equation*}
  \xy \UCMT \xymatrix{
\cT \times_{\text{{\bf Set}}} \cT \ar[r] \ar[d]   &     \cT \times \cT \ar[d]_{u \times u}
\\
\text{{\bf Set}} \ar[r]^(0.45){\Delta}    &  \text{{\bf Set}}  \times \text{{\bf Set}}
} 
\endxy 
\end{equation*}
define the operation 
$(T_1, r_1 \colon T_1 \ra S, s_1 \colon S \ra T_1) \wed_{S} (T_2, r_2 \colon T_2 \ra S, s_2 \colon S \ra T_2)$, 
abbreviated $(T_1,r_1, s_1) \wed_S (T_2, r_2, s_2)$, or even $T_1 \wed_S T_2$.
Set
\begin{equation*}
  T_1 \wed_S T_2 = T_1 \amalg T_2/\sim  ,
\end{equation*}
where $\sim$ is the equivalence relation generated by setting $s_1(x)\sim s_2(x)$ for $x \in S$.
Set $i_j \colon T_j \ra T_1 \wed_S T_2$ to be the inclusion $T_j \ra T_1 \amalg T_2$ followed by the quotient map to $T_1 \wed_s T_2$. 
For the rest of the structure, set
\begin{align*}
 r \colon &T_1 \wed_S T_2 \ra S,  
\intertext{to be the unique function satisfying $ri_j = r_j$,  for $j = 1, \, 2$ and let}
  s \colon &S \ra T_1 \wed_S T_2 
\end{align*}
satisfy $s(x)= i_1s_1(x) = i_2 s_2 (x)$ for $x \in S$. 
Define $(i_1, i_1'=\id) \colon (T_1, s_1, r_1) \ra (T_1 \wed_S T_2, r, s)$ to obtain a morphism in $\cT$.
The identities $r i_1 = i_1' r_1$ and $s i_1' = i_1 s_1$ are satisfied by definition and by the condition $r_1 s_1 = \id$. 
Define $(i_2, i_2') \colon (T_2, s_2, r_2) \ra (T_1 \wed_S T_2, r,s )$  similarly.  If $(T', r', s' )$
is another object of $\cT$, and suppose  $(f_i,f_i') \colon (T_i, r_i, s_i) \ra (T', r', s') $ is a morphism in 
$\cT \times_{\text{\bf Set}} \cT$ for $i = 1, 2$.  This just means that $f_1' = f_2'  \colon S \ra S'$. 
Then the categorical sum properties of the disjoint union on the category $\text{\bf Set}$ and the quotient construction
deliver a unique morphism
 \begin{equation*}
   (h, h') \colon (T_1\wed_S T_2, r', s') \ra (T',r',s')
 \end{equation*}
such that $(h,h') \com (i_1, i_1') = (f_1, f_1')$ and $(h,h') \com (i_2, i_2') = (f_2, f_2')$.
When the base set is fixed, we obtain a categorical sum;
in general, when the base set varies, we obtain a (partially defined) categorical sum on $\cT$. 

We have observed that an object of the category $\cR(X)$ is a simplicial object in the category $\cT$, that is,  
a functor $\Delta^{\rm op} \ra \cT$.  A pair of objects 
$(Y_1, r_1, s_1)$ and $(Y_2, r_2, s_2)$ in $\cR(X)$ defines a functor
$\Delta^{\rm op}  \ra \cT\times_{\text{\bf Set}} \cT$.  
We obtain the operation 
$(Y_1, r_1, s_1) \wdX (Y_2, r_2, s_2)$ 
 based on the dimension-wise operation $(Y_1)_p \wed_{X_p} (Y_2)_p$. 
This makes $\wdX $  a categorical sum in $\cR(X)$, with unit (zero element, thinking additively) the space $X$.
The commutativity isomorphisms $\gamma'$, associativity isomorphims $\alpha'$, 
left and right unit isomorphisms $\lambda'$ and $\rho'$ are straightforward
consequences of the analogous properties of the disjoint union operation on sets.
Essentially, all the basic properties required for coherence of the sum operation $\wdX$ are automatically fulfilled.  
That $\wdX$ is 
the categorical sum simplifies almost all coherence considerations involving diagrams involving both 
$\wdX$ and $\sma$. 

To complete the input for LaPlaza's coherence result we need to identify  in $\cR(X)$ an additive identity,  a multiplicative zero
element, a multiplicative identity, commutativity and associativity isomorphisms for 
$\sma$, and, finally, distributivity isomorphisms.  

Clearly $(X, \id, \id)$ is the identity for $\wdX$. 
Example \ref{zeroobject} implies that $(X, \id, \id)$ is a zero object from the left and the right for $\sma$,
in the sense that there are natural isomorphims 
$\lambda_Y^* \colon X \sma Y \ra X$ and  $\rho_Y^* \colon Y \sma X \ra X$.
Example \ref{esmaction} combined with proposition \ref{internalizing} delivers the fact that
$i_{e*}(S^0) = X{\wed}S^0$, where the base point of $S^0$ is identified with the multiplicative identity of $X$ and the retraction
collapses $S^0$ to the identity of $X$, is a multiplicative identity in the sense that there are natural isomorphisms
$\lambda_Y \colon (X{\wed}S^0)\sma Y \ra Y$ and $\rho_Y \colon Y \sma (X{\wed}S^0) \ra Y$.

For commutativity of the product $\sma = \mu_*\com \esm$, we have the following considerations.
Use commutativity for cartesian products and apply the definitions  from \eqref{intsmashprod} of the internal smash product
to obtain the following diagram. 
\begin{equation}  \label{multcommutativity}
  \xy \UCMT \xymatrix{
X \ar@{=}[d] &  \ar[l]_{\mu} X \times X \ar[d]_{\gamma} & \ar[l]  Y_1{\times}X \cup_{X{\times}X} X{\times}Y_2   \ar@{ >->}[r] \ar[d]_{\gamma} & Y_1 \times Y_2 \ar[d]_{\gamma}
\\
X  &  \ar[l]_{\mu} X \times X &  \ar[l]  Y_2{\times}X \cup_{X{\times}X} X{\times}Y_1   \ar@{ >->}[r]                    & Y_2 \times Y_1}
\endxy
\end{equation}
In the diagram the arrows labeled $\gamma$ are the isomorphisms switching the factors in the cartesian products. 
Note that 
$ r_1\sma r_2 = \mu_*(r_1\esm r_2) = r_1\cdot r_2 = r_2\cdot r_1 = \mu_*(r_2 \esm r_1) = r_2 \sma r_1$,
since $X$ is abelian.  
Passage to pushouts yields an isomorphism
\begin{equation*}
  \gamma_{Y_1,Y_2} \colon (Y_1{\sma}Y_2, r_1{\sma}r_2, s_1{\sma}s_2) \stackrel{\iso}{\lra}  (Y_2{\sma}Y_1, r_2{\sma}r_1, s_2{\sma}s_1)
\end{equation*}
It is easily seen that $\gamma_{Y_2,Y_1}\gamma_{Y_1,Y_2} = \id$ holds, often written ``$\gamma^2= \id"$ and called the inverse law, 
and that the left and right unit laws are compatible. These facts are recorded in the following commuting diagrams. 
\begin{equation*}
  \xy \UCMT   \xymatrix@R-1.5em{
   &   Y_2 \sma Y_1  \ar[dr]^{\gamma_{Y_2, Y_1}} 
            & &  Y \sma (X{\wed}S^0) \ar[rr]^{\gamma_{Y,X\wed S^0}} \ar[dr]_{\rho_Y} 
                                 & & (X {\wed} S^0) \sma Y \ar[dl]^{\lambda_Y} 
\\
Y_1 \sma Y_2  \ar[ur]^{\gamma_{Y_1, Y_2}}  \ar@{=}[rr] &  &  Y_1 \sma Y_2  & &  Y
}
\endxy
\end{equation*}

Consider now associativity, for which we use the following diagram.
\begin{equation}  \label{assoc1}
  \xy \UCMT  \xymatrix@R=3ex@C+3ex{
    &  \bigl((Y_1 {\times} Y_2) {\times} X \bigr) {\cup} \bigl((Y_1 {\times} X){\times} Y_3\bigr) {\cup} \bigl((X {\times} Y_2){\times} Y_3  \bigr)
             \ar[dl]|{\mu( \mu(r_1, r_2), \id) \cup \mu(\mu( r_1, \id ), r_3) \cup \mu(\mu(\id, r_2), r_3)} \ar@{ >->}[r]
                    &   (Y_1 \times Y_2) \times Y_3
\\
X    &   &
\\
       &   \bigl( (Y_1 {\times} X)  {\times} X\bigr) \cup \bigl((X {\times} Y_2) {\times} X\bigr)    \ar@{ >->}[uu] 
               \ar[dd]|{(\mu(r_1, \id), \id)\cup (\mu(\id , r_2), \id)}
                 \ar[dl]|{\mu( \mu(r_1, \id) , \id) \cup \mu( \mu( \id, r_2), \id)}   \ar@{ >->}[r]^{=}
                    &  \bigl( (Y_1 {\times} X)  {\times} X\bigr) \cup\bigl( (X {\times} Y_2) {\times} X  \bigr)  \ar@{ >->}[uu]  
                \ar[dd]|{(\mu(r_1, \id), \id)\cup (\mu(\id , r_2), \id)}
\\
X    \ar@{ >->}[uu]_{=} \ar[dd]^{=} &   &
\\
       &   X \times X \ar[dl]_{\mu}  \ar@{ >->}[r]^{=}
                    &   X \times X 
\\
X  &   &   }
\endxy
\end{equation}
The point is that the associativity for $\sma$ rests on associativity for $\times$, $\cup$, and associativity of the 
multiplication $\mu$ on $X$.
By passage to colimits we obtain associativity for $\sma$.
For the usual smash product, associativity for cartesian products passes to associativity for smash products;
our argument is similarly structured. 
%
 The first step is to obtain an expression for $(Y_1 \sma Y_2)\sma Y_3$ that involves only cartesian products and colimits.
Diagram \eqref{assoc1} fulfills the hypotheses of lemma \ref{iteratedcolimlemma}, so we may calculate the colimit
iteratively in two ways.   Taking the colimit along the columns produces the diagram
\begin{equation*}
  \xy \UCMT \xymatrix@C+=14ex{
X  
     &  \ar[l]_(0.68){\mu(r_{12}, \id) \cup \mu(\id, r_3)}   \bigl( (Y_1{\sma} Y_2) \times X\bigr) \cup \bigl( X \times Y_3 \bigr) \ar@{ >->}[r] 
              & (Y_1{\sma} Y_2) \times Y_3
}
\endxy
\end{equation*}
whose colimit is by definition $(Y_1 \sma Y_2)\sma Y_3$.
On the other hand, computing the colimit along the rows produces this diagram, a copy of the top row in diagram \eqref{assoc1}.
\begin{equation*}
  \xy \UCMT \xymatrix{
    &  \bigl((Y_1 {\times} Y_2) {\times} X \bigr) {\cup} \bigl((Y_1 {\times} X){\times} Y_3\bigr) {\cup} \bigl((X {\times} Y_2){\times} Y_3  \bigr)
             \ar[dl]|{\mu( \mu(r_1, r_2), \id) \cup \mu(\mu( r_1, \id ), r_3) \cup \mu(\mu(\id, r_2), r_3)} \ar@{ >->}[r]
                    &   (Y_1 \times Y_2) \times Y_3
\\
X    &   &
}
\endxy
\end{equation*}
Therefore, the colimit, or pushout, of this diagram is another representation of $(Y_1\sma Y_2) \sma Y_3$, and we record the completed diagram
\begin{equation} \label{assoc1a}
   \xy \UCMT \xymatrix{
      \bigl((Y_1 {\times} Y_2) {\times} X \bigr) {\cup} \bigl((Y_1 {\times} X){\times} Y_3\bigr) {\cup} \bigl((X {\times} Y_2){\times} Y_3  \bigr)
             \ar[d]|{\mu( \mu(r_1, r_2), \id) \cup \mu(\mu( r_1, \id ), r_3) \cup \mu(\mu(\id, r_2), r_3)} \ar@{ >->}[r]
                    &   (Y_1 \times Y_2) \times Y_3  \ar[d]
\\
X   \ar@{ >->}[r] &   (Y_1\sma Y_2)\sma Y_3
}
\endxy
\end{equation}
as a preferred alternative representation of $(Y_1 \sma Y_2) \sma Y_3)$. 
Starting from a diagram similar to \eqref{assoc1}, but with parentheses shifted to the right, there is another completed pushout diagram
\begin{equation} \label{assoc2a}
   \xy \UCMT \xymatrix{
      \bigl(Y_1 {\times} (Y_2 {\times} X) \bigr) {\cup} \bigl(Y_1 {\times} (X{\times} Y_3)\bigr) {\cup} \bigl(X {\times} (Y_2{\times} Y_3) \bigr)
             \ar[d]|{\mu(r_1, \mu(r_2, \id)) \cup \mu( r_1, \mu(\id , r_3)) \cup \mu(\id,\mu( r_2, r_3))} \ar@{ >->}[r]
                    &   Y_1 \times (Y_2 \times Y_3)  \ar[d]
\\
X   \ar@{ >->}[r] &   Y_1\sma (Y_2\sma Y_3)
}
\endxy
\end{equation}
representing $Y_1 \sma (Y_2 \sma Y_3)$. 
Consequently, the associativity isomorphisms
\begin{equation*}
  \alpha_{Y_1, Y_2, Y_3} \colon  Y_1 {\times} (Y_2 {\times} Y_3) \ra  (Y_1 {\times} Y_2) {\times} Y_3, \; 
  \alpha_{Y_1, Y_2, X} \colon Y_1 {\times} ( Y_2 {\times} X) \ra (Y_1 {\times}Y_2 ) \times X,
\end{equation*}
 and so on, induce an isomorphism of diagram \eqref{assoc2a} with diagram \eqref{assoc1a} and an associativity isomorphism 
 \begin{equation}
   \label{assocsma}
   \alpha_{Y_1, Y_2, Y_3} \colon  Y_1 {\sma} (Y_2 {\sma} Y_3) \ra  (Y_1 {\sma} Y_2) {\sma} Y_3.
 \end{equation}

In Laplaza's framework \cite{LaPlazaCoh1}  left distributivity of the product over the sum operation is
encoded by  a monomorphism
$\delta_{Y_0, Y_1 , Y_2} \colon Y_0 \sma (Y_1 \wdX Y_2) \ra (Y_0{\sma}Y_1) \wdX (Y_0{\sma}Y_2)$.
That $\wdX$ is a categorical sum enables us to construct an isomorphism 
$\delta^{-1}_{Y_0, Y_1, Y_2}  \colon (Y_0{\sma}Y_1) \wdX (Y_0{\sma}Y_2) \ra Y_0 \sma (Y_1 \wdX Y_2)$ 
quite easily, as follows.
Applying the functor $Y_0{\sma}-$ to the sum diagram $Y_1 \ra Y_1 \wdX Y_2 \la Y_2$ provides
a diagram $Y_0{\sma}Y_1 \ra Y_0{\sma}(Y_1 \wdX Y_2) \la Y_0{\sma}Y_2$.  Since $\wdX$ is a categorical
sum, there results a map $(Y_0{\sma}Y_1) \wdX (Y_0{\sma}Y_2) \ra Y_0 \sma (Y_1 \wdX Y_2)$.
To check that this map is an isomorphism observe that in a simplicial dimension $p$ the 
$p$-simplices outside of $X$ in the domain are $(Y_0{-}X)_p{\times}(Y_1{-}X)_p \coprod (Y_0{-}X)_p{\times}(Y_2{-}X)_p$,
 the $p$-simplices outside of $X$ in the target are $(Y_0{-}X)_p{\times} \bigl( (Y_1{-}X)_p \coprod (Y_2{-}X)_p\bigr)$,
and the induced map is a one-to-one correspondence.  Thus, we obtain the isomorphism
$ \delta^{-1}_{Y_0, Y_1, Y_2} \colon (Y_0{\sma}Y_1) \wdX (Y_0{\sma}Y_2) \ra Y_0 \sma (Y_1 \wdX Y_2),$ whose inverse
\begin{equation}
  \label{leftdist}
 \delta_{Y_0, Y_1 , Y_2} \colon Y_0 \sma (Y_1 \wdX Y_2) \iso (Y_0{\sma}Y_1) \wdX (Y_0{\sma}Y_2)
\end{equation}
can be shown to meet LaPlaza's conditions. 
Similarly, we obtain an isomorphism 
\begin{equation}
  \label{rightdist}
  \delta^{\#}_{Y_0, Y_1 , Y_2} \colon (Y_0 \wdX Y_1) \sma Y_2 \iso (Y_0{\sma}Y_2) \wdX (Y_1{\sma}Y_2)
\end{equation}
This concludes the catalog of basic inputs for LaPlaza's theorem.

Given the basic inputs, the next step is to establish the commutativity of certain diagrams, twenty-four in number. 
Because $\wdX$ is a categorical sum and $\sma$ is biexact, preserving sums, checking the commutativity of seventeen of the diagrams is routine.
The other seven diagrams involve the multiplicative or additive neutral objects or the multiplicative zero object and are straightforward to verify.
LaPlaza's main theorem applies and ``all diagrams that should commute do, in fact, commute.''  These remarks complete the proof of part one of 
theorem \ref{symmetricbimonoid}.
\end{proof}
\section{Defining the operations} \label{Extension}
The ingredients for the operations take values in categories of retractive spaces on which groups
are acting.  We first establish language and notation following 
\cite[Definitions 5.1, 5.2, 5.3, and 5.4]{GSAthyops}
for the following definitions.
\begin{definition}
  A set $\cF$ of subgroups  of $\Sigma_n$ is called a family of subgroups if 
it contains at most one member from each conjugacy class of subgroups.
\end{definition}
\begin{definition}
  For a finite group $G$, a $G$-simplicial set $Y$ has orbit types in a family $\cF$ relative to 
another $G$-simplicial set $W$ if $Y$ may be obtained from $W$ by
direct limit and by formation of pushouts of diagrams of the form
\begin{equation} \label{basicorbittypeattachment}
  \xy \UCMT \xymatrix{
Y' &  \ar[l]  \partial \Delta^n{\times} (G/H)  \ar@{ >->}[r]  & \Delta^n{\times} \bigl( G/H \bigr), }
\endxy
\end{equation}
where $\Delta^n$ is the standard simplicial $n$-simplex, $\partial \Delta^n$ is the simplicial boundary, and
$H \in \cF$. 
\end{definition}
\begin{definition} \label{equivariant}
  For a $\Sigma_n$-simplicial set $W$ let $\cR(W, \Sigma_n, \cF)$ 
denote the category whose objects are the triples $(Y,r,s)$ 
where $Y$ is a $\Sigma_n$-simplicial set with orbit types in $\cF$
relative to a $\Sigma_n$-section $s \colon W \ra Y$.  The map
$r \colon Y \ra W$ is $\Sigma_n$-retraction of $Y$ to $W$,
that is, $r \com s = \id_W$. 
Morphisms are $\Sigma_n$-equivariant maps commuting with the retractions and sections. 
\end{definition}
\begin{definition} \label{equivwcofandwe}
  Let $\cR_f(W, \Sigma_n, \cF)$ denote the full subcategory of $\cR(W, \Sigma_n, \cF)$
whose objects are the triples
$(Y,r,s)$ such that $Y$ is built from $W$ by formation of finitely many pushouts
of the form of diagram \ref{basicorbittypeattachment}.
The category $\cR_f(W, \Sigma_n, \cF)$ is also equipped  with 
cofibrations and weak equivalences   A cofibration 
$\xy \UCMT \xymatrix@1{( W_1, r_1, s_1) \ar@{ >->}[r] & (W_2, r_2, s_2)} \endxy $
is an injective $\Sigma_n$-map and a weak equivalence  
$(Y_1, r_1, s_1) \ra (Y_2, r_2, s_2)$ is a morphism for which the geometric
realization of the underlying map $Y_1 \ra Y_2$ is a $\Sigma_n$-equivariant
homotopy equivalence.
\end{definition}
For $X$ a connected simplicial abelian group on which $\Sigma_n$ acts trivially, 
we will need the categories
$\cR_f(X, \Sigma_n, \{{\rm all}\})$
of retractive left $\Sigma_n$-spaces $\tilde{Y}$ over $X$ which are finite
relative to $X$.
In principle, we may also allow $X$ to be a connected commutative simplicial
monoid with unit element.
We will write $\Loops \abs{w S_{\bullet}\cR_f(X, \Sigma_n, \{{\rm all}\})} = A_{\Sigma_n, \{{\rm all}\}}(X)$.
The category of retractive left $\Sigma_n$-spaces on which $\Sigma_n$ acts with trivial isotropy
 outside of $X$ is then $\cR_f(X, \Sigma_n, \{e\})$.  In other words, the $\Sigma_n$-action on 
simplices outside of $X$ is free on those simplices. 
 Later, we will abbreviate   $\cR_f(X, \Sigma_n, \{e\})  = \cR_f(X, \Sigma_n)$.
In lemma \ref{interpretations} we will justify the notation
$\Loops \abs{w S_{\bullet}\cR_f(X, \Sigma_n, \{e\})} = A(X{\times}B\Sigma_n)$.

There are two constructions underlying our approach to the Segal operations.   First is a family of  bi-exact functors 
\begin{equation*}
  \btn_{k,\ell} \colon \cR_f(X, \Sigma_k, \{\text{all}\}) \times \cR_f(X, \Sigma_{\ell}, \{\text{all}\})
\lra
   \cR_f(X, \Sigma_{k+\ell}, \{\text{all}\})
\end{equation*}
defined for $k$, $\ell \geq 0$, called  box-tensor operations (Definition \ref{boxtensor}). 
Second is a family of functors
\begin{equation*}
  \dop_{n,k} \colon \cR_f(X, \Sigma_n,  \{{\rm all}\} )^{[k]} \lra \cR_f(X, \Sigma_{kn}, \{\text{all}\}),
\end{equation*}
called  diamond operations (Definition \ref{basicdiamond}). 
Here $\cR_f(X, \Sigma_n, \{{\rm all}\} )^{[k]}$ is the category of filtered objects
\begin{equation*}
  \xy \UCMT \xymatrix{
Y_1 \ar@{ >->}[r] & Y_2 \ar@{ >->}[r] & \cdots \ar@{ >->}[r] & Y_k 
}
\endxy
\end{equation*}
with $Y_i$ in $\cR_f(X, \Sigma_n, \{{\rm all}\} )$ and natural transformations of such sequences.  

First we set up the box-tensor operation.  For a connected simplicial abelian group $X$, 
let $n=k{+}\ell$ and define an induction functor
\begin{equation} \label{smallinduction1}
  \Ind_{\Sigma_k{\times}\Sigma_{\ell}}^{\Sigma_n} \colon 
            \cR_f(X, \Sigma_k{\times}\Sigma_{\ell}, \{{\rm all}\}) \ra \cR_f(X, \Sigma_n , \{{\rm all}\}).
\end{equation}
Let $\fset{n}$ be a finite set of cardinality $n$ (for example, the standard example), 
let $\fset{k}{\cup}\fset{l}$ be the disjoint union of finite sets of cardinality 
$k$ and $l$, respectively, and let $\Iso(\fset{n},\fset{k}{\cup}\fset{l})$ be the set
of isomorphisms from $\fset{n}$ to the disjoint union. 
Let  $\Iso(\fset{n},\fset{k}{\cup}\fset{l})_+ = \Iso(\fset{n},\fset{k}{\cup}\fset{l}){\cup}\{*\}$
be viewed as an object of $\cR_f(*)$, with the obvious section and with the retraction the constant
map to $\{*\}$. The group $\Sigma_n$ acts from the left on $\Iso(\fset{n},\fset{k}{\cup}\fset{l})_+$ by fixing the basepoint
and by the rule $\sigma \cdot f = f \com \sigma^{-1}$ for  $\sigma \in \Sigma_n$ and $f \colon \fset{n} \ra \fset{k}{\cup}\fset{l}$. 
Normally $\Sigma_k{\times}\Sigma_{\ell}$ also acts from the left by post-composition, but we find it  convenient to 
use the right action defined by $f\cdot(\sigma_1, \sigma_2) = (\sigma_1^{-1}, \sigma_2^{-1})\com f$.
For $(Y, r, s) \in \cR_f(X, \Sigma_k{\times}\Sigma_{\ell}, \{ {\rm all} \})$ we unwind the defining pushout square 
\begin{equation}  \label{presmallinduction}
  \xy \UCMT \xymatrix{
(\Iso(\fset{n}, \fset{k}{\cup}\fset{l})_+{\times}X)\cup_{{*}\times X} ({*}\times Y) \ar@{ >->}[r] \ar[d]
         &     \Iso(\fset{n},\fset{k}{\cup}\fset{l})_+ \times Y \ar[d]
\\
 {*} \times X     \ar@{ >->}[r]   &   \Iso(\fset{n}, \fset{k}{\cup}\fset{l})_+ \esm Y}
\endxy
\end{equation}
to find  that the exterior smash product 
$\Iso(\fset{n}, \fset{k}{\cup}\fset{l})_+ \esm Y$
amounts to 
$\Iso(\fset{n}, \fset{k}{\cup}\fset{l})$-copies of $Y$, pasted together along the common subspace $X$.  
The retraction 
$r' \colon \Iso(\fset{n}, \fset{k}{\cup}\fset{l})_+ \esm Y \ra X$
given by $r'([f , y]) = r(y)$ is $\Sigma_n$-equivariant when $\Sigma_n$ acts trivially on $X$.   
We may also apply the principle of proposition \ref{internalizing} to re-express the exterior smash product
as an internal smash product and write
\begin{equation*}
  \Iso(\fset{n}, \fset{k}{\cup}\fset{l})_+ \esm Y \iso (X \wed \Iso(\fset{n}, \fset{k}{\cup}\fset{l})_+) \sma Y.
\end{equation*}

Define $\Iso(\fset{n}, \fset{k}{\cup}\fset{l})_+ \esm^{\Sigma_k{\times}\Sigma_{\ell}} Y$ to be the quotient space
of $\Iso(\fset{n}, \fset{k}{\cup}\fset{l})_+ \esm Y$ obtained by imposing the equivalence relation generated by
$\bigl[f\cdot (\sigma_1, \sigma_2) , y \bigr] \sim \bigl[f, (\sigma_1, \sigma_2)\cdot y \bigr]$. 
The left action of $\Sigma_n$ passes to the quotient, and, since the action of $\Sigma_n$ on $X$ is trivial, 
the retraction  $ \Iso(\fset{n}, \fset{k}{\cup}\fset{l})_+ \esm Y \ra X$
is given by $r'([f , y]) = r(y)$ also passes to the quotient, as does the section.  Thus, we obtain the necessary structure maps
\begin{equation}  \label{smallinduction2}
  \xy \UCMT \xymatrix{
X     \ar@{ >->}[r]   &   \Iso(\fset{n}, \fset{k}{\cup}\fset{l})_+ \esm^{\Sigma_k{\times}\Sigma_{\ell}} Y \ar[r]^(0.75)r & X}
\endxy
\end{equation}
This completes the definition of the induction functor
$  \Ind_{\Sigma_k{\times}\Sigma_{\ell}}^{\Sigma_n} \colon 
            \cR_f(X, \Sigma_k{\times}\Sigma_{\ell}, \{{\rm all}\}) \ra \cR_f(X, \Sigma_n , \{{\rm all}\})$.

Next we need an elementary pairing functor. 
\begin{equation} \label{pairings1}
  \cR_f(X, \Sigma_k,  \{{\rm all}\}) \times \cR_f(X, \Sigma_{\ell}, \{{\rm all}\}) \ra \cR_f(X, \Sigma_k \times \Sigma_{\ell}, \{{\rm all}\})
\end{equation}
The pairing sends
$\bigl((Y_1, r_1, s_1) ,   (Y_2, r_2, s_2)\bigr)$ to $ (Y_1, r_1,s_1) \sma (Y_2, r_2, s_2)$.
\begin{definition}  \label{boxtensor}
  Define the box-tensor operations following the pairing functor \eqref{pairings1} 
with the induction functor \eqref{smallinduction2}.
  \begin{multline}
    \label{eq:boxten}
    \btn_{k,\ell} \colon  \cR_f(X, \Sigma_k, \{{\rm all}\}) \times \cR_f(X, \Sigma_{\ell}, \{{\rm all}\}) 
\\
           \stackrel{\sma}{\ra} \cR_f(X, \Sigma_k \times \Sigma_{\ell}, \{{\rm all}\})
                  \stackrel{\Ind_{\Sigma_k\times \Sigma_{\ell}}^{\Sigma_n}}{\lra} 
                           \cR_f(X, \Sigma_n, \{{\rm all}\})
  \end{multline}
\end{definition}
\begin{proposition} \label{boxtensorassociative}
  The box-tensor operations are associative up to natural isomorphism.
\end{proposition}
\begin{proof}
  The associativity of the box-tensor operations is a consequence of the symmetric monoidal
structure on $\cR_f(X)$ associated with $\sma$, along with properties of the cartesian product of groups and disjoint union of sets. 
Abbreviating $\id_{\cR_f(X, \Sigma_{k_1}, \{{\rm all}\})}$, $\id_{\cR_f(X, \Sigma_{k_3}, \{{\rm all}\})}$ by $\id_1$, $\id_3$, respectively, 
 the assertion in detail is that the diagram
\begin{equation*}
\xy  \UCMT \xymatrix@C-13em{
& \cR_f(X, \Sigma_{k_1}, \{{\rm all}\}) \times  \cR_f(X, \Sigma_{k_2}, \{{\rm all}\}) \times  \cR_f(X, \Sigma_{k_3}, \{{\rm all}\}) 
 \ar[dl]_(0.6){\btn_{k_1,k_2} \times \id_3}  \ar[dr]^(0.6){\id_1 \times \btn_{k_2,k_3}} &
\\
 \cR_f(X, \Sigma_{k_1+k_2}, \{{\rm all}\}) \times  \cR_f(X, \Sigma_{k_3}, \{{\rm all}\}) \ar[dr]_(0.4){\btn_{k_1+k_2, k_3}}
         &   &
           \cR_f(X, \Sigma_{k_1}, \{{\rm all}\}) \times  \cR_f(X, \Sigma_{k_2+ k_3}, \{{\rm all}\}) \ar[dl]^(0.4){\btn_{k_1, k_2+k_3}}
\\
  &   \cR_f(X, \Sigma_{k_1+k_2+k_3}, \{{\rm all}\})   &
}
\endxy
\end{equation*}
commutes up to canonical isomorphism.
Given a triple  $(Y_1, Y_2, Y_3)$ in the category at top of the diagram, the value of the lefthand sequence of arrows is
\begin{equation*}
  \Iso(\fset{k_1{+}k_2{+}k_3}, \fset{k_1{+}k_2}\cup\fset{k_3})_+ \esm^{\Sigma_{k_1+k_2}\times \Sigma_{k_3}}
  \biggl(  \Bigl( \Iso(\fset{k_1+k_2},\fset{k_1} \cup \fset{k_2})_+ \esm^{\Sigma_{k_1}\times \Sigma_{k_2}} Y_1\sma Y_2 \Bigr) \sma Y_3\biggr),
\end{equation*}
and we claim this space is isomorphic to 
\begin{equation}  \label{interpretleft}
   \Iso(\fset{k_1{+}k_2{+}k_3},(\fset{k_1}\cup\fset{k_2})\cup\fset{k_3})_+ \esm^{(\Sigma_1 \times \Sigma_2) \times \Sigma_3} (Y_1 \sma Y_2) \sma Y_3.
\end{equation}
To clarify the notations, $\fset{k_1{+}k_2{+}k_3}$ denotes the standard finite set of cardinality $k_1{+}k_2{+}k_3$,
$\fset{k_1{+}k_2}{\cup}\fset{k_3})$ denotes the disjoint union of finite sets of cardinality $k_1{+}k_2$ and $k_3$, and so on. 
Similarly, the value of the righthand sequence of arrows is
\begin{equation*}
    \Iso(\fset{k_1{+}k_2{+}k_3}, \fset{k_1}\cup\fset{k_2{+}k_3})_+ \esm^{\Sigma_{k_1}\times \Sigma_{k_2+k_3}}
  \biggl( Y_1\sma \Bigl( \Iso(\fset{k_2+k_3},\fset{k_2} \cup \fset{k_3})_+ \esm^{\Sigma_{k_2}\times \Sigma_{k_3}}  Y_2 \sma Y_3 \Bigr) \biggr),
\end{equation*}
which we claim is isomorphic to 
\begin{equation}  \label{interpretright}
   \Iso(\fset{k_1{+}k_2{+}k_3},\fset{k_1}\cup(\fset{k_2}\cup\fset{k_3}))_+ \esm^{\Sigma_1 \times (\Sigma_2 \times \Sigma_3)} Y_1 \sma (Y_2 \sma Y_3).
\end{equation}
The spaces in \eqref{interpretleft} and \eqref{interpretright} are isomorphic by combining the
associativity isomorphisms for disjoint union, cartesian products of groups, and the smash product $\sma$.   Thus, we have proved
that the box-tensor operations are naturally associative, granting the two isomorphisms.

To see one of these isomorphisms requires several steps.  We concentrate on the first case, since the second is completely parallel.
First, since $\Iso( \fset{k_3}, \fset{k_3}) = \Sigma_{k_3}$, there is an isomorphism
\begin{equation}   \label{contractioniso}
  \Iso(\fset{k_3}, \fset{k_3})_+ \esm^{\Sigma_{k_3}} Y_3 \stackrel{\iso}{\lra} Y_3
\end{equation}
in $\cR_f(X, \Sigma_3)$ induced by the formula $[f_3 , y] \mapsto f_3^{-1}y$.  
With the right action of $\Sigma_{k_3}$ on $\Iso(\fset{k_3}, \fset{k_3})$ given by $f\cdot\sigma = \sigma^{-1}\com f$,
we have $[f_3\cdot \sigma, y] \mapsto (\sigma^{-1}f_3)^{-1}y = f_3^{-1}\sigma y$; 
starting from $[f_3, \sigma y]$, we also arrive at $f_3^{-1}\sigma y$. 
Thus, a map $ \Iso(\fset{k_3}, \fset{k_3})_+ \esm^{\Sigma_{k_3}} Y_3 \ra Y_3$ exists.  Surjectivity is clear. 
For injectivity, if $[f_3, y]$ and $[f'_3, y']$ map to the same element of $Y$, we have $f_3^{-1}y = (f'_3)^{-1}y'$.  Putting
$\sigma = f_3'f_3^{-1}$, we have $y'=\sigma y$ and  $f'_3 \cdot \sigma = \sigma^{-1}f'_3 = f_3(f'_3)^{-1}f'_3 = f_3$, so
so $[f_3 , y ] = [f'_3 \cdot \sigma, y] \sim [f'_3 , \sigma y] = [f'_3, y'] $.  For equivariance, recall that the left action
of $\Sigma_{k_3}$ on $\Iso(k_3, k_3)$ is given by $\sigma \cdot f_3 =   f \com \sigma^{-1} $.  Thus, 
\begin{equation*}
 [ \sigma * f, y] = [ f \com \sigma^{-1}, y] \mapsto ( f \com \sigma^{-1})^{-1} y = \sigma( f^{-1} y)
\end{equation*}
shows equivariance.

Consequently,
\begin{equation*}
  \Iso(\fset{k_1{+}k_2{+}k_3}, \fset{k_1{+}k_2}\cup\fset{k_3})_+ \esm^{\Sigma_{k_1+k_2}\times \Sigma_{k_3}}
  \biggl(  \Bigl( \Iso(\fset{k_1+k_2},\fset{k_1} \cup \fset{k_2})_+ \esm^{\Sigma_{k_1}\times \Sigma_{k_2}} Y_1\sma Y_2 \Bigr) \sma Y_3\biggr)
\end{equation*}
is isomorphic to 
\begin{multline*}
  \Iso(\fset{k_1{+}k_2{+}k_3}, \fset{k_1{+}k_2}\cup\fset{k_3})_+ \esm^{\Sigma_{k_1+k_2}\times \Sigma_{k_3}} \\
  \biggl(  \Bigl( \Iso(\fset{k_1+k_2},\fset{k_1} \cup \fset{k_2})_+ \esm^{\Sigma_{k_1}\times \Sigma_{k_2}} Y_1\sma Y_2 \Bigr) 
     \sma \Bigl( \Iso(\fset{k_3}, \fset{k_3})_+ \esm^{\Sigma_{k_3}} Y_3 \Bigr) \biggr).
\end{multline*}
Applying a commutativity isomorphism of the product $\esm$, this object is isomorphic to 
\begin{multline*}
    \biggl( \Iso(\fset{k_1{+}k_2{+}k_3}, \fset{k_1{+}k_2}\cup\fset{k_3})_+ \esm^{\Sigma_{k_1+k_2}\times \Sigma_{k_3}}  \\
     \Bigl( \Iso(\fset{k_1+k_2},\fset{k_1} \cup \fset{k_2})_+ \esm  \Iso(\fset{k_3}, \fset{k_3})_+ \Bigr) \biggr) 
        \esm^{(\Sigma_{k_1}\times \Sigma_{k_2}) \times \Sigma_{k_3}}
               (Y_1 \sma Y_2) \sma Y_3.
\end{multline*}
Now we claim there is an isomorphism
\begin{multline*}
  \biggl( \Iso(\fset{k_1{+}k_2{+}k_3}, \fset{k_1{+}k_2}\cup\fset{k_3})_+ \esm^{\Sigma_{k_1+k_2}\times \Sigma_{k_3}}
     \Bigl( \Iso(\fset{k_1+k_2},\fset{k_1} \cup \fset{k_2})_+ \esm  \Iso(\fset{k_3}, \fset{k_3})_+ \Bigr) \biggr) 
\\
 \iso
   \Iso\bigl(\fset{k_1{+}k_2{+}k_3},(\fset{k_1}\cup\fset{k_2})\cup\fset{k_3}\bigr)_+
\end{multline*}
induced by the formula 
$[f_{123}, [ f_{12} , f_3] ] \mapsto (f_{12}, f_3) \com f_{123}$.  We check that balanced expressions in 
\begin{equation*}
  \biggl( \Iso(\fset{k_1{+}k_2{+}k_3}, \fset{k_1{+}k_2}\cup\fset{k_3})_+ \esm
     \Bigl( \Iso(\fset{k_1+k_2},\fset{k_1} \cup \fset{k_2})_+ \esm  \Iso(\fset{k_3}, \fset{k_3})_+ \Bigr) \biggr)
\end{equation*}
map to the same element of the target.  
\begin{equation*}
[f_{123}\cdot (g_{12},g_3), [f_{12}, f_3]]= [(g_{12}^{-1}, g_3^{-1})\com f_{123}, [f_{12}, f_3]] 
                              \mapsto (f_{12}, f_3)\com \bigl( (g_{12}^{-1}, g_3^{-1})\com f_{123} \bigr)
\end{equation*}
On the other hand, 
\begin{equation*}
  [f_{123},( g_{12},g_3)\cdot [f_{12}, f_3]]= [f_{123},  [ f_{12} \com g_{12}^{-1} , f_3 \com g_3^{-1}]]
                              \mapsto (f_{12} \com g_{12}^{-1} , f_3 \com g_3^{-1}) \com f_{123}
\end{equation*}
and these expressions are the same, by associativity of composition of functions.  Now suppose that
$[f_{123}, [ f_{12} , f_3] ]$ and $[f'_{123}, [ f'_{12} , f'_3] ]$ map to the same isomorphism.
The equation $(f_{12}, f_3) \com f_{123} = (f'_{12}, f'_3) \com f'_{123}$ is equivalent to 
$(f'_{12}, f'_3)^{-1} \com (f_{12}, f_3)  =  f'_{123} \com  f_{123}^{-1}$.
 Putting $(\sigma_{12}, \sigma_3) = (f'_{12}, f'_3)^{-1} \com (f_{12}, f_3)  = f'_{123} \com  f_{123}^{-1}$, we have
\begin{align*}
    f'_{123} \cdot (\sigma_{12}, \sigma_3) &= (f'_{123} \com f_{123}^{-1})^{-1} \com f'_{123} = f_{123},
\intertext{and} 
  (\sigma_{12}, \sigma_3) \cdot (f_{12}, f_3) &= (f_{12}, f_3) \com \bigl( (f'_{12}, f'_3)^{-1} \com (f_{12}, f_3) \bigr)^{-1}
= (f'_{12}, f'_3),
\end{align*}
so that 
\begin{equation*}
  [f_{123}, [f_{12}, f_3]] = [f'_{123}\cdot (\sigma_{12}, \sigma_3), [f_{12},f_3]] 
\sim
[f'_{123}, (\sigma_{12}, \sigma_3)\cdot [f_{12}, f_3]] = [f'_{123}, [f'_{12}, f'_3]].  \qedhere
\end{equation*}
\end{proof}
The diamond operation $\dop_{k,1} = \dop_k$ requires some preliminary definitions.  First recall the category
of filtered objects $F_k\cR_f(X)$ from \cite[section 1.1]{Waldhausen85}; this is a category with cofibrations
and weak equivalences.  
Let 
\begin{equation*}
\xy \UCMT 
\xymatrix@1{ \underline{P} = (P_1 \ar@{ >->}[r] & P_2 \ar@{ >->}[r] &\cdots \ar@{ >->}[r]&  P_k) } 
\endxy
\end{equation*}
be an object of $F_k\cR_f(X)$.
For functions $f, g \colon  \fset{k} \ra \fset{k}$ we say 
$f \leq g$ if $f(i) \leq g(i)$ for all $i \in \fset{k}$.
 Let 
$I(\fset{k}) = \{f \colon \fset{k} \ra \fset{k}\; | \; \text{There is $\sigma \in \Sigma_k$ such
that $f \leq \sigma$.} \}$.

The set $I(\fset{k})$ is partially ordered by $\leq$, and the sequence 
$\underline{P}$ defines a functor 
$\underline{\underline{P}} \colon I(\fset{k}) \ra \cR_f(X)$ by the rule
$\underline{\underline{P}}(f) = P_{f(1)} \sma P_{f(2)} \sma \cdots \sma P_{f(k)}$ on objects. 
We apply the convention that parentheses in iterated products are collected to the left.  
In particular, $P_{f(1)} \sma P_{f(2)} \sma P_{f(3)} = (P_{f(1)} \sma P_{f(2)}) \sma P_{f(3)}$, and, in general, 
\begin{equation*}
   P_{f(1)} \sma P_{f(2)} \sma \cdots \sma P_{f(k)} = \bigl(\cdots ( P_{f(1)} \sma P_{f(2)}) \sma \cdots \sma P_{f(k)}\bigr).
\end{equation*}
For arrows we  observe that $f \leq g$ implies there are cofibrations
$\xy \UCMT \xymatrix@1{P_{f(i)} \ar@{ >->}[r]& P_{g(i)}} \endxy$
which induce a cofibration
$\xy \UCMT 
\xymatrix@1{ \underline{\underline{P}}(f) \ar@{ >->}[r]&  \underline{\underline{P}}(g)} 
\endxy$.
This depends on the exactness of $\sma$,  proved in theorem \ref{symmetricbimonoid}. 
\begin{definition}  \label{basicdiamond}
Define the functor $  \dop_k \colon F_k\cR_f(X) \ra \cR_f(X, \Sigma_k, \{{\rm all}\})$ on objects by making a choice of
$\colim_{I(\fset{k})} \underline{\underline{P}}$ and setting
\begin{equation*}
  \dop_k( \underline{P} ) = \colim_{I(\fset{k})} \underline{\underline{P}}.
\end{equation*}
The definition extends to arrows by the universal property of the colimit.  The $\Sigma_k$ action is 
induced by the permutation of factors. 
\end{definition}
\begin{example}
  \label{keyex1diamond}
Applied to a constant cofibration sequence 
$\underline{Y} = (\xy \UCMT \xymatrix{Y \ar@{ >->}[r]^= & Y \ar@{ >->}[r]^= & \cdots \ar@{ >->}[r]^= & Y } \endxy) $
of length $k$, we obtain simply
\begin{equation*}
  \dop_k (\underline{Y}) = Y \sma Y \sma \cdots \sma Y
\end{equation*}
with the group $\Sigma_k$ permuting the factors.   Thus, the purpose of $\dop_k$ is to extend $\sma$-powers to filtered objects.
\end{example}
\begin{definition}  \label{gendiamond}
  The generalized diamond operation 
  \begin{equation*}
\dop_{n,k} \colon F_k \cR_f(X, \Sigma_n, \{{\rm all}\}) \ra \cR_f(X , \Sigma_{nk}, \{{\rm all}\})
  \end{equation*}
is a composition
\begin{equation*}
  \dop_{n,k} \colon 
     F_k \cR_f(X, \Sigma_n, \{{\rm all}\}) \stackrel{\dop_k}{\lra} \cR_f(X , B_{n,k}, \{{\rm all}\}) 
                \stackrel{\Ind_{B_{n,k}}^{\Sigma_{nk}}}{\lra} \cR_f(X , \Sigma_{nk}, \{{\rm all}\}),
\end{equation*}
 with a basic diamond operation $\dop_k$ followed by  an induction construction $\Ind_{B_{n,k}}^{\Sigma_{nk}}$.
The intermediate group $B_{n,k}$ is the group of block permutations of $nk$ objects blocked into $k$ groups of $n$ objects. 
Thus, the group $B_{n,k}$ is a wreath product: $B_{n,k}=\Sigma_k \wr \Sigma_n$.
Explicitly, there is a short exact sequence of groups $1 \ra (\Sigma_n)^k  \ra B_{n,k} \ra \Sigma_k \ra 1$.
 \end{definition}

Recall $G_{\bullet}$ briefly here, following \cite{GSVWKthy}.
For a simplicial set  $Z$ the corresponding simplicial path set $PZ$ is defined by 
$PZ_n = Z_{n+1}$.  The face operator $d_i \colon PZ_n \ra PZ_{n-1}$ coincides with 
$d_{i+1} \colon Z_{n+1} \ra Z_n$; the degeneracy operator $s_i \colon PZ_n \ra PZ_{n+1}$ coincides 
with $s_{i+1} \colon Z_{n+1} \ra Z_{n+2}$.  The face operator $d_0 \colon Z_{n+1} \ra Z_n$ induces 
a simplicial map $d_0 \colon PZ \ra Z$.  The simplicial set $PZ$ is simplicially homotopy 
equivalent to the constant simplicial set $Z_0$ \cite[Lemma 1.5.1, p.328]{Waldhausen85}.
Viewing $Z_1 = PZ_0 $ as another constant simplicial set provides a simplicial map $Z_1 \ra PZ$.
\begin{definition}\cite[p.257]{GSVWKthy}\label{Gdotdef}
For a category $\cC$ with cofibrations and weak equivalences the simplicial category $G_{\bullet}\cC$  is defined by the cartesian square
\begin{equation}  \label{Gdotdiagram}
\xy \UCMT \xymatrix{
  wG_{\bullet}\cC  \ar[r] \ar[d]  &   P w S_{\bullet} \cC  \ar[d]_{d_0}
\\
  P w S_{\bullet} \cC    \ar[r]^{d_0}   &   wS_{\bullet}\cC
}  
\endxy
\end{equation}
\end{definition}
Recalling a few more details from \cite{GSVWKthy}, 
$G_{\bullet} \cC$ has cofibrations and weak equivalences.  
Since  $G_n\cC = (PS_{\bullet}\cC)_n \times_{S_n\cC} (PS_{\bullet} \cC)_n = S_{n+1}\cC \times_{S_n\cC} S_{n+1}\cC$, 
the weak equivalences and cofibrations in $w G_{\bullet} \cC$ are given by pullback.
There is also a stabilization map $\eta \colon \cC \ra G_{\bullet}\cC$, where $\cC$ is viewed as a constant simplicial category with cofibrations
and weak equivalences, defined as follows.
We have the map $\cC  = (PwS_{\bullet}\cC)_0 \ra PwS_{\bullet}\cC$ and the constant map  $\cC \ra PwS_{\bullet}\cC$ carrying 
$\cC$ to the terminal object.  These two combine
to give an inclusion $\eta \colon \cC \ra G_{\bullet}\cC$.  
After passing to diagonals, the construction may be iterated so there results a stabilization sequence 
\begin{equation*}
  \cC \ra  G_{\bullet}\cC  \ra G_{\bullet} \bigl( G_{\bullet}\cC \bigr) 
             \ra \cdots \ra G_{\bullet}^n\cC \mathrel{\mathop :}=G\bigl(G_{\bullet}^{n-1}\cC \bigr) \ra \cdots 
                   \ra \colim_n G^n_{\bullet}\cC  \mathrel{\mathop :}=  G_{\bullet}^{\infty}\cC
\end{equation*}
of simplicial categories with cofibrations and weak equivalences.  
Returning to the square \eqref{Gdotdiagram}, after passage to nerves in the $w$-direction, diagonalization,  and geometric realization, 
there results a natural map
\begin{equation*}
  \abs{wG_{\bullet} \cC} \ra \Loops \abs{wS_{\bullet}\cC}.
\end{equation*}
This may not always be a homotopy equivalence, but it is a homotopy equivalence when 
$\cC$ has a property called pseudo-additivity
\cite[Definition 2.3 and Theorem 2.6]{GSVWKthy}. 
In our case, with $\cC = \cR_f(X)$ we follow 
\cite{GSVWKthy} to achieve the pseudo-additivity property by 
using the cylinder functor defined in
\cite[Section 1.6]{Waldhausen85}.  The cylinder functor 
induces a cone functor $c \colon \cR_f(X) \ra \cR_f(X)$ 
and a suspension functor $\Sigma \colon \cR_f(X) \ra \cR_f(X)$ so that we may define
a category of prespectra 
\begin{equation*}
  \Sigma^{\infty} \cR_f(X) = \colim \bigl( \cR_f(X) \stackrel{\Sigma}{\ra}   \cR_f(X) \stackrel{\Sigma}{\ra}   
                \cR_f(X) \stackrel{\Sigma}{\ra} \cdots \bigr).
\end{equation*}
Then $\Sigma^{\infty}\cR_f(X)$ has the pseudo-additivity property \cite[Remark 2.4 and Lemma 2.5, p.258--259]{GSVWKthy} so 
\begin{equation*}
   \abs{wG_{\bullet} \Sigma^{\infty}\cR_f(X)} \ra \Loops \abs{wS_{\bullet}  \Sigma^{\infty}\cR_f(X)}
\end{equation*}
is a weak homotopy equivalence.  Additionally, by 
\cite[Proposition 1.6.2]{Waldhausen85}, 
$\abs{wS_{\bullet}\cR_f(X)} \ra \abs{wS_{\bullet}\Sigma^{\infty}\cR_f(X)} $
is a weak homotopy equivalence. 

Additionally, we need the fact that there  are weak homotopy equivalences
\begin{equation}
  \label{eq:Ginfty}
  \abs{wG_{\bullet}^{\infty} \cC} \ra \Loops\abs{wG_{\bullet}^{\infty} S_{\bullet}\cC}  \la \Loops \abs{wS_{\bullet}\cC}
\end{equation}
 for any category $\cC$ with cofibrations
and weak equivalences 
\cite[Theorem 4.1, p.268]{GSVWKthy}.

The $G_{\bullet}\cC$ construction has an explicit description as a category of exact functors.
For full  details  refer to \cite{GSAthyops} and \cite{GraysonOps}.  
First, extend the partially ordered set $A \in \Delta$ to the set 
$\gamma(A) =\{L,R\} \coprod A$ 
with the ordering in which  $L$ and $R$ are not comparable; $L < a$ and $R < a$ for every $a \in A$; and, for $a, a' \in A$, 
$a < a'$ in $\gamma(A)$ if and only if $a < a'$ in $A$.   Pictorially, for $A = [n]$,  $\gamma(A)$ looks like
\begin{equation*}
  \xy \UCMT \xymatrix@R=1ex{
L \ar[dr] 
\\
 &   0  \ar[r] & 1 \ar[r] & \cdots \ar[r] & n
\\
R \ar[ur]
}
\endxy
\end{equation*}
The category $\Gamma(A)$ is the category of arrows in $\gamma(A)$, omitting the identity arrows on $L$ and $R$.  Diagrammatically,
$\Gamma(A)$ looks like
\begin{equation*}
  \xy \UCMT \xymatrix@R=2ex@C=1.5em{
  &  0/L \ar@{ >->}[rr] \ar'[d][dd] &  & 1/L \ar@{ >->}[rr]  \ar'[d][dd] & & \cdots \ar@{ >->}[rr] & & n/L  \ar[dd]
\\
  0/R \ar@{ >->}[rr]  \ar[dr] & & 1/R \ar@{ >->}[rr] \ar[dr] & &   \cdots \ar@{ >->}[rr] &&  n/R \ar[dr]
\\
 &   0/0 \ar@{ >->}[rr]  & & 1/0 \ar@{ >->}[rr] \ar[d] & &  \cdots \ar@{ >->}[rr] & &  n/0 \ar[d]
\\
       &  &  & 1/1 \ar@{ >->}[rr] & & \cdots \ar@{ >->}[rr] & & n/1 \ar[d]
\\   
& &  & & & & \cdots  & \stackrel{\cdot}{:} \ar[d] 
\\
& & & & & & & n/n.
}
\endxy
\end{equation*}
Here $a/b$ stands for $b \ra a$ (or $b < a$), and an arrow $a/b \ra c/d$ stands for a square
$\begin{smallmatrix} \xy \UCMT  \xymatrix@!0{ a \ar[r] & c \\ b \ar[u] \ar[r] & d \ar[u] } \endxy \end{smallmatrix} $
in $\gamma(A)$.   The exact sequences in $\Gamma(A)$ are sequences
$j/k \ra i/k \ra i/j$ where $k \ra j \ra i$ in $\gamma(A)$.  
 Then, for  $A \in \Delta$,
\begin{equation*}
  G_A \cC = \Exact(\Gamma(A), \cC). 
\end{equation*}
Since $\Gamma(A)$ is functorial in $A$, preserving exact sequences $j/k \ra i/k \ra i/j$, 
we have another description of $G_{\bullet} \cC \colon \Delta^{\rm op} \ra \Cat$. 
In this interpretation the stabilization $\eta \colon \cC \ra G_{\bullet}\cC$ sends an object
$C$ of $\cC$ to the functor $\eta(C) \colon \Gamma(A) \ra \cC$ whose value at $i/L$ is $C$ for all $i \in A$
and whose value at any other object of $\Gamma(A)$ is the zero object of $\cC$. Given an arrow $i/L \ra i'/L$
in $\Gamma(A)$, $\eta(C)$ assigns to it the identity on $C$; other arrows are assigned by the universal property of the zero object. 
\begin{definition}[Compare 2.1, p.268, \cite{GSAthyops} and Section 4, \cite{GraysonOps}] \label{subdivision}
 Let $Z$ be a simplicial object in a category $\cD$. Define a 
concatenation operation ${\rm con} \colon \Delta^k \ra \Delta$.  For a sequence $(A_1, \ldots, A_k)$ of finite non-empty
ordered sets, order their disjoint union $A_1 \coprod \cdots \coprod A_k$ so that the subset $A_i$ inherits
 the original order and so that, if $i \leq j$ and $a_i \in A_i$ and $a_j \in A_j$, then $a_i < a_j$. 
Then define the  $k$-fold edgewise subdivision of a simplicial object $Z$ to be the composite functor
\begin{equation*}
  \sub_k Z \colon \Delta^k \stackrel{{\rm con}}{\lra} \Delta \stackrel{Z}{\lra} \cD
\end{equation*}
For a simplicial set $Z$ there is a natural homeomorphism $\abs{\sub_k Z} \ra \abs{Z}$.
\end{definition}
Several more constructions are necessary before we can define for every integer $k \geq 1$ operations
\begin{equation*}
  \omega^k \colon w \sub_kG_{\bullet}\cR_f(X) \lra wG^k_{\bullet}\cR_f(X, \Sigma_k, \{{\rm all}\}).
\end{equation*}
The original framework has proved to be quite robust, so we refer to  
\cite[Sections 5 through~7]{GraysonOps} and \cite[Section 2]{GSAthyops} 
for complete details and summarize what we use.
\begin{theorem}[Sections 5 through 7, pp.253--257, of \cite{GraysonOps}]\label{Gammacategories}
For $A \in \Delta$, let $\Gamma^1(A)$ be the category $\Gamma(A)$ discussed before definition \ref{subdivision}.
\begin{enumerate}
\item For each $A \in \Delta$ and for each integer $k \geq 1$ there is a category with exact sequences
$\Gamma^k(A)$.  The category $\Gamma^k(A) $ is natural in the variable $A$.
\item For $A_1, \ldots, A_k \in \Delta$, let $A_1\ldots A_k$ be the concatenation.  There is a functor
  \begin{equation*}
    \Xi_k \colon \Gamma(A_1) \times \cdots \times \Gamma(A_k) \lra \Gamma^k(A_1 \ldots A_k).
  \end{equation*}
the functor $\Xi_k$ is multi-exact, i.e., exact in each variable separately, and is natural in
each of the variables.  \qed
\end{enumerate}  
\end{theorem}
Grayson \cite[pp.255--256]{GraysonOps} enumerates compatibility conditions 
(E1), (E2), (E3), (E4), and (E5)
abstracted from properties of higher exterior powers and tensor products when applied to 
filtered modules.  Given that the box-tensor operations $\btn$ and diamond operations $\dop_{n,k}$ fulfill 
(E1), (E2), (E3), (E4), and (E5)
the robustness of the framework enables us to make the following observation.
\begin{definition}  \label{lambdafunctors}
For $A \in \Delta$, the collection of operations $\dop_{n,k} $ and $\btn$ define functors
  \begin{equation*}
    \Lambda^k_{\dop, \btn} \colon \Exact(\Gamma(A) , \cR_f(X, \Sigma_n, \{ {\rm all} \}) )
           \lra  \Exact( \Gamma^k(A) , \cR_f(X, \Sigma_{n,k}, \{ {\rm all} \}) ).
  \end{equation*}
These functors are natural in $A$.
\end{definition}
\begin{remark}
Since we don't need the explicit formula for $\Lambda^k_{\dop, \btn}$ except in a few specific cases, we 
refer the reader to the discussion in \cite[p.256--257]{GraysonOps} for all the details.
For guidance, we point out that the categories $\Gamma^k(A)$ mentioned in theorem \ref{Gammacategories} 
are constructed precisely to deliver the definition of $\Lambda^k_{\dop, \btn}$ on an object. 
Properties (E1) through (E4) ensure that the formulas on arrows yield a well-defined functor.  
Property (E5) of Grayson's list ensures that the functors $\Lambda^k_{\dop, \btn}$ carry an exact functor $M$
to another exact functor.  
\end{remark}
In our situation we need the following property of a category with cofibrations. 
\begin{definition} (Compare Definition 4.3, \cite[p.274]{GSVWKthy}.) \label{extensiondefn}
  A category $\cC$ with cofibrations has the extension property if for all commutative diagrams of cofibration sequences
  \begin{equation*}
    \xy \UCMT \xymatrix{
A \ar@{ >->}[r] \ar@{ >->}[d]  &    B  \ar@{ >->}[r] \ar[d]_{i}  &    C \ar@{ >->}[d]
\\
A'\ar@{ >->}[r] &    B' \ar@{ >->}[r]   &    C'
}
\endxy
  \end{equation*}
 in $\cC$, with vertical cofibrations as indicated, it follows that the middle arrow $i$ is also a cofibration.
\end{definition}
\begin{lemma} \label{extproperty}
Let $\cC$ be a category with cofibrations,  $A_1, \ldots A_k \in \Delta$, and let $A_1 \ldots A_k$ be the concatenation.
  \begin{enumerate}
  \item If $\cC$ has the extension property, then the natural inclusion
    \begin{equation*}
      G^k_{A_1 \ldots A_k}\cC \lra \Exact(\Gamma(A_1){\times}\cdots{\times}\Gamma(A_k), \cC)
    \end{equation*}
is an isomorphism.
\item The categories $\cR_f(X, \Sigma_n, \cF)$ with cofibrations have the extension property.
  \end{enumerate}
  \end{lemma}
\begin{proof}
Part 1 is \cite[4.4 Remark, p.274]{GSAthyops}.
 For part 2, because we are working inside $\cR(X)$ with simplicial sets, cofibrations are the injective maps.  Therefore, the 
extension property holds for $\cR_f(X, \Sigma_n, \cF)$. 
\end{proof}
\begin{proposition}[Proposition 4.5 of \cite{GSAthyops}]
  \label{input1}
  The box-tensor operations and the
diamond operations fulfill properties (E1) through (E5).
\end{proposition}
\begin{proof} 
Properties (E1) through (E4) are  consequences of the symmetric bimonoidal structure of theorem \ref{symmetricbimonoid}. 
Properties (E3) and (E4) also 
depend on the associativity of $\btn$ established in proposition \ref{boxtensorassociative}.
Property (E5) depends on the extension property of definition \ref{extensiondefn} and  takes some additional work 
manipulating cocartesian diagrams, cofibration sequences and colimits.
The necessary steps are laid out in \cite[Lemmas 4.6 through 4.10, Corollary 4.11]{GSAthyops}.
Because all those manipulations rely just on the coherence of the symmetric bimonoidal category structure,
all steps work in the present, more general, situation. 
\end{proof}
The subdivision construction (concatenation), the functors $\Lambda^k_{\dop, \btn}$, and the functors $\Xi_k$ come into play
in the following definition. 
\begin{definition}  \label{omegafunctors}
  For $k \geq 1$, the components $\omega^k$ for the total Segal operation are defined as follows.
\begin{equation}
  \label{operationdiagram}
  \xy \UCMT \xymatrix{
 \Exact( \Gamma(A_1 \ldots A_k ), \cR_f(X, \Sigma_n, \{ {\rm all} \}) )   \ar[r]^{\sub_k\Lambda^k} \ar[d]_{\omega^k}
              &   \Exact( \Gamma^k(A_1\ldots A_k) , \cR_f(X, \Sigma_n, \{ {\rm all} \}) ) \ar[dl]_{\Xi_k}
\\
  \Exact( \Gamma(A_1){\times}\cdots{\times}\Gamma(A_k ), \cR_f(X, \Sigma_n, \{ {\rm all} \}) )  & 
}
\endxy
\end{equation}
\end{definition}
By part 1 of lemma \ref{extproperty} we may interpret 
$\Exact( \Gamma(A_1){\times}\cdots{\times}\Gamma(A_k ), \cR_f(X, \Sigma_n, \{ {\rm all} \}) )$
as 
$G^k_{A_1\ldots A_k} \cR_f(X, \Sigma_n, \{ {\rm all} \})$. 
 The result is a family of functors
\begin{equation*}
  \omega^k \colon w \sub_k G_{\bullet} \cR_f(X) \lra wG^k_{\bullet}\cR_f(X, \Sigma_k, \{\text{all}\})
\end{equation*}
for $k \geq 1$. 

Referring to the discussion preceding definition \ref{subdivision}, 
the stabilization $\eta \colon \cR_f(X) \ra G_0\cR_f(X)$ 
has been concisely written in \cite{GSAthyops} as follows.
\begin{equation*}
  \xy \UCMT \xymatrix@R=-1.5ex@C=1ex{
  &   & Y & \\ 
 (Y,r,s) \mapsto & \eta\bigl((Y,r,s)\bigr) =  \ar@{=}[rr] & & 
 \\  &   & X &} \endxy 
\end{equation*}
The extension to higher simplicial dimensions 
admits the description $(s_0)^k\bigl( \eta(Y) \bigr)$, where $s_0^k \colon G_0\cR_f(X) \ra G_k\cR_f(X) $ is the iterated degeneracy.
This can be denoted
\begin{equation} \label{etadiagram}
  \xy \UCMT \xymatrix@R=-0.25ex@C=1ex{
   &  Y= Y = \cdots = Y  & 
\\
      \ar@{=}[rr] &  &
\\
  &  X = X = \cdots = X & }
\endxy
\end{equation}
where the top row indicates constant filtered object
and the bottom row indicates the constant filtration of the zero object. 
Since $\sub_kG_{\bullet}\cC$ in simplicial dimension 0 can be identified with $G_k\cC$, diagram \ref{etadiagram} also represents 
\begin{equation*}
  \eta \colon \cR_f(X) \lra \sub_k G_0 \cR_f(X)
\end{equation*}
for each $k \geq 1$.  The next example incorporates example \ref{keyex1diamond} and is fundamental.
\begin{example} \label{importantomegavalue}
The formula for the composite
\begin{equation}
  \label{alphaone}
 \tilde{\alpha}_1^k \colon \cR_f(X) 
           \stackrel{\eta}{\ra}  \sub_k G_{\bullet} \cR_f(X) 
              \stackrel{\omega^k}{\lra} G^k_{\bullet}\cR_f(X, \Sigma_k, \{{\rm all}\})
\end{equation}
is the functor $\Gamma([0])^k \ra \cR_f(X, \Sigma_k, \{{\rm all}\})$ given by 
\begin{equation*}
  \begin{cases}
    Y{\sma}Y {\sma}\cdots {\sma} Y,&\text{in positions $0/L$, $0^{(2)}/L^{(2)}$, \ldots, $0^{(k)}/L^{(k)}$,}
\\
    X                             ,& \text{in all other positions.}
  \end{cases}
\end{equation*}
\end{example}
 \section{$E_{\infty}$-structure and restriction to spherical objects} \label{EinftyStructure}
We have already seen that, 
in order to obtain the algebraic $K$-theory of spaces using the $G_{\bullet}$-model,
one uses a category of prespectra $\Sigma^{\infty}\cR_f(X)$ obtained from $\cR_f(X)$ by passage to a limit using a suspension operation.
We are now going to deal with natural transformations of semi-group valued functors
$[ -, \cR_f(X) ] \ra [ -, \{1\}\times \prod_{n\geq 1} A_{\Sigma_n, \{ {\rm all}\}} (X)]$,
where the target is an abelian-group-valued functor.
First we restrict to categories of $n$-spherical objects $\cR_f^n(X)$, whose definition is recalled below.
Segal's group completion theorem
\cite[Proposition 4.1]{SegalGamma} provides a unique natural transformation of 
abelian-group-valued functors 
$[ -, \Loops\abs{hN_{\Gamma}\cR_f^n(X)}] \ra [-, \{1\}\times \prod_{n\geq 1} A_{\Sigma_n, \{ {\rm all}\}} (X)]$.
In the domain, $hN_{\bullet}\cR_f^n(X)$ is the simplicial category arising from the
categorical sum operation $\wed$, as described in 
\cite[Section 1.8]{Waldhausen85}, and maps are weak homotopy equivalences.
 The following diagram displays this result as the diagonal arrow.
\begin{equation}
  \label{gpcompdiagram1}
  \xy \UCMT \xymatrix{
[ - , \Loops\abs{hN_{\Gamma}\cR_f^n(X)}] \ar@{-->}[d]\ar[drr] &  \ar[l]  [-, \abs{h\cR_f^n(X)}]  \ar[r]  &  [-, \abs{h\cR_f(X)}] \ar[d]_{\omega}
\\
[ -, A(X)]   \ar@{-->}[rr]         &                                       & [- ,  \{1\}\times \prod_{n\geq 1} A_{\Sigma_n, \{ {\rm all}\}} (X)]
}
\endxy
\end{equation}
In this section we show that the diagonal arrow is induced by an $E_{\infty}$-map
$\Loops\abs{hN_{\Gamma}\cR_f^n(X)} \ra  \{1\}\times \prod_{n\geq 1} A_{\Sigma_n, \{ {\rm all}\}} (X)$.

But we want 
a natural tranformation of abelian-group-valued functors
$[ -, A(X) ] \ra [ -, 1\times \prod_{n\geq 1} A_{\Sigma_n, \{ {\rm all}\}} (X)]$
as displayed by the lower horizontal arrow in the diagram, and we want it to be induced by an $E_{\infty}$-map 
$A(X) \ra \{1\} \times \prod_{n \geq 1} A_{\Sigma_n, \{ {\rm all}\}} (X) $.
There is a natural chain of equivalences
\begin{equation*}
    \lim_{n \ra \infty}hN_{\bullet}\cR^n(X) \simeq hS_{\bullet}\cR_f(X) \simeq hS_{\bullet}\Sigma^{\infty}\cR_f(X),
\end{equation*}
where the colimit is taken over suspension relative to $X$ \cite[Theorems 1.7.1 and 1.8.1]{Waldhausen85}.  
This implies we have to examine the behavior of our constructions as
they relate to suspension, which we analyse in  section \ref{Suspension}.

We recall from \cite[Section 1.7, p.360]{Waldhausen85} a definition of spherical objects in the category
$\cR_f(X)$, where $X$ is a connected space.
On this category we have the homology theory 
$H_*(Y,r,s) = H_*\bigl(Y, s(X); r^*(\bZ[\pi_1X])\bigr)$
(homology with local coefficients),
and we say $(Y,r,s)$ is $n$-spherical if 
$H_q(Y,r,s) = 0$ for $q \neq n$ and 
$H_n(Y,r,s)$ is a stably-free $\bZ[\pi_1X]$-module. 
 For $n \geq 0$ denote by $\cR_f^n(X)$ the full subcategory of $\cR_f(X)$ whose 
objects are $n$-spherical.  
For example, in case $X$ is a connected simplicial abelian group, 
$\cR_f^n(X)$ contains spaces homotopy equivalent to retractive spaces $(Y,r,s)$
obtained by completing to pushouts diagrams of the form
\begin{equation*}
  \xy \UCMT \xymatrix{
  X  &  \ar[l]_(0.65){\wed    \phi_i} \bigvee_{i=1}^N \partial \Delta^n \ar@{ >->}[r] &  \Delta^n
}
\endxy
\end{equation*}
where the attaching maps $\phi_i$ are constant maps to the identity element of $X$.

Let $\bN$ be the natural numbers $\{0, 1,\ldots\}$,
and $F$ the category of finite subsets of $\bN$ and injections.  
Let $F_+ \subset F$ be the full subcategory of non-empty finite subsets.  
Let $\coprod$ denote the associative sum on $F_+$ given by 
\begin{equation*}
  \{ x_i | 1 \leq i \leq m\} \coprod \{y_j | 1 \leq j \leq n\} = \{x_i | 1 \leq i \leq m \} \cup \{ y_j + x_m - y_1 + 1| 1 \leq j \leq n\}, 
\end{equation*}
where we assume $x_1 < \cdots < x_m$ and $y_1 < \cdots < y_n$. 
\begin{lemma}[See 10.2 Lemma, p.289, \cite{GSAthyops}]
  The category $F_+$ is contractible. 
\end{lemma}
\begin{proof}
  The functor $t \colon F_+ \ra F_+$ defined by $t(x) = \{0\} \coprod x$ receives natural transformations
from the identity functor on $F_+$ and from the constant functor with value $\{0\}$.  Geometric realization 
of the nerve converts the natural transformations to homotopies, so the identity map on the realization of
the nerve of $F_+$ is homotopic to a constant map.
\end{proof}
Under the assumption that the category $\cC$ satisfy the extension property for cofibrations,
which has been verified for $\cR_f(X)$ and $\cR_f(X, \Sigma_n , \{ \rm{all}\})$ in lemma \ref{extproperty} part 2,  
one may identify the iterated $G_{\bullet}$-construction $G^n_{\bullet}\cC$ with 
$\Exact(\Gamma(-)^n, \cC)$ according to lemma \ref{extproperty}, part 1.  
Using the adjointness relation, or diagonals, we have
\begin{equation*}
  G^3_A(\cC) \mathrel{\mathop :}=  \Exact(\Gamma(A)^3, \cC)  
       =   \Exact(\Gamma(A){\times}\Gamma(A), \Exact(\Gamma(A), \cC)) = \cdots = G_A\bigl( G_A(G_A\cC)\bigr),
\end{equation*}
for example.
Now  extend $n \mapsto \Exact(\Gamma(-)^n, \cC) = G^n_{\bullet}\cC$ to 
$G_{\bullet}^{(-)} \colon F \ra Cat^{\Delta^{\rm op}}$
following the recipe in 
\cite{GSAthyops}. 
Thus, on objects $x \in \Ob(F_+)$ and $A \in \Delta$, put 
$G^x_A\cC \mathrel{\mathop :}= \Exact(\Gamma(A)^x, \cC)$.   
To obtain the extension to $F$, identify $\Gamma(A)^{\emptyset}$ with the one point category, so that 
$G_{\bullet}^{\emptyset}\cC \mathrel{\mathop :}= \Exact(\Gamma(A)^{\emptyset}, \cC) = \cC$.

For the behavior on morphisms we distinguish cases.  An isomorphism $x \ra x'$ in $F$ induces a natural morphism 
$G^x_{\bullet}\cC \ra G^{x'}_{\bullet}\cC$ by permuting coordinates.  An injection $i \colon x \ra y$ induces
$G^x_{\bullet}\cC \ra G^y_{\bullet}\cC$ using stabilization
\begin{equation*}
  \xy \UCMT   \xymatrix{
\Gamma(A)^{i(x)} \ar[r]^(0.3){\iso} \ar[d]^{\equiv}
              &  \Gamma(A)^{i(x)} \times \{L/0\}^{y\backslash i(x)} \ar[r] \ar@{-->}[d] 
                             & \Gamma(A)^{i(x)} \times \Gamma(0)^{y\backslash i(x)}  \ar@{-->}[dl]_{X'}
\\
\Gamma(A)^x  \ar[r]_X     &   \cC    &   \ar@{-->}[l]^(0.65){i_*X}   \Gamma(A)^x \times \Gamma(A)^{y \backslash i(x)} \equiv \Gamma(A)^y \ar[u],
}
\endxy
\end{equation*}
where we recall $\Gamma(0) = \{L/0, R/0\}$ is the two point discrete category,
 and we define $X'$ to be zero outside $\Gamma(A)^{i(x)} \times \Gamma(0)^{y\backslash i(x)}$.
This is the $\eta$-stabilization given by inclusion of $\cC$ on the $L$-line in $G_0\cC$,
as described before definition \ref{subdivision}.

Let $F \gro G_A\cC$ be Thomason's homotopy colimit construction,
which is the category consisting of objects $(x, X \colon \Gamma(A)^x \ra \cC)$
and morphisms $(x,X) \ra (y,Y)$ given by $i \colon x \ra y$ in $F$ and a natural transformation
$i_*X \ra Y$ in $G_A^y\cC$, 
\cite[Def 1.2.2]{Thomasonhocolim}.  The unique morphisms $\emptyset \ra x$ in $F$ provide
functors $\cC \ra \Exact(\Gamma(A)^x, \cC)$, eventually  functors
$\cC \ra F \gro G_A\cC$ natural in $A$, and finally a functor $\cC \ra F \gro G_{\bullet}\cC$.
With the next result, we have made a step toward the righthand column of diagram \ref{gpcompdiagram1}.
\begin{theorem}[Compare  \cite{GSAthyops}, 10.3 Theorem] \label{Kthymodels}
  The construction $F \gro wG_{\bullet}\cC$ gives a model for $K$-theory.
\end{theorem}
\begin{proof}[Comments on the proof]
The proof given in \cite{GSAthyops} can be summarized in the following chain of weak homotopy equivalences.
  \begin{multline*}
 \Loops \abs{wS_{\bullet}\cC}  \stackrel{(1)}{\lla} \Loops\abs{wG_{\bullet}^{\infty} S_{\bullet}\cC}  \stackrel{(1)}{\lra}   
         \abs{wG^{\infty}_{\bullet}\cC} \stackrel{(2)}{\lra} \abs{wG_{\bullet}G_{\bullet}^{\infty}\cC} \stackrel{(3)}{\lra} 
\\
\stackrel{(3)}{\lra} \abs{F_+ \gro wG_{\bullet}G^{\infty}_{\bullet}\cC} 
 \stackrel{(4)}{\lla} \displaystyle{\colim_{\tilde{t}} \abs{F_+ \gro wG_{\bullet}\cC}} \stackrel{(5)}{\lla} \abs{F_+ \gro wG_{\bullet}\cC} 
      \stackrel{(6)}{\lra} \abs{F \gro wG_{\bullet}\cC}
\end{multline*}
Concerning the links in the chain, the arrows labeled $(1)$ are recorded in  \eqref{eq:Ginfty}; the arrow $(2)$ results 
from swallowing the extra $G_{\bullet}$ into the colimit defining $G^{\infty}_{\bullet}$. That $(3)$ is an equivalence 
depends on the fact that $\abs{F_+ \gro wG_{\bullet}G^{\infty}_{\bullet}\cC} \ra \abs{F_+}$ can be shown to be a quasifibration
with $\abs{F_+}$ contractible. To account for $(4)$, the functor $t \colon F_+ \ra F_+$ induces a functor 
$\tilde{t} \colon F_+ \gro wG_{\bullet}\cC \ra F_+ \gro wG_{\bullet}\cC$ for which
\begin{equation*}
  \colim_{\tilde{t}} F_+ \gro wG_{\bullet}\cC 
     = \colim \bigl( F_+ \gro wG_{\bullet}\cC \stackrel{\tilde{t}}{\ra} F_+ \gro wG_{\bullet}\cC F_+ 
                   \stackrel{\tilde{t}}{\ra} F_+ \gro wG_{\bullet}\cC \stackrel{\tilde{t}}{\ra} \cdots \bigr)
\end{equation*}
is naturally identifiable to
$F_+ \gro w G_{\bullet}G^{\infty}_{\bullet}\cC$. The realizations of the functors $\tilde{t}$ are all cofibrations, so the
inclusion $(5)$ into the base of the telescope is a weak equivalence.  Finally, cofinality of $F_+$ in $F$ implies
that the arrow $(6)$ is a weak homotopy equivalence. 
\end{proof}
As in 
\cite{GSAthyops} 
the $E_{\infty}$-structure on the total Segal operation will be described in terms of the following diagram.
\begin{equation}
  \label{SeOp}
  \xy \UCMT \xymatrix{
\cR_f^n(X)    \ar[d]^{\alpha_1}  \ar[dr]^{\tilde{\alpha}_1} &  
\\
\{1 \} \times \prod_{n \geq 1} \cR_f(X, \Sigma_n , \{ {\rm all} \})  \ar[r]^{\beta_1}   \ar[d]^{\alpha_2}  
     &   \{1 \} \times \prod_{n \geq 1}  G^n_{\bullet}  \cR_f(X, \Sigma_n , \{ {\rm all} \})   \ar[d]^{\beta_3}
 \\
 \{1 \} \times \prod_{n \geq 1} F \gro G_{\bullet} \cR_f(X, \Sigma_n , \{ {\rm all} \})   \ar[r]^(0.48){\beta_2} 
     &   \{1 \} \times \prod_{n \geq 1} F \gro G_{\bullet} G^n_{\bullet} \cR_f(X, \Sigma_n , \{ {\rm all} \}) 
}
\endxy
\end{equation}
The components of the map $\tilde{\alpha}_1$ are defined in example \ref{importantomegavalue}.
The other maps in diagram \ref{SeOp} are defined as follows.
\begin{definition} \label{SeOpmaps}
    For $\alpha_1$ the $n$th component $\alpha_1(Y)_n = Y \sma \stackrel{\text{$n$ terms}}{\cdots}  \sma Y$, 
 where $\Sigma_n$ acts by permuting factors using the coherence data.

The maps $\beta_1$ and $\beta_2$ come from stabilizations 
$j^n \colon \cR_f(X , \Sigma_n , \{ {\rm all}\}) \ra G^n_{\bullet} \cR_f(X , \Sigma_n , \{ {\rm all}\})$.

The maps $\alpha_2$ and $\beta_3$ are given by the identification
\begin{equation*}
\cR_f(X , \Sigma_n , \{ {\rm all}\}) \iso \{\emptyset \} \,  \textstyle{\gro} \, G_{\bullet} \cR_f(X , \Sigma_n , \{ {\rm all}\}) 
                        \iso  G^{\emptyset}_{\bullet} \, \cR_f(X , \Sigma_n , \{ {\rm all}\}).  
\end{equation*}
\end{definition}
The category $\cR_f^n(X)$ has the pairing derived from the categorical sum $\wdX$.
This feature allows us to dispense with the subdivision construction. 
Each of the four categories in the lower part of the diagram also has a natural
pairing derived from the box tensor pairings  
\begin{equation*}
   \btn_{k,\ell} \colon  \cR_f(X, \Sigma_k, \{{\rm all}\}) \times \cR_f(X, \Sigma_{\ell}, \{{\rm all}\}) 
                     \lra      \cR_f(X, \Sigma_{k+\ell}, \{{\rm all}\}).
\end{equation*}
Underlying the coherence properties of these pairings are the facts established in theorem \ref{symmetricbimonoid}
 that $\cR_f(X)$ is a category with cofibrations and weak equivalences,
with a categorical sum $\vee$ and a symmetric monoidal biexact product $\sma$.
We refer to \cite[pp.291--292]{GSAthyops} for explicit formulas for the pairings, which are given in the abstract context
of a category $\cC$ with cofibrations and weak equivalences and subcategories $\cC_{\Sigma_n}$ of $\Sigma_n$-equivariant objects.
Here we record only notations for use in the next theorem.
\begin{enumerate}
\item There is a  product denoted $\widetilde{\btn}$ on $\{1 \} \times \prod_{n \geq 1} \cR_f(X,  \Sigma_n , \{ {\rm all} \}) $
and a product also denoted  $\widetilde{\btn}$ on $\{1 \} \times \prod_{n \geq 1} G^n_{\bullet}\cR_f(X,  \Sigma_n , \{ {\rm all} \}) $.
\item There is a product denoted $\widehat{\btn}$ on  $  \{1 \} \times \prod_{n \geq 1} F \gro G_{\bullet} \cR_f(X, \Sigma_n , \{ {\rm all} \}) $
and a product also denoted $\widehat{\btn}$  on $\{1 \} \times \prod_{n \geq 1} F \gro G_{\bullet}G_{\bullet}^n \cR_f(X, \Sigma_n , \{ {\rm all} \})$.
\end{enumerate}
\begin{theorem}
  [Compare \cite{GSAthyops}, 10.7 Theorem, p.292]  \label{structure} \leavevmode
  \begin{enumerate}
  \item In the lefthand column of diagram \ref{SeOp}, the categories 
$\{1 \} \times \prod_{n \geq 1} \cR_f(X, \Sigma_n , \{ {\rm all} \})$ and 
$ \{1 \} \times \prod_{n \geq 1} F \gro G_{\bullet} \cR_f(X, \Sigma_n , \{ {\rm all} \})$ 
with respective composition laws $\widetilde{\btn}$ and $\widehat{\btn}$ 
inherit symmetric monoidal structures from the coherence data on $\cR_f(X)$. 
  \item In the righthand column of diagram \ref{SeOp}, the categories 
$ \{1 \} \times \prod_{n \geq 1}  G^n_{\bullet}  \cR_f(X, \Sigma_n , \{ {\rm all} \})$ 
and $\{1 \} \times \prod_{n \geq 1} F \gro G_{\bullet} G^n_{\bullet} \cR_f(X, \Sigma_n , \{ {\rm all} \})   $ 
with  respective composition laws $\widetilde{\btn}$ and $\widehat{\btn}$ 
inherit monoidal structures from the coherence data on $\cR_f(X)$.
  \item The maps $\alpha_1$ and $\alpha_2$ in diagram \ref{SeOp}   are maps of symmetric monoidal categories.
  \item The maps $\beta_1$, $\beta_2$, and $\beta_3$ are maps of monoidal categories. 
  \item The map $\beta_2$ is a homotopy equivalence and in the pseudo-additive case $\beta_3$ is also a 
homotopy equivalence.  
  \item The diagram \ref{SeOp} is commutative in the category of monoidal categories. 
  \end{enumerate}  
\end{theorem}
\begin{theorem}
  \label{powerseries}
Let $X$ be a connected simplicial abelian group.  The functor
\begin{equation*}
  Z \mapsto [Z, \{1\} \times \prod_{n \geq 1} A_{\Sigma_n, \{ {\rm all} \}}(X)]
\end{equation*}
takes values in the category of abelian groups. 
\end{theorem}
\begin{proof}
By theorem \ref{Kthymodels} we take 
\begin{equation*}
   \{1\} \times \prod_{n \geq 1} A_{\Sigma_n, \{ {\rm all} \}}(X) 
              =  \{1 \} \times \prod_{n \geq 1} \abs{F \gro G_{\bullet} \cR_f(X, \Sigma_n , \{ {\rm all} \})}.
\end{equation*}
Since the category  $\{1 \} \times \prod_{n \geq 1} F \gro G_{\bullet} \cR_f(X, \Sigma_n , \{ {\rm all} \})$ 
has a symmetric monoidal structure by part 1 of theorem \ref{structure}, 
the functor $[ -, \{1\} \times \prod_{n \geq 1} A_{\Sigma_n, \{ {\rm all} \}}(X)]$ takes values in the category of abelian monoids.
Repeating the argument of \cite[Lemma 2.3, p.404]{Waldhausen82} shows that values taken are actually in the category of abelian groups. 
  \end{proof}
\begin{proof}[Remarks on the proof of theorem \ref{structure}.]
The entire proof of the analogous result in 
\cite[pp.293--295]{GSAthyops} is essentially a formal appeal to LaPlaza's coherence theorem \cite{LaPlazaCoh1}, so 
it carries over completely.

The reader who investigates further will find 
the symmetry of the pairing on $\{1 \} \times \prod_{n \geq 1} F \gro G_{\bullet} \cR_f(X, \Sigma_n , \{ {\rm all} \})$
involves manipulating products of values of functors
\begin{equation*}
  Y \in G^m_{\bullet}\cR_f(X, \Sigma_n , \{ {\rm all} \})
\; \text{and} \;
Z \in G^n_{\bullet}\cR_f(X, \Sigma_n , \{ {\rm all} \}). 
\end{equation*}
 What is required is comparison of an expression
\begin{equation*}
  Y(i_1/j_1, \ldots, i_m/j_m) \sma Z(i'_1/j'_1, \ldots, i'_n/j'_n) 
\; \text{with} \;
  Z(i'_1/j'_1, \ldots, i'_n/j'_n) \sma Y(i_1/j_1, \ldots, i_m/j_m)
\end{equation*}
and one sees that, not only are commutativity isomorphims for $\sma$ involved, but so are permutations of
inputs, which are taken care of by means of the homotopy colimit.  

Another interesting part of the proof is the claims about the maps $\alpha_1$ and $\alpha_2$,
so it deserves a comment. The biexactness and coherence of $\sma$ give canonical natural isomorphisms $\gamma_n^k$
called Cartan multinomial formulas:
\begin{equation*}
  \gamma^k_n \colon \bigl(\sma \bigr) _n \Bigl( \bigvee_{i=1}^k c_i \Bigr) \stackrel{\iso}{\lra}
        \bigvee_{s_1+\cdots s_k = n} \Ind_{\Sigma_{s_1} \times \cdots \Sigma_{s_k}}^{\Sigma_n} \bigl(\sma\bigr)_{i=1}^k \Bigl( \bigl(\sma \bigr)_{s_i} c_i \Bigr).
\end{equation*}
 These induce natural isomorphisms
\begin{equation*}
  \gamma^k \colon \alpha_1 \com \wdX^k \stackrel{\iso}{\Rightarrow} \bigl( \widetilde{\btn} \bigr)^k \com \alpha_1^k.
\end{equation*}
Then the coherence theorem implies that $\alpha_1$ has a (lax) symmetric monoidal structure.
The functor $\alpha_2$ is the inclusion of a symmetric monoid subcategory, so the assertion for $\alpha_2$ is immediate.

In contrast to the algebraic roles played by $\alpha_1$ and $\alpha_2$, the roles of $\beta_1$, $\beta_2$, and $\beta_3$ 
are to assure us that we are ending in the correct target.  Since the proof that $\beta_3$ is a homotopy equivalence requires 
the pseudo-additivity condition, which is fulfilled by suspension, this part of the argument actually depends on the next section.
\end{proof}
\section{Suspension}  \label{Suspension}
The main theorem of this section is
\begin{theorem}
  \label{infiniteloop}
Let $X$ be a simplicial abelian group.  The total Segal operation
\begin{equation*}
  \omega \colon A(X) \lra  \{1\} \times \prod_{n \geq 1} A_{\Sigma_n, {\rm all}}(X)
\end{equation*}
carries an infinite loop map structure.
\end{theorem}
Section \ref{EinftyStructure} has delivered an infinite loop map
$\Loops\abs{hN_{\Gamma}\cR_f^n(X)} \ra \{1\} \times \prod_{n \geq 1} A_{\Sigma_n, {\rm all}}(X)  $
with domain the $K$-theory of a category of $n$-spherical objects. To obtain theorem \ref{infiniteloop} 
we have to examine the passage to the limit over suspension in view of Waldhausen's result
\begin{equation*}
    \lim_{n \ra \infty}hN_{\bullet}\cR^n_f(X) \simeq hS_{\bullet}\cR_f(X).
\end{equation*}
The technically challenging part is the compatibility of the operations with suspension.
Fortunately, the machinery  set up in  \cite[section 10]{GSAthyops} is sufficiently general
that we need only extend some definitions and quote  a sequence of results to prove our generalization.

First we need a description of the suspension operation that is  amenable to coherence considerations.  
To this end, we go step-by-step through Waldhausen's cone and suspension constructions and identify 
the result with a construction involving the operation $\esm$. 
 The cone construction for $(Y, r,s)$ in $\cR_f(X)$ takes the ordinary mapping
 cylinder of the retraction $M(r)$ and collapses out the cylinder $\Delta^1{\times}X$ so that 
end result is in $\cR_f(X)$.  To amplify the definition, consider the following diagram, 
which fulfills the hypotheses of lemma \ref{iteratedcolimlemma}.
\begin{equation} \label{conediagram}
  \xy \UCMT  \xymatrix{
Y \coprod X & \ar[l]_{\id \coprod r}   \partial \Delta^1 \times Y   \ar@{ >->}[r]   &     \Delta^1 \times Y
\\
X \coprod X \ar[d]  \ar@{ >->}[u] 
                &   \ar[l]          \partial \Delta^1 \times X   \ar@{ >->}[r] \ar[d]   \ar@{ >->}[u]
                                                   &     \Delta^1 \times X \ar@{ >->}[u] \ar[d]
\\
X           &   \ar[l]               X  \ar@{ >->}[r]   &    X
}
\endxy
\end{equation}
Taking the pushouts of the rows produces a diagram
\begin{equation*}
  \xy \UCMT \xymatrix{
X & \ar[l]  \Delta^1  \times X   \ar@{ >->}[r] & M(r)
}
\endxy
\end{equation*}
where $M(r)$ is the usual mapping cylinder of $r$ and the pushout of the top row. As described above, taking the pushout of this diagram
produces $cY$, the underlying space of the cone construction.  
The retraction to $X$ arises from a map of diagram \eqref{conediagram} to a trivial diagram of identity maps 
on $X$; the section $X \ra cY$ and a cofibration $\xy \UCMT \xymatrix{ i \colon Y \ar@{ >->}[r] & cY}\endxy$ arise from canonical
maps of ingredients of the diagram to the colimit.   
Then the suspension $\Sigma Y$ is defined as the pushout of the diagram 
$  \xy \UCMT \xymatrix@1{X & \ar[l]_{r} Y \ar@{ >->}[r] & cY} \endxy$.
\begin{lemma}  \label{smashequalscone}
  For $Y \in \cR_f(X)$ there is a commuting diagram
  \begin{equation}  \label{cone1}
\xy \UCMT \xymatrix{
    \{0\} \times Y   \ar[d]  \ar@{ >->}[r]^{i_0}     &    \Delta^1_1 \esm Y   \ar[d]_{\iso}
\\
   Y \ar@{ >->}[r]^i &     cY 
 }
\endxy 
  \end{equation}
where $\Delta^1_1 \in \cR_f(*)$ is the standard simplicial one-simplex given the base point $1$,
and $i_0$ is induced from the inclusion $\{0\} \ra \Delta^1$.

Moreover,
  \begin{equation*}
 \Sigma Y \mathrel{\mathop :}= cY/Y \iso  S^1  \esm Y,
 \end{equation*}
 where $S^1 = \Delta^1/\partial \Delta^1$ is the standard simplicial circle. 
\end{lemma}
\begin{proof}
  Pass to pushouts in the commutative diagram
  \begin{equation}  \label{smashtocone}
     \xy \UCMT  \xymatrix@C+1ex{
 X \ar[d]  
        &   \ar[l]_(0.75){p_2 \cup rp_2}   \Delta^1_1  \times X  \cup_{ \{1 \}  \times X}  \{1 \} \times Y   \ar@{ >->}[r] \ar[d]_{\id \cup r}   
                     &   \Delta^1 \times Y  \ar[d]
\\
     X   &  \ar[l]_{p_2}   \Delta^1  \times X       \ar@{ >->}[r]  &     M(r)  
}
\endxy
  \end{equation}
to obtain a unique natural map $\eta_1 \colon \Delta^1_1 \esm Y   \ra cY$ making the following diagram commute.
\begin{equation} \label{conecompare1}
  \xy \UCMT \xymatrix@R=1ex{
  &    \Delta^1_1 \esm Y \ar[dd]_{\eta_1}  &   
\\
 \Delta_1^1  \times Y \ar[ur] \ar[dr] &   &   X \ar@{ >->}[ul]_{s'} \ar@{ >->}[dl]
\\
  & cY   & 
}
\endxy
\end{equation}
Restricting $ \Delta_1^1  {\times} Y \ra  \Delta_1^1  \esm Y $ to $  \partial \Delta_1^1 {\times} Y$ yields a diagram
\begin{equation*}
     \xy \UCMT  \xymatrix{
\partial \Delta_1^1  {\times} Y  \ar@{ >->}[r]  \ar[d]_{r'}  &  \Delta_1^1 {\times} Y  \ar[d]
\\
 Y \coprod X  \ar@{ >->}[r]^{i'}   &     \Delta_1^1 \esm Y,
}
\endxy
  \end{equation*}
where $r'(0,y) = y$, $r'(1,y)=r(y)$ and $i'(y) = i_0(y)$, $i'(x) = s'(x)$.
There results a canonical arrow $M(r) \ra  \Delta_1^1 \esm Y $ such that the following square commutes.
\begin{equation*}
  \xy \UCMT \xymatrix{
  \Delta_1^1  \times X \ar@{ >->}[r] \ar[d]_{p_2} &    M(r) \ar[d] 
\\
  X    \ar@{ >->}[r]^{s'}  &      \Delta_1^1  \esm Y .
}
\endxy
\end{equation*}
In turn, there is a unique map $\bar{\eta}_1 \colon cY \ra  \Delta_1^1 \esm Y$ such that the following diagram commutes.
\begin{equation} \label{conecompare2}
  \xy \UCMT \xymatrix@R=1ex{
  &   cY  \ar[dd]_{\bar{\eta}_1}  &   
\\
  \Delta_1^1{\times} Y \ar[ur] \ar[dr] &   &   X \ar@{ >->}[ul] \ar@{ >->}[dl]^{s'}
\\
  &  \Delta^1_1 \esm Y  & 
}
\endxy
\end{equation}
Combining diagrams \eqref{conecompare1} and \eqref{conecompare2} shows that $\eta_1$ and $\bar{\eta}_1$ are mutually inverse
isomorphisms, relative to the common subspace $X$ and compatible with the retractions.

Restricting the left half of diagram \eqref{conecompare1} to $ \{ 0 \}{\times}Y  \subset  \Delta_1^1 {\times} Y$ gives diagram
\eqref{cone1}:
\begin{equation} \label{cone1var}
  \xy \UCMT \xymatrix@R=1ex{
  &   \Delta^1_1\esm Y \ar[dd]_{\eta_1}  &   
\\
 \{ 0 \} {\times} Y  \equiv Y   \ar[ur]^{i_0} \ar[dr]_{i} &   &  
\\
  & cY   & 
}
\endxy
\end{equation}
Replace $S^0 = \{ *, *' \}$ with basepoint $*$ in example \ref{esmaction}  by $\partial\Delta_1^1$ with basepoint~$1$,
and  obtain the diagram
\begin{equation} \label{suspension1}
  \xy \UCMT \xymatrix{
X   \ar[d]_{=}  
           & \ar[l]_{r}  \partial\Delta_1^1  \esm Y   \ar@{ >->}[r] \ar[d]_{\iso}       
                    &  \Delta^1_1 \esm Y  \ar[d]^{\iso}
\\
X  &  \ar[l]_{r}   Y    \ar@{ >->}[r]^{i} \ar@{ >->}[ur]^{i_0}  &   cY.
}
\endxy
\end{equation}
Passage to pushouts shows that the quotient  $(\Delta^1_1 {\esm Y})/(\partial \Delta_1^1 {\esm} Y)$ is isomorphic to $\Sigma Y$ 
in $\cR_f(X)$.  
According to proposition \ref{externalbiexactness} the functor $- \esm Y \colon \cR_f(*) \ra \cR_f(X{\times}\{*\}) \cong \cR_f(X)$
preserves colimits so we deduce 
\begin{equation*}
  (  \Delta^1_1{\esm} Y )/(\partial \Delta_1^1{\esm}Y ) \iso   (\Delta_1^1 / \partial \Delta_1^1) \esm Y  \equiv  S^1  \esm Y,
\end{equation*}
where we define $S^1 \mathrel{\mathop :}=  \Delta_1^1 / \partial \Delta_1^1 $ in $\cR_f(*)$.
\end{proof}
According to proposition \ref{internalizing}, the action of $\cR_f(*)$ on $\cR_f(X)$ may be made internal. Explicitly, there is
a natural isomorphism $i_{e*}S^1 \sma Y \iso S^1 \esm Y$.  In the following we abuse notation slightly and write simply $S^1 \sma Y$
leaving $i_{e*}$ understood, where $i_e \colon \{*\} \ra X$ is the inclusion of the one point space as the identity element of $X$.
We do this to emphasize the dependence of the rest of this section on the coherence of the operation $\sma$.
\begin{proposition}[Compare \cite{GSAthyops},6.1 Proposition,p.283]  \label{lont}
The following diagram commutes up to natural isomorphism.
\begin{equation*}
  \xy \UCMT \xymatrix{
w \sub_k G_{\bullet}\cR_f(X) \ar[r]^{\omega^k} \ar[d]_{S^1 \sma}    
            &    wG^k_{\bullet}\cR_f(X, \Sigma_k, \{{\rm all}\}) \ar[d]_{\dop_k S^1 \sma}
\\
w \sub_k G_{\bullet}\cR_f(X) \ar[r]^{\omega^k}    
                         &     wG^k_{\bullet}\cR_f(X, \Sigma_k, \{{\rm all}\}) }
\endxy
\end{equation*}
\end{proposition}
\begin{proof}
Write $F_1$ for the composite functor $\omega^k \com (S^1 \sma - )$ and $F_2$ for the composite $\dop_k S^1 \sma \omega^k(-)$.
Although $\omega^k(S^1) = \dop_k S^1 = S^1 \sma \stackrel{\text{$k$ terms}}{\cdots} \sma S^1$ we use the $\dop_k$-notation
for orientation purposes.
Following \cite[p.297]{GSAthyops} and \cite[p.257]{GraysonOps}, 
given a functor $M \colon \Gamma(A_1\cdots A_k) \ra \cR_f(X)$ representing an object of
$\sub_kG_{\bullet}\cR_f(X)$, the value of $\omega^k(M)$ on a typical element
 of $\Gamma^k(A_1\cdots A_k)$ has the form
\begin{equation*}
 \bigl( \dop_{n_1} M(-) \bigr)  \btn \bigl( \dop_{n_2} M(-) \bigr) \btn \cdots \btn \bigl( \dop_{n_{k}} M(-)\bigr) = Z_{n_1} \btn \cdots \btn Z_{n_k},
\end{equation*}
where $Z_{n_i} \mathrel{\mathop :}= \dop_{n_i} M(-)$ is an object of $\cR_f(X, \Sigma_{n_i}, \{{\rm all}\})$.
Extending the formulas in the argument of theorem \ref{input1} for the associativity of $\btn$, we write
\begin{equation*}
  Z_{n_1} \btn \cdots \btn Z_{n_k} = \Ind^{\Sigma_{n_1+ \cdots + n_k}}_{\Sigma_{n_1} \times \cdots \times \Sigma_{n_k}}( Z_{n_1} \sma \cdots \sma Z_{n_k} ),
\end{equation*}
and set $n=n_1{+} \cdots {+} n_k$. 

Then a typical  value of $F_1(M)$  has the form
\begin{multline*}
  \Ind^{\Sigma_n}_{\Sigma_{n_1} \times \cdots \times \Sigma_{n_k}}\bigl( (S^1\sma \stackrel{\text{$n_1$ terms}}{\cdots} \sma S^1 \sma Z_{n_1}) 
          \sma \cdots \sma (S^1 \sma \stackrel{\text{$n_k$ terms}}{\cdots} \sma S^1 \sma Z_{n_k}) \bigr)
\\
\iso  \Ind^{\Sigma_{n_1+ \cdots + n_k}}_{\Sigma_{n_1} \times \cdots \times \Sigma_{n_k}}
\bigl( (S^1\sma \stackrel{\text{$n_1$ terms}}{\cdots} \sma S^1) 
      \sma \stackrel{\text{$k$ groups}}{\cdots} \sma (S^1 \sma \stackrel{\text{$n_k$ terms}}{\cdots} \sma S^1 ) \bigr) \sma( Z_{n_1} \sma \cdots \sma Z_{n_k}) \bigr),
\end{multline*}
applying commutativity and associativity isomorphisms.  Now proposition \ref{lontinput} applies to deliver an isomorphism of 
$\Sigma_{n_1+ \cdots + n_k}$-spaces.
\begin{multline*}
  \Ind^{\Sigma_{n}}_{\Sigma_{n_1} \times \cdots \times \Sigma_{n_k}}
\bigl( (S^1\sma \stackrel{\text{$n_1$ terms}}{\cdots} \sma S^1) 
      \sma \stackrel{\text{$k$ groups}}{\cdots} \sma (S^1 \sma \stackrel{\text{$n_k$ terms}}{\cdots} \sma S^1 ) \bigr) \sma( Z_{n_1} \sma \cdots \sma Z_{n_k}) \bigr),
\\
\stackrel{\iso}{\ra}
 ({\dop_k}S^1) \sma \stackrel{\text{$n$ terms}}{\cdots} \sma ({\dop_k}S^1) \sma
\bigl(\Ind^{\Sigma_n}_{\Sigma_{n_1}{\times}\cdots{\times}\Sigma_{n_k}} ( Z_{n_1} \sma \cdots \sma Z_{n_k})  \bigr).
\end{multline*}
This final expression is the value of $F_2$ on the same typical element $M$, so we have a natural isomorphism of functors
$ \epsilon \colon F_1 \Rightarrow F_2$.
\end{proof} 
Now we prove the general lemma \ref{downthenup} and its specialization proposition \ref{lontinput}.
\begin{lemma}  \label{downthenup}
  Let $H$ be a subgroup of $G$, let $ Y \in \cR(X,G)$, and let $Z \in \cR(X, H)$.  By restricting the $G$-action
on $Y$ to $H$, we obtain $Y \sma Z \in \cR(X, H)$, where the action is diagonal.  Then there is a natural isomorphism of
left $G$-spaces
\begin{equation*}
 G_+ \esm^H   (Y \sma Z)  \stackrel{\iso}{\lra} Y \sma (  G_+ \esm^H Z )
\end{equation*}
where the $G$-action on the righthand space is diagonal.  
\end{lemma}
\begin{proof}
  First define a $G$-map $f \colon  G_+ \esm  (Y \sma Z) \ra Y \sma (  G_+ \esm^H Z )  $ by the formula
  \begin{equation*}
    f\bigl(g,  (y,z) \bigr) = \bigl(gy, [g,z]  \bigr).
  \end{equation*}
Applying the equivalence relation defining  $Y \sma  (G_+ \esm^H Z) $,
\begin{equation*}
f\bigl(g, (hy, hz) \bigr) =  \bigl(g(hy), [g, hz]  \bigr) =  \bigl( (gh)y, [gh,z] \bigr) = f \bigl( gh, (y,z) \bigr).  
\end{equation*}
 Therefore, there is an induced $G$-map
\begin{equation*}
  f' \colon  G_+ \esm^H   (Y \sma Z)  \lra  Y \sma (  G_+ \esm^H Z ) .
\end{equation*}
To reverse this map, define $F \colon   Y \sma (G_+ \esm Z)  \ra G_+ \esm^H   (Y \sma Z) $
by the formula
\begin{equation*}
  F\bigl(y,  [  g, z ] \bigr) = \bigl[ g, (g^{-1}y, z)  \bigr] .
\end{equation*}
Now
\begin{multline*}
 F\bigl(y, [gh, z] \bigr) = \bigl[ gh, ( h^{-1}g^{-1}y, z )  \bigr] \\ = \bigl[ g, ( hh^{-1}g^{-1}y, hz)  \bigr] 
   = \bigl[ g, ( g^{-1}y, hz)   \bigr] = F\bigl(y,  [g, hz] \bigr),  
\end{multline*}
so there is an induced $G$-map
\begin{equation*}
  F' \colon Y \sma (  G_+ \esm^H Z)  \ra G_+ \esm^H   (Y \sma Z)
\end{equation*}
Clearly the composites $f'F'$ and $F'f'$ are the respective identities.
\end{proof}
\begin{proposition}  \label{lontinput}
 Let $n = n_1+ \cdots n_k$.
 Let  $Z \in \cR(X, \Sigma_{n_1}{\times}\cdots{\times}\Sigma_{n_k}, \{{\rm all}\})$.  
There is a natural isomorphism of $\Sigma_n$-spaces
\begin{multline*}
\Iso(\fset{n}, \fset{n_1} \cup \cdots \cup \fset{n_k})_+ \esm^{\Sigma_{n_1}{\times}\cdots{\times}\Sigma_{n_k}}
     \bigl(   (S{\dop}\stackrel{\text{$n_1$ terms}}{\cdots}{\dop}S)\sma \cdots
      \sma (S{\dop}\stackrel{\text{$n_k$ terms}}{\cdots}{\dop}S) \sma Z \bigr)
 \\
\stackrel{\iso}{\lra}
 (S{\dop}\stackrel{\text{$n$ terms}}{\cdots}{\dop}S) \sma
\bigl(\Iso(\fset{n}, \fset{n_1} \cup \cdots \cup \fset{n_k})_+ \esm^{\Sigma_{n_1}{\times}\cdots{\times}\Sigma_{n_k}}  Z \bigr).
\end{multline*}
\end{proposition}
\begin{proof}
 Apply lemma \ref{downthenup}, and observe that the operation $\dop$ is defined in terms of $\sma$, which is coherently associative.  
Collect all parentheses in expressions
$(S{\dop}\stackrel{\text{$n_1$ terms}}{\cdots}{\dop}S)\sma \cdots \sma (S{\dop}\stackrel{\text{$n_k$ terms}}{\cdots}{\dop}S)$ 
to the left. 
Note that we need only the map $f' \colon G_+ \esm^H   (Y \sma Z)  \lra Y \sma (  G_+ \esm^H Z )$ from the lemma. 
We do require the choice of an identification of $\Iso(\fset{n}, \fset{n_1} \cup \cdots \cup \fset{n_k})$ 
with $\Sigma_n$ to make sense of $f'$. This amounts to identifying 
$\fset{n_1} \cup \cdots \cup \fset{n_k}$ with $\{1 , \ldots, n_1, n_1{+}1, \ldots, n_1{+}n_2, \ldots, n_1+\ldots +n_k\}$.
\end{proof}
We use the Thomason homotopy colimit construction on functors defined on the category $F$ to pass to the limit with suspensions.
To treat suspension by $S^1$ on $\sub_k w G_{\bullet}\cR_f(X)$
define an op-lax functor $\Phi_1 \colon F \ra Cat^{\Delta^{\rm op}}$ as follows.
\begin{align*}
  \Phi_1(x) &= \sub_k w G_{\bullet}\cR_f(X), \quad \text{for an object $x \in F$,}
\\
  \Phi_1(\sigma)   &= \id, \quad \text{for an isomorphism $\sigma \colon x \ra x$,}
\\
  \Phi_1(i \colon y \ra x) &\quad \text{is induced by suspension by $x \backslash i(y)$ factors $S^1$.}
\end{align*}
Interpreting the smash product with an empty number of factors as $S^0$, the definitions coincide on isomorphisms.
For $x \stackrel{i}{\la} y \stackrel{j}{\la} z$ we need to  produce the natural transformation 
$\Phi_1(ij) \Rightarrow \Phi_1(i)\com \Phi_1(j)$. 
 On $(Y,r,s)$ the value of $\Phi_1(j)$  is $\bigl( (S^1)^{y \backslash j(z)}\esm Y, r', s'\bigr)$ and the value of $\Phi_1(i)$ 
applied to this is $\bigl((S^1)^{x \backslash i(y)} \esm ((S^1)^{y \backslash j(z)}\esm Y), r'', s''\bigr)$.  Since $i$ is injective,
the set $y \backslash j(z)$ is identified with $i(y \backslash j(z))$.  
Since $x \backslash ij(z) = x \backslash i(y) \cup i(y \backslash j(z))$ we use associativity isomorphisms of the $\esm$-action to write 
$\Phi_1(i \com j) \stackrel{\iso}{\Rightarrow} \Phi_1(i)\com \Phi_1(j)$. 
The coherence properties of the action imply commutativity of the necessary diagrams \cite[Definition 3.1.1, p.99]{Thomasonhocolim}.

In a similar way we treat $\dop_k S^1 \sma -$ on $ wG^k_{\bullet}\cR_f(X, \Sigma_n \{{\rm all}\})$, defining an op-lax functor
$\Phi_2 \colon  F \ra Cat^{\Delta^{\rm op}}$.
\begin{align*}
  \Phi_2(x) &=  w G^k_{\bullet}\cR_f(X, \Sigma_n \{{\rm all}\}), \quad \text{for an object $x \in F$,}
\\
  \Phi_2(\sigma)   &= \id, \quad \text{for an isomorphism $\sigma \colon x \ra x$,}
\\
  \Phi_2(i \colon y \ra x) &\quad \text{is induced by suspension by $x \backslash i(y)$ factors $\dop_k S^1$.}
\end{align*}
The natural transformation $\Phi_2(i \com j) \stackrel{\iso}{\Rightarrow} \Phi_2(i)\com \Phi_2(j)$ is treated in the same manner.

The results are two categories
\begin{equation*}
  \hocolim_{S^1 \sma - } \sub_k w G_{\bullet}\cR_f(X) \mathrel{\mathop :}= F \gro \Phi_1 \quad \text{and} \quad 
   \hocolim_{\dop_k S^1 \sma - }  w G^k_{\bullet}\cR_f(X) \mathrel{\mathop :}= F \gro \Phi_2.
\end{equation*}
\begin{remark}
  There are a number of constructions in \cite{Thomasonhocolim} that may justifiably termed homotopy colimits.  
This particular construction $F \gro \Phi_i$ is essential, but we use the $\hocolim$ notation to provide context
for the reader.
\end{remark}
Now we explain how proposition  \ref{lont} promotes 
$\omega^k \colon \sub_k w G_{\bullet}\cR_f(X) \ra  w G^k_{\bullet}\cR_f(X)$
to a left-op natural transformation (lont) $\epsilon \colon \Phi_1 \Rightarrow \Phi_2$.
First, we need for an object $x$ of $F$, a functor $\epsilon(x) \colon \Phi_1(x) \ra \Phi_2(x)$. 
This is just $\omega^k$. Then we need for each arrow 
$i \colon y \ra x$ in $F$ a natural tranformation 
$\epsilon(i) \colon \epsilon(x) \com \Phi_1(i) \Rightarrow \Phi_2(i) \com \epsilon(y)$.
For any morphism $i$ such that $x\backslash i(y)$ has cardinality 1, we obtain $\epsilon(i)$ by inverting 
the isomorphism of functors provided by proposition \ref{lont}. 
For the general case, one just goes back to the proof and replaces the symbol $1$ by $x\backslash i(y)$ everywhere it occurs.
The coherence results of section \ref{General} guarantee that the necessary diagrams commute, so $\epsilon$ is a lont. 
By \cite[Definition 3.1.4, p.101]{Thomasonhocolim} $\epsilon$ induces a functor
\begin{equation*}
  F \gro \epsilon \colon F \gro \Phi_1 \ra F \gro \Phi_2.
\end{equation*}
We have now proved the following result. 
\begin{theorem} \label{stableoperations}
  The operations $\omega^k$ pass through the Thomason homotopy colimit construction to deliver operations
  \begin{equation*}
    F \gro \epsilon \mathrel{\mathop :}= \omega^k \colon \hocolim_{S^1 \sma - } \sub_k w G_{\bullet}\cR_f(X) 
                  \ra \hocolim_{\dop_k S^1 \sma - }  w G^k_{\bullet}\cR_f(X) \qed
  \end{equation*}
\end{theorem}
\begin{proof}[Proof of theorem \ref{infiniteloop}]
  The main result of section \ref{EinftyStructure} is that 
  \begin{equation*}
    \Loops\abs{hN_{\Gamma}\cR_f^n(X)} \ra \{1\} \times \prod_{n \geq 1} A_{\Sigma_n, {\rm all}}(X) 
  \end{equation*}
is an infinite loop  map, and this section shows these maps are compatible with suspension. 
Likewise for the equivalence $\Loops\abs{hN_{\Gamma}\cR_f^n(X)} \ra \Loops\abs{w S_{\bullet}\cR_f^n(X)}$.
The maps obtained by passing to the limit over suspension remain infinite loop maps and
we know $\Loops \colim \abs{w S_{\bullet}\cR_f^n(X)} \simeq \Loops\abs{w S_{\bullet}\cR_f(X)} = A(X)$.
\end{proof}
\section{Projecting to the free part} \label{Splitting}
As stated in theorem \ref{infiniteloop} the  constructions of
\cite{GSAthyops} as modified in section \ref{Suspension}
deliver a total operation
\begin{equation*}
  \omega = \prod \omega^n \colon A(X) \ra \prod_{n \geq 1} A_{\Sigma_n, \{{\rm all}\}}(X),
\end{equation*}
where
$ A_{\Sigma_n, \{{\rm all}\}}(X) = \Loops|hS_{\bullet}\cR_{hf}(X, \Sigma_n, \{{\rm all}\})|$.
We  examine the target of this map, and introduce the Weyl group notation 
 $W_{\Sigma_n}H = N_{\Sigma_n}H/H$, 
where $H$ is a subgroup of the permutation group $\Sigma_n$ and $N_{\Sigma_n}H$ is
the normalizer in $\Sigma_n$ of $H$.
\begin{theorem}  \label{splitting}
Let $X$ be a space on which symmetric groups $\Sigma_n$ act trivially. 
  For each $n$ there is a homotopy equivalence 
  \begin{equation*}
 h_n  \colon  A_{\Sigma_n, \{{\rm all}\}}(X) \lra \prod_{H \in \{{\rm all}\}} A\bigl(X{\times}BW_{\Sigma_n}H)\bigr)
  \end{equation*}
of infinite loop spaces. Here
$A\bigl(X{\times}B(W_{\Sigma_n}H)\bigr)= \Loops|hS_{\bullet}\cR_f(X, W_{\Sigma_n}H,\{e\})|$ is 
the $K$-theory of the category of retractive $W_{\Sigma_n}H$-spaces relative to $X$ with the
action being free outside of~$X$.
\end{theorem}
\begin{proof}
  The argument is largely formal, based on some well-known facts. 
Let $\cF$  be the set of conjugacy classes $(H_i)$ of
subgroups of $\Sigma_n$.   The set is a finite set partially ordered in the usual way:
$(H_i) \preceq (H_j)$ if some conjugate of $H_i$ is contained in $H_j$.    The partial ordering may
be extended to a linear ordering, or enumeration
$\{(H_0), (H_1), \ldots, (H_N)\}$, so that $(H_i) \prec (H_j)$ implies $i < j$. 
Observe that $(H_0) = \{e\}$, we may take $(H_1)$ as the class of tranpositions, and $(H_N) = \Sigma_n$. 

 For any $\Sigma_n$-space $Z$ we may define
 \begin{equation*}
   \cF_{\succ (H)}Z = \colim_{(K) \succ (H)} Z^{(K)},
 \end{equation*}
essentially the union of the fixed point sets of the conjugates of all 
the subgroups properly containing a conjugate of $H$.
$ \cF_{\succ (H)}Z $ is by  definition a $\Sigma_n$-invariant subspace of $Z$.  
If $(H_i) \prec (H_{i+1})$ in the enumeration then we may compute $ \cF_{\succ (H_{i+1})} \bigl(  \cF_{\succ (H_i}Z \bigr)$,
essentially the fixed points of conjugates  of $H_{i+1}$ inside the fixed points of $H_i$.  
On the complement $ \cF_{\succ (H_i)}Z \backslash\Bigl( \cF_{\succ (H_{i+1}} \bigl(  \cF_{\succ (H_i}Z \bigr)\Bigr)$
the group $\Sigma_n$ acts and the Weyl group $W_{\Sigma_n}H_i = N_{\Sigma_n}H_i / H_i$ acts freely.

Inductively define exact functors 
\begin{equation*}
  S_i, Q_j \colon \cR_f(X,  \Sigma_n , \{{\rm all}\}) \lra \cR_f(X, \Sigma_n , \{{\rm all}\}), \quad \text{$-1 \leq i \leq N$, $0 \leq j \leq N$}
\end{equation*}
by $S_{-1}$ is the identity functor, and for $i \geq 0$ put 
$S_i(Y) = \cF_{\succ (H_i)}(S_{i-1}(Y))$.  Then the functors 
$Q_j$ are defined by the natural cofibration sequences
\begin{equation*}
\xy \UCMT \xymatrix{
  S_j(Y)  \ar@{ >->}[r] & S_{j-1}(Y) \ar@{ ->>}[r] & Q_j(Y), }
\endxy \quad 0 \leq j \leq N.
\end{equation*}
For us, the important case will be $S_0$: Since $H_0 = \{e\}$, 
$S_0(Y)$ is going to be the union of  the fixed point sets of all the non-identity
subgroups of $G$.  Then the quotient $Q_0(Y)$ can be thought of as extracting 
the part of $Y$ on which $G$ acts freely.

Let $i_k \colon \cR_f(X,\Sigma_n, \{H_k\}) \ra \cR_f(X, \Sigma_n, \{{\rm all}\})$ 
be the inclusion. Since $Q_k(Y)$ actually lies in $\cR_f(X,\Sigma_n, \{H_j\}) $,
we may formally write $Q_k = i_k\com \overline{Q}_k$ where 
$\overline{Q}_k \colon \cR_f(X,  \Sigma_n , \{{\rm all}\}) \ra \cR_f(X,\Sigma_n, \{H_k\})$ is a retraction.
We want to make an inductive application of the additivity theorem for the $G_{\bullet}$ construction,
but this requires that the input be pseudo-additive.  
Passing to prespectra $\Sigma^{\infty}\cR_f(X)$, 
\cite{GSVWKthy} there results a splitting
\begin{equation*}
\hocolim  wG_{\bullet}\cR_f(X, \Sigma_n, \{{\rm all}\}) \ra \prod_{H \in \{{\rm all}\}} \hocolim wG_{\bullet}\cR_f(X, \Sigma_n, \{H\})
\end{equation*}
induced by the functors $\overline{Q}_k$ for $0 \leq k \leq N$.
Recalling that $W_{\Sigma_n}H = N_{\Sigma_n}H/H$ is the Weyl group of $H$,  consider the exact functor
\begin{equation*}
  \cR_f(X, \Sigma_n, \{H\}) \lra \cR_f(X, W_{\Sigma_n}H, \{e\}) \quad \text{taking $Y \mapsto Y^H$.}
\end{equation*}
The induction construction $Z \mapsto Z \times^{W_{\Sigma_n}H} \Sigma_n$ provides an exact functor going 
the other way and the composites in either order are equivalent to the identities. 
Putting these equivalences together and specializing the notation establishes a chain of homotopy equivalences
\begin{multline*}
   \hocolim wG_{\bullet}\cR_f(X, \Sigma_n, \{{\rm all}\}) \ra \prod_{H \in \{{\rm all}\}}\hocolim wG_{\bullet}\cR_f(X, \Sigma_n, \{H\})\\
                    \ra   \prod_{H \in \{{\rm all}\}}\hocolim wG_{\bullet}\cR_f(X, W_{\Sigma_n}H, \{e\})
\end{multline*}
\end{proof}
This completes the proof of theorem \ref{mainthm1}; to explain theorem \ref{gencalcintro} is the object of the next
two sections.
We are focusing on the composition
\begin{multline*}
 \theta^n \colon  A(X) \stackrel{\omega^n}{\lra} 
  A_{\Sigma_n, \{{\rm all}\}}(X) \stackrel{h_n}{\lra}
  \prod_{H \in \{{\rm all}\}} \Loops|hS_{\bullet}\cR_f(X, N_{\Sigma_n}H/H,\{e\})| 
\\
  \stackrel{p_e}{\lra}
  \Loops|hS_{\bullet}\cR_hf(X, \Sigma_n, \{e\}).
\end{multline*}
In section \ref{Transfer} we justify the interpretation
$\Loops|hS_{\bullet}\cR_hf(X, \Sigma_n, \{e\}) = A(X{\times}B\Sigma_n)$. Then we want to understand what happens when we follow 
this composition by the transfer
$\phi_n \colon A(X{\times}B\Sigma_n) \ra A(X{\times}E\Sigma_n) \simeq A(X)$.
We start by introducing notation for the composition
\begin{equation*}
  \cR_f(X) \lra \sub_n G_{\bullet} \cR_f(X) 
\\
  \stackrel{\omega^n}{\lra} G^n_{\bullet} \cR_f(X, \Sigma_n, \{{\rm all}\}) \stackrel{Q_0=S_{-1}/S_0}{\lra} G^n_{\bullet} \cR_f(X, \Sigma_n, \{e\}).
\end{equation*}
On $(Y,r,s) \in \cR_f(X)$,  the composition of the first two maps in the chain is $\tilde{\alpha}_n(Y)$
 in the notation of example \ref{importantomegavalue},
so we want to evaluate the functor $Q_0 = S_{-1}/S_0 \com \tilde{\alpha}_n$ on the object $( Y, r,s)$.
By the terminology used in the proof of theorem \ref{splitting}, 
$S_{-1}$ is the identity and $S_0$ is the union of subobjects that are fixed by some non-identity subgroup of $\Sigma_n$. 
The interpretation and transfer issues are taken up in the next section \ref{Transfer}; 
to prepare for the analysis of $\phi_n\com \theta_n$ in  section \ref{Formulas} we introduce some notation.

 The definitions of the Segal operations in 
\cite{Waldhausen82} use certain subfunctors $P^n_j$ of the smash power functor $P^n$ on pointed
sets.  We extend the considerations to define certain subfunctors of  $\esm$ and $\sma$ powers.
For $(Y,r,s) \in \cR(X)$, the set $Y^{\esm n}$ is a quotient of the cartesian product $Y^n$.
In a fixed simplicial dimension,  we view this as
the set of functions $y \colon \fset{n} \ra Y$.  The pushout construction identifies any such
function $y$ with at least one value $y_i$ in $X$ with the composite function $r\com y$. Thus, to
represent points of $Y^{\esm n}$ in a given dimension, we just need to look at functions all of whose values are
in $Y-X$ and functions all of whose values are in $X$. 
For $0 \leq j \leq n$ we define 
$\widetilde{P}^n_jY$ to be the subset of functions $y \colon i \mapsto y_i$ such that
 the cardinality of $y^{-1}(Y{-}X)$ is less than or equal to $j$, if the image of $y$ is
contained in $(Y{-}X)$. 
Said another way,
$\widetilde{P}^n_jY$ is the set of $n$-tuples where at most $j$ distinct elements of $Y-X$ are
involved.  For example,
$\widetilde{P}^n_0Y = X^n$ and $\widetilde{P}^n_1Y$ is the union of $X^n$ with 
 the diagonal of $(Y{-}X)^n$. 
Most important for us,  the subset $\widetilde{P}^n_{n-1}Y$ consists of all $n$-tuples involving no more
than $n{-}1$ distinct elements of $Y$, so that if no member of $(y_1, \ldots , y_n)$ is in $X$,
then there are at least two distinct indices $i$, $j$ with $y_i = y_j$.

When   $X$ is a connected abelian group, then we can push out along the 
iterated multiplication $X^n \ra X$, obtaining functors
$P^n_jY$ relative to $X$. In particular, 
$P^n_{n-1}Y$ is the subset of $P^nY$ consisting of points fixed by some non-trivial subgroup of $\Sigma_n$,
so not all members of an $n$-tuple can be distinct. 
Thus  $P^n_{n-1}Y = S_0\tilde{\alpha}_n(Y)$.  In terms of functions $y \colon \fset{n} \ra Y$, 
$P^n_{n-1}Y$ is the set of functions where the cardinality of $y^{-1}(Y{-}X)$ is at most $n{-}1$.
\begin{definition}  \label{thetadef}
Define
  \begin{equation*}
   \xy \UCMT \xymatrix{
   \widetilde{P}^n_{n-1}Y \ar@{ >->}[r] \ar[d]_{\esm^n  r} & \widetilde{P^n}Y \ar[d]   &
                                   & P^n_{n-1}Y \ar@{ >->}[r] \ar[d]_{r} & P^nY \ar[d]
\\
   X^n  \ar@{ >->}[r]  &   \widetilde{\theta}^nY  &
                      & X  \ar@{ >->}[r]  &   \theta^nY         }
\endxy    
  \end{equation*}
 \end{definition}
Letting $j^n \colon \cR_f(X, \Sigma_n, \{e\}) \ra G^n_{\bullet}\cR_f(X, \Sigma_n, \{e\})$ be the iterated stabilization, 
we combine the preceding observations with the definitions to immediately obtain the following proposition.  
\begin{proposition}  \label{identification}
As functors from $\cR_f(X)$ to $G^n_{\bullet}\cR_f(X, \Sigma_n, \{e\})$, 
 $Q_0 \com \tilde{\alpha}_n = j^n \com \theta^n$. \qed
\end{proposition} 
\section{Transfer constructions} \label{Transfer}
 Our immediate goals are to interpret
$\Loops| hS_{\bullet}\cR_f(X^n, \Sigma_n, \{e\}) |$  and 
$\Loops| hS_{\bullet}\cR_f(X, \Sigma_n, \{e\}) |$
in terms of the algebraic $K$-theory of topological spaces.
In this section families of subgroups play no role, so we revert to 
the less ornate notation $\Loops| hS_{\bullet}\cR_f(X,G) |$ for $\Loops| hS_{\bullet}\cR_f(X,G, \{e\}) |$,
the algebraic $K$-theory of $G$-spaces retracting to $X$, finite relative to $X$ and with $G$ acting freely
outside $X$.

There are  two steps to these goals and each step uses arguments based on 
\cite{Waldhausen85}.
We let $G$ be a finite group and let $Z$ be a $G$-space. Let $EG$ be
the canonical contractible free left $G$-space. We prefer the model
$EG_n = G^{n+1}$ with the $G$-action given by multiplication on the left in 
each factor, face maps defined by projecting away from a coordinate, and 
degeneracies defined by repeating a coordinate.  
An isomorphism of the quotient space $* \times^G EG \iso BG$ is induced by
$(g_0, \ldots, g_{i-1}, g_i, \ldots g_n) \mapsto (g_0^{-1}g_1, \ldots , g_{i-1}^{-1}g_i, \ldots ,g_{n-1}^{-1}g_n)$.

First, Lemma 2.1.3 \cite[p.381]{Waldhausen85} applies to yield the following result.  
\begin{lemma}\label{divideout}
  There is an equivalence of categories $\cR(EG {\times}^G Z) \sim \cR(EG{\times}Z, G)$. \qed
\end{lemma}
For reference, pullback along the projection
\begin{equation*}
 EG{\times}Z \ra EG {\times}^G Z
\end{equation*}
defines a functor 
$\cR(EG {\times}^G Z ) \ra \cR(EG {\times} Z, G)$; the orbit map defines a functor
in the opposite direction.  The composites in either order are isomorphic to the respective 
identity functors.   Moreover, these functors preserve weak equivalences  and 
homotopy finite objects.  

Next, we want the following lemma, which permits us to replace the $G$-action
on $Z$ with a free $G$-action on a homotopy equivalent space.
\begin{lemma}\label{makefree}
  The projection $EG \times Z  \ra Z$ induces a homotopy equivalence
  \begin{equation*}
    hS_{\bullet}\cR_{hf}(EG {\times}Z, G) \lra hS_{\bullet}\cR_{hf}(Z, G).
  \end{equation*}
\end{lemma}
\begin{proof}
  The argument here is similar to that given to prove Proposition 2.1.4
\cite[p.382]{Waldhausen85}. 
In detail, 
let $(Y', r',s') \in \cR_{hf}(EG {\times} Z, G)$.  Completing the diagram
\begin{equation*}
  \xy \UCMT \xymatrix{
Y' & \ar@{ >->}[l]_(0.6){s'}  EG {\times} Z \ar[r]^(0.65){p_2} &  Z}
\endxy
\end{equation*}
to a pushout defines an exact functor 
$\cR_{hf}(EG{\times}Z, G) \ra \cR_{hf}(Z, G)$. 
Certainly, homotopy finite objects are carried to homotopy-finite objects,
and, incidentally, finite  objects are carried to finite objects.
Also, weak equivalences are mapped to weak  equivalences. 

Taking the product with $EG$ provides an exact functor 
$\cR_{hf}(Z,G) \ra \cR_{hf}(EG{\times}Z, G)$.  In this case, when $G$ is nontrivial, finite 
objects are carried to homotopy finite objects, since $EG$ is contractible. 

For $(Y,r,s)$ in $\cR_{hf}(Z,G)$, taking the induced map of pushouts in the
following diagram provides a natural transformation from
the composite endofunctor on $\cR_{hf}(Z, G)$ to the identity functor.  This natural transformation is a 
weak equivalence.  
\begin{equation*}
  \xy \UCMT \xymatrix{
EG{\times}Y \ar[d]_{p_2} & \ar@{ >->}[l] EG{\times}Z \ar[r] \ar[d]_{p_2} & Z \ar[d]_{\id}
\\
Y     &     \ar@{ >->}[l]_{s} Z \ar[r]^{\id} & Z }
\endxy
\end{equation*}
For $(Y',r',s')$ in $\cR_{hf}(EG{\times}Z,G)$, taking the induced map of pushouts in the next diagram 
provides a natural transformation from the identity functor on $\cR_{hf}(EG{\times}Z, G)$ to 
the other composite endofunctor. Again, this natural transformation is a weak equivalence. 
\begin{equation*}
  \xy \UCMT \xymatrix{
Y' \ar[d]_{ p_1r' \times \id} 
             & \ar@{ >->}[l]_{s'}  EG {\times} Z \ar[r]^{\id} \ar[d]_{ \Delta \times \id } 
                                             & EG{\times} Z \ar[d]_{\id}
\\
EG{\times}Y'
             &  \ar@{ >->}[l]_(0.6){\id \times s'}  EG {\times} EG {\times} Z \ar[r]^(0.6){p_1\times p_3}
                                               &  EG{\times}Z }
\endxy
\end{equation*}
By Proposition 1.3.1 \cite[p.330]{Waldhausen85}
$ hS_{\bullet}\cR_{hf}(Z {\times} EG, G) \ra hS_{\bullet}\cR_{hf}(Z, G)$ is a homotopy equivalence.
\end{proof}
Substituting for $G$ the symmetric group $\Sigma_n$,
we combine lemmas \ref{divideout} and \ref{makefree} 
to record useful alternative models for $A(B\Sigma_n{\times}X)$ and $A(D_nX)$.
The first part covers a remark made following definition \ref{equivwcofandwe}.
\begin{lemma}\label{interpretations} 
Let $X$ have the trivial $\Sigma_n$-action, so that $B\Sigma_n {\times} X $ is the quotient of 
$E\Sigma_n {\times} X$ by the action of $\Sigma_n$.
  There are homotopy equivalences
  \begin{equation} \label{equivalencechain}
    hS_{\bullet}\cR_{hf}(B\Sigma_n {\times} X ) \simeq hS_{\bullet}\cR_{hf}(E\Sigma_n {\times} X, \Sigma_n) 
                      \simeq  hS_{\bullet}\cR_{hf}(X, \Sigma_n).
  \end{equation}
Thus, the space  $\Loops| hS_{\bullet}\cR_f(X, \Sigma_n) |$ is homotopy equivalent to $A(B\Sigma_n {\times} X )$. 

Similarly, let $X^n$ have the permutation action of $\Sigma_n$,
 and let $D_nX = E\Sigma_n \times^{\Sigma_n} X^n $ be the quotient of $E\Sigma_n {\times} X^n $ by the diagonal action of $\Sigma_n$.
  There are homotopy equivalences
  \begin{equation*}
      hS_{\bullet}\cR_{hf}(D_nX ) \simeq hS_{\bullet}\cR_{hf}(E\Sigma_n {\times} X^n, \Sigma_n) \simeq hS_{\bullet}\cR_{hf}(X^n, \Sigma_n).
  \end{equation*}
Thus, the space  $\Loops| hS_{\bullet}\cR_{hf}(X^n, \Sigma_n) |$ is homotopy equivalent to $A(D_nX)$. \qed
 \end{lemma}

We recall here basic facts about the transfer in the algebraic $K$-theory of spaces
adapted to our context.
We are actually interested in two cases of transfer operations.  For the first case 
the transfer operations are associated with finite subgroups of
the symmetric groups $\Sigma_n$. In the second case the
operations are associated with (injective)  homomorphisms of simplicial abelian groups 
$\widetilde{X} \ra X$, where the fiber is homotopy-finite.   

In terms of the description $A(X) = \Loops|hS_{\bullet}\cR_f(X)|$, 
we have the following direct transfer construction. A fiber bundle projection $p \colon E \ra B$ with finite
fiber, resp., homotopy-finite fiber, induces by pullback a functor
$\cR_f(B) \ra \cR_f(E)$, resp., $\cR_f(B) \ra \cR_{hf}(E)$.  We then obtain
a transfer morphism $p^* \colon A(B) \ra A(E)$. In terms of equivariant models for
algebraic $K$-theory, there are other descriptions of the transfer, as given below.
We need to relate the various descriptions.  

Eventually we need the transfer operations
 $A( B\Sigma_n {\times} X) \ra A( BH {\times} X)$,
where $H$ is a subgroup of $\Sigma_n$.
Our working definition is 
$A( B\Sigma_n {\times} X) = \Loops | hS_{\bullet}\cR_{hf}(X, \Sigma_n)|$
but, in view of the equivalences \eqref{equivalencechain},
we have to compare three definitions in each context. 

To this end, let $G$ be a discrete group, $H$ be a subgroup of finite index, and let $Z$ be a trivial $G$-space.
Observe that $EG{\times}Z$ is the total space of a principal $G$-bundle with base 
$EG \times^G Z$. To make this transparent, and for use in the study of diagram \eqref{maintransferdiagram}, 
 we replace the notation 
$EG \times^G Z$ by $* \times^G (EG{\times}Z)$.
To explain the connection, $* \times^G (EG{\times}Z)$ is the orbit space of $EG{\times}Z$
under the diagonal left $G$-action, thought of as the balanced product of $EG{\times}Z$
with the trivial right $G$-space $*$.  We can turn the left action of $G$ on $EG$ into a right action
by setting $e \cdot_r g = g^{-1} \cdot_l e $.  Then left $G$-orbits in $EG{\times}Z$ are seen to correspond to equivalence
classes in $EG{\times}Z$ under the equivalence relation generated by $(e\cdot_rg ,z) \sim (e, gz)$. 
The associated quotient space is usually denoted $EG \times^G Z$. 

We consider the following diagram where the vertical arrows represent transfer constructions. 
\begin{equation} \label{maintransferdiagram}
  \xy \UCMT \xymatrix{
 \cR(Z, H) \ar[r]  &      \cR(EG \times Z, H)      &    \ar[l]  \cR( EG \times^H Z)
\\
 \cR(Z, G)  \ar[u]_{p_1^*}  \ar[r] &      \cR(EG \times Z, G) \ar[u]_{p_2^*}  &      \ar[l] \cR(EG \times^G Z) \ar[u]_{p_3^*}
} \endxy
\end{equation}
The forgetful functor 
$p_1^* \colon \cR(Z, G) \ra \cR(Z, H)$ just restricts the action to the subgroup $H$.
This provides the simplest path to  
$p_1^* \colon A(BG{\times}Z) \ra A(BH{\times}Z)$, 
using the basic model $A(BG {\times} Z) = \Loops|hS_{\bullet}(\cR_{hf}(Z, G)|$. 
In the middle, the functor 
$p_2^* \colon \cR(EG {\times} Z, G) \ra \cR(EG {\times} Z,H)$ 
is also a forgetful functor.
At the right, the functor $p_3^* \colon \cR( EG \times^G Z) \ra \cR(EG \times^H Z)$ is given by 
a pullback construction, explained in detail below.

To reach the categories in the middle column from those in the left column we compute products with
$EG$. Along the top, the fact that $EG$ is a non-standard contractible $H$-space is an insignificant 
detail.  Comparing with $p_1^*$ on the  left, the transfer $p_2^*$ in the middle column
 is also obtained by restricting the action of $G$ to $H$.  
Thus, the lefthand square in diagram \eqref{maintransferdiagram}  obviously commutes.

Before we compare $p_3^*$ with $p_2^*$, we discuss $p_3^*$, the rightmost column in diagram \eqref{maintransferdiagram}, in detail.
In order to manipulate pullback squares efficiently we replace the notation 
$EG \times^G Z$ by $* \times^G (EG {\times} Z)$ as discussed before lemma \ref{divideout}.
Suppose $H$ is a subgroup of the group $G$, and let $EG$ be
the standard model for a contractible $G$-space on which $G$ acts freely from the  
right.  The
space  $EG$  plays a similar role relative to the subgroup $H$.  In order to 
compare situations, we take the standard model $X =* \times^G (EG {\times} Z) $ and
a modified model $\widetilde{X} =* \times^H (EG {\times} Z) $. 
In this situation we have the basic pullback square 
\begin{equation} \label{basicpullback}
  \xy \UCMT \xymatrix{
  ({*} \times^H G) \times (EG{\times}Z)  \ar[d]_{p_2} \ar[r]   & ({*} \times^H EG{\times}Z) =\widetilde{X}  \ar[d]
\\
  EG{\times}Z         \ar[r]             &     {*} \times^G (EG{\times}Z) = X }
\endxy
\end{equation}
This displays the comparison map $\widetilde{X} \ra X$ of the chosen models as a fiber bundle,
with fiber $* \times^H G$.  One may identify 
$*  \times^H (EG{\times}Z) \iso  (* {\times}^H G)\times^G (EG{\times}Z)$
and then the righthand vertical arrow is isomorphic to the map
$(* {\times}^H G) \times^G (EG{\times}Z) \ra * \times^G (EG{\times}Z)$ induced by projecting
the coset space $* {\times}^H G$ to a point. 
This replacement also displays the upper horizontal map as the quotient projection
$(* {\times}^H G) \times (EG{\times}Z)  \ra  (* {\times}^H G) \times^G (Z{\times}EG)$.

The direct construction 
$p^* \colon \cR(X) \ra \cR(\widetilde{X})$
maps $(Y,r,s)$ to $(\widetilde{Y}, \tilde{r}, \tilde{s})$ derived from the following pullback square.
\begin{equation}\label{nonequivariant}
  \xy \UCMT \xymatrix{
\widetilde{Y} \ar[r]^(0.25){\tilde{r}} \ar[d]   &   \widetilde{X} =    {*} \times^H (EG{\times}Z) \ar[d]_p
\\
  Y   \ar[r]^(0.25)r   &          X = {*} \times^G (EG{\times}Z) }
\endxy
\end{equation}
Augmenting the righthand column of \eqref{nonequivariant} to the square of \eqref{basicpullback} shows that
$\widetilde{Y} \ra Y$ 
is a fiber bundle with fiber  $* {\times}^H G$.

Now we address commutativity of  the righthand square in diagram \eqref{maintransferdiagram}.  
To reach the categories in the middle column from those in the right column we also compute pullbacks.
Recalling lemma \ref{divideout}, the equivalence of categories
$\cR(EG {\times}^G Z) \simeq \cR(EG \times Z, G)$ \cite[Lemma 2.1.3]{Waldhausen85}
describes the functor moving left to the middle column. 
This functor assigns to a retractive space $(Y,r,s)$ over $EG {\times}^G Z$
 the retractive $G$-space $(Y',r',s')$ over $EG{\times}Z$ defined as the  pullback in the following diagram.
\begin{equation*}
  \xy \UCMT \xymatrix{
 Y'  \ar[r]^{r'} \ar[d]   &    EG \times Z \ar[d]
\\
 Y   \ar[r]^(0.30)r     &   X =   {*} \times^G (EG{\times}Z) }
\endxy
\end{equation*}
Then moving up to $\cR(EG{\times}Z, H)$ amounts to 
restricting the $G$-action in this pullback to $H$. 

On the other hand, to move from the lower right to the upper middle by going up and then to the left, 
compute first the pullback
\eqref{nonequivariant}
and then compute
\begin{equation*}
  \xy \UCMT \xymatrix{
 \widetilde{Y}'  \ar[r]^{\tilde{r}'} \ar[d]   &    EG {\times} Z \ar[d]
\\
 \widetilde{Y}   \ar[r]^(0.25){\tilde{r}}     &   \widetilde{X} =    {*} \times^H (EG{\times}Z).
} 
\endxy
\end{equation*}
The composition of the two functors may be displayed in the stacked diagram
\begin{equation*}
  \xy \UCMT \xymatrix{
\widetilde{Y}' \ar[r]^{\tilde{r}'} \ar[d]    &                EG \times Z \ar[d]
\\
\widetilde{Y} \ar[r]^(0.30){\tilde{r}} \ar[d]   &   \widetilde{X} =   {*} \times^H (EG{\times}Z) \ar[d]_p
\\
  Y   \ar[r]^(0.30)r   &          X =  {*}  \times^G (EG{\times}Z)}
\endxy
\end{equation*}
The end result is that $(\widetilde{Y}', \tilde{r}', \tilde{s}')$ is simply
the $G$-space $(Y', r',s')$ with the action restricted to $H$.  Therefore, the righthand square commutes.

\begin{lemma}[Compare \cite{Waldhausen82}, Lemma 1.3, p.399]\label{transferproperty}
Let $G$ be a finite group, $EG$ a universal $G$-bundle, 
$BG = * \times^{G} EG$ a classifying space, and let $Z$ be a space with a trivial $G$-action.
Then the composition
\begin{equation*}
\xy \UCMT \xymatrix@C=8ex{
A(Z) \ar[r]^(0.4){\text{inclusion}} & A(BG \times Z) \ar[r]^(0.43){\text{transfer}} 
                                                                 & A(EG \times Z) \heq A(Z)}
\endxy 
\end{equation*}
is given by multiplication by the order of $G$, in the sense of the additive $H$-space 
structure. \qed
\end{lemma}
\section{A fundamental cofibration sequence} \label{Formulas}
Waldhausen's main result is this proposition.
\begin{proposition}[Compare \cite{Waldhausen82}, Proposition 2.7, p.407]
  \label{modelrecursion}
 The composition of the operation $\theta^n \colon A(*) \ra A( B\Sigma_n \times *)$
with the transfer map $\phi_n \colon A( B\Sigma_n \times *) \ra A(*)$
is the same, up to weak homotopy, as the polynomial map on $A(*)$ given by the polynomial
\begin{equation*}
  p(x) = x(x-1)\cdots(x{-}n{+}1). \qed
\end{equation*}
\end{proposition}
The analogous result for the present situation with the one point space
replaced by a simplicial abelian group $X$ is more complicated to formulate and to work with.
To prepare for the analogue of Waldhausen's result, we develop the following constructions, taking up where we left off
with definition \ref{thetadef} and proposition \ref{identification}. 
We make use of the maps
\begin{equation*}
 {\delta}_{n-1}^{n,k} \colon X^{n-1} \ra X^n \quad  
                   \text{given by} \quad \delta_{n-1}^{n,k}(x_1, \ldots, x_{n-1}) = (x_1, \ldots, x_{n-1}, x_k)
\end{equation*}
and  the respective induced functors $\delta_{n-1*}^{n,k} \colon \cR_f(X^{n-1}) \ra \cR_f(X^n)$.
The pushout construction
\begin{equation*}
  \xy \UCMT \xymatrix{ 
  X^{n-1} \ar@{ >->}[r]^s \ar[d]_{{\delta}_{n-1}^{n,k}} &   Z \ar[d]_{i_{n-1}^{n,k}}
\\
  X^n    \ar@{ >->}[r]      &   {\delta}_{n-1*}^{n,k}Z }
\endxy
\end{equation*}
defines an exact functor ${\delta}_{n-1*}^{n,k} \colon \cR_f(X^{n-1}) \ra \cR_f(X^n)$.  
For a retractive space $(Z, r, s)$ over $X^{n-1}$ with retraction $r\colon Z\ra X^{n-1}$ written
in terms of components as $r=(r_1, \ldots, r_{n-1})$, the composition of the canonical map
$i_{n-1}^{n,k}$ followed by the retraction ${\delta}_{n-1*}^{n,k}r$ is given by the formula
$({\delta}_{n-1*}^{n,k}r) \com i_{n-1}^{n,k} (z) = {\delta}_{n-1}^{n,k} \com r(z) 
                   = \bigl( r_1(z), \ldots, r_k(z), \ldots, r_{n-1}(z),  r_k(z) \bigr).$
Note that in the special case
$Z=\widetilde{P}^{n-1}Y = \bigl(\esm\bigr)^{n-1} Y$, we have, for each $k$, $1 \leq k \leq n{-}1$, 
\begin{equation}
  \label{diagonalformula1}
 \bigl( \delta_{n-1*}^{n,k} (\widetilde{P}^{n-1} r)\bigr) \com  i_{n-1}^{n,k} (y_1, \ldots y_{n-1}) 
                                = \bigl( r(y_1), \ldots, r(y_k), \dots, r(y_{n-1}), r(y_k)  \bigr).
\end{equation}
Next we assemble these functors by gluing along the common space $X^n$, obtaining
\begin{equation*}
  \widetilde{\Delta}^n_{n-1} \colon \cR_f(X^{n-1}) \lra \cR_f(X^n)
\end{equation*}
given on objects by 
$\widetilde{\Delta}_{n-1}^n(Z) = \delta_{n-1 *}^{n,1}Z \cup_{X^n} \ldots \cup_{X^n} \delta_{n-1 *}^{n,n-1}Z$, 
which can be viewed as an iterated pushout or as the colimit of a diagram modeled on the
cone on $n{-}1$ points.  We also need to push this construction forward to $\cR_f(X)$ by $\mu_*$, 
the iterated multiplication, obtaining
\begin{equation*}
  \Delta^n_{n-1} = \mu_* \com \widetilde{\Delta}^n_{n-1} \colon \cR_f(X^{n-1}) \lra \cR_f(X)
\end{equation*}
 given on objects by 
${\Delta}_{n-1}^n(Z) = \mu_*(\delta_{n-1*}^{n,1}Z) \cup_{X} \ldots \cup_{X} \mu_*(\delta_{n-1*}^{n,n-1}Z)$.
If we start with $Z = Y \esm \stackrel{\text{$n{-}1$ factors}}{\cdots} \esm Y = \widetilde{P}^{n-1}Y$, 
then the formula for the retraction on the $k$th summand
$\mu_*(\delta_{n-1*}^{n,k} \widetilde{P}^{n-1}Y)$ is
\begin{equation}
\label{diagonalformula2}
  \bigl( \mu_*\delta_{n-1*}^{n,k} (\widetilde{P}^{n-1} r)\bigr) \com  i_{n-1}^{n,k} (y_1, \ldots y_{n-1}) 
                                = \mu\bigl( r(y_1), \ldots, r(y_k), \dots, r(y_{n-1}), r(y_k)  \bigr),
\end{equation}
where $\mu$ is the iterated multiplication.

We can now succinctly state our general results. Let
\begin{equation*}
  \tilde{\phi}_k \colon \cR_f(X^k, \Sigma_k, \{{\rm all}\} ) \ra \cR_f(X^k)
\quad \text{and} \quad
 \phi_k \colon \cR_f(X, \Sigma_k, \{{\rm all}\} ) \ra \cR_f(X)
\end{equation*}
be the functors that forget the group action.
\begin{proposition}[Compare \cite{Waldhausen82}, Proposition 2.7, page 407] \label{recursion1}
 There is a cofibration sequence of functors $\cR_f(X) \ra \cR_f(X^n)$
 \begin{equation} \label{unreducedcofibseq}
     \xy \UCMT \xymatrix{
\widetilde{\Delta}_{n-1}^n \tilde{\phi}_{n-1}\tilde{\theta}^{n-1} Y  \ar@{ >->}[r]
                    &   \tilde{\phi}_{n-1}\tilde{\theta}^{n-1} Y \esm \tilde{\theta}^1 Y  \ar@{->>}[r] 
                           & \tilde{\phi}_n\tilde{\theta}^n Y }
\endxy
 \end{equation}
In the case $X$ is a connected simplicial abelian group, we have the cofibration sequence
 \begin{equation}   \label{reducedcofibseq}
   \xy \UCMT \xymatrix{
\Delta_{n-1}^n\tilde{\phi}_{n-1}{\tilde{\theta}}^{n-1} Y  \ar@{ >->}[r]
                    &   \phi_{n-1}{\theta}^{n-1} Y \sma \theta^1 Y   \ar@{->>}[r] 
                           & \phi_n{\theta}^n Y }
\endxy
 \end{equation}
of functors $\cR_f(X) \ra \cR_f(X)$. 
\end{proposition}
\begin{remark}
  The second cofibration sequence is obtained by applying the exact functor induced by 
the iterated multiplication $\mu \colon X^n \ra X$ to the first sequence.  The result in the
middle term of the second sequence is open to interpretation.  The formulation chosen amounts to interpretation
of the factorization $\mu = \mu \com (\mu{\times}\id)$ along with the facts that $\mu_* \com \esm = \sma$ and
$\tilde{\theta}^1Y = \theta^1Y = Y$. 
\end{remark}
\begin{proof}
Following section \ref{Transfer}, we interpret  the transfer maps 
\begin{equation*}
  \phi_n \colon A(D_nX) \ra A(X^n) 
\quad \text{and} \quad
\phi_n \colon A(X{\times}B\Sigma_n) \ra A(X)
\end{equation*}
 as induced by the forgetful functors
 \begin{equation*}
\cR_f(X^n, \Sigma_n , \{{\rm all}\}) \ra \cR_f(X^n, \{e\}) \quad \text{and} \quad
\cR_f(X, \Sigma_n  \{{\rm all}\}) \ra \cR_f(X, \{e\}),
 \end{equation*}
 respectively. This means we have to make  non-equivariant analyses of the functors
$\tilde{\theta}^n$ and $\theta^n$, respectively.

To obtain the surjections, we consider the following diagram
\begin{equation}\label{diagram1recursion1}
  \xy \UCMT \xymatrix@C=11ex{
X \ar[d]
   & \ar[l]_{\mu} X^{n-1}{\times}X \ar[d]
           & \ar[l]_{(( \esm )^{n-1} r ) \esm r} \widetilde{P}^{n-1}_{n-2}Y \esm Y \ar@{ >->}[r] \ar[d]
                             & \widetilde{P}^{n-1}Y \esm Y \ar[d]^{\iso}
\\
X 
   & \ar[l]_{\mu} X^n 
           & \ar[l]_{r^n} \widetilde{P}^{n}_{n-1}Y \ar@{ >->}[r] 
                             & \widetilde{P}^{n}Y }
\endxy
\end{equation}
Clearly, $\widetilde{P}^{n-1}_{n-2} \esm Y$ maps into $\widetilde{P}^{n}_{n-1}$,
because, if there are two indices $i$, $j$ with $1 \leq i, j \leq n{-}1$ and $i \neq j$ and with $y_i = y_j$,
then this still holds for $\bigl((y_1, \ldots, y_{n-1}), y \bigr)$ rebracketed as
$(y_1, \ldots, y_{n-1}, y)$. 
Taking the pushouts along the rows using the columns two, three, and four
produces a surjection 
\begin{equation*}
   \xy \UCMT \xymatrix{
   \phi_{n-1}\tilde{\theta}^{n-1}Y \esm \widetilde{\theta}^1Y \ar@{->>}[r] & \phi_n\widetilde{\theta}^n Y}
\endxy
\end{equation*}
in $\cR_f(X^n)$
 and  pushing out along the rows using columns one, three and four yields
 \begin{equation*}
    \xy \UCMT \xymatrix{
\phi_{n-1}{\theta}^{n-1} Y \sma \theta^1 Y \ar@{=}[r]  & \mu_*(\phi_{n-1}\theta^{n-1}Y \esm \theta^1Y) \ar@{->>}[r] & \phi_n\theta^nY,}
\endxy
 \end{equation*}
the surjection in $\cR_f(X)$.  Now we have to identify the ``kernels.''

Reviewing the remarks at the end of section \ref{Splitting}, 
$\widetilde{P}^{n-1}Y \esm Y=\widetilde{P}^nY$ is the space whose simplices
outside of $X^n$ are $n$-tuples of simplices from $Y{-}X$; 
$\widetilde{P}^{n-1}_{n-2}Y \esm Y$ is
the space whose simplices outside of $X^n$ are $n$-tuples $((y_1, \ldots, y_{n-1}),y)$
with the condition that there are at least two distinct indices $1 \leq i,j\leq n{-}1$
with $y_i=y_j$; and 
$\widetilde{P}^n_{n-1}Y $ is the space whose simplices outside of $X^n$ are
$n$-tuples $((y_1, \ldots, y_{n-1}, y_n)$
with the condition that there are at least two distinct indices $1 \leq i,j\leq n$
with $y_i=y_j$.  Then the simplices of 
$\widetilde{P}^n_{n-1}Y $ not in the image of $\widetilde{P}^{n-1}_{n-2}Y \esm Y$
are those $n$-tuples where the first $n{-}1$ are distinct but $y_n = y_k$ for
some $1 \leq k \leq n{-}1$. 

Using this observation we extend the diagram \ref{diagram1recursion1} by means of the following
constructions.  For $1 \leq k \leq n{-}1$, consider the diagrams
\begin{equation*}
  \xy \UCMT \xymatrix{
   X^n    \ar[d]  &    \ar[l]_{\delta_{n-1}^{n,k}}  X^{n-1} \ar[d]_{\delta_{n-1}^{n,k}} \ar[r] & \widetilde{P}^{n-1}Y \ar[d]^{\delta_{n-1}^{n,k}}
\\
   X^n &  \ar[l]  \widetilde{P}^{n-1}Y{\times}X \cup_{X^n} X^{n-1}{\times}Y \ar[r]   &    \widetilde{P}^{n-1}Y \times Y}
\endxy
\end{equation*}
where $\delta_{n-1}^{n,k} \colon X^{n-1} \ra X^n$ is given by
$\delta_{n-1}^{n,k}(x_1, \ldots , x_{n-1}) = (x_1, \ldots, x_{n-1}, x_k)$
and the other maps labeled $\delta_{n-1}^{n,k}$ are given by similar formulas.
For each $k$, taking the pushout of the first row extends $\widetilde{P}^{n-1}Y$ over $X^{n-1}$ to 
the space $\delta_{n-1*}^{n,k}\widetilde{P}^{n-1}Y$ over $X^n$; taking the pushout of the second row yields $\widetilde{P}^{n-1}Y \esm Y$.
Since the diagram commutes, we obtain a family of maps over $X^n$
\begin{equation*}
\delta_{n -1}^{n,k} \colon {\delta}_{n-1*}^{n,k}\tilde{\phi}_{n-1}\widetilde{P}^{n-1}Y \ra  \tilde{\phi}_{n-1}\widetilde{P}^{n-1} Y \esm Y
\end{equation*}
with $\delta_{n -1}^{n,k} ( y_1, \ldots, y_{n-1}) = (y_1, \ldots,y_k, \ldots, y_{n-1}, y_k)$.  

Now we are ready to augment diagram \ref{diagram1recursion1} after which we can compute the desired
cofibration sequence. 
Having established the notation 
\begin{equation*}
\widetilde{\Delta}_{n-1}^n \tilde{\phi}_{n-1}\widetilde{P}^{n-1}Y =  
  \delta_{n-1 *}^{n,1}\tilde{\phi}_{n-1}\widetilde{P}^{n-1}Y \cup_{X^n} 
            \ldots \cup_{X^n} \delta_{n-1 *}^{n,n-1}\tilde{\phi}_{n-1}\widetilde{P}^{n-1}Y
\end{equation*}
write 
$\Delta_{n-1}^n \colon \Delta_{n-1}^n \phi_{n-1}\widetilde{P}^{n-1}Y  \ra \phi_{n-1}\widetilde{P}^{n-1}Y \esm Y$
for the union of the maps $\delta_{n-1}^{n,k}$ just defined. 
Add this map above the upper right corner of the  diagram \eqref{diagram1recursion1} and fill out the following diagram.
\begin{equation}
  \label{diagram2recursion1}
  \xy \UCMT \xymatrix@C+=8ex{
X \ar[d]
     & \ar[l]_{\mu} X^n \ar[d]_{\iso} 
              & \ar[l] \Delta_{n-1}^n \tilde{\phi}_{n-1} \widetilde{P}^{n-1}_{n-2}Y  \ar@{ >->}[r]^{i'} \ar@{ >->}[d]_{\Delta_{n-1}^n}
                    &  \Delta_{n-1}^n \tilde{\phi}_{n-1}\widetilde{P}^{n-1}Y    \ar@{ >->}[d]_{\Delta_{n-1}^n}
\\
X \ar[d]
   & \ar[l]_{\mu} X^{n-1}{\times}X \ar[d]
           & \ar[l]_(0.55){r^{n-1}{\esm}r} \tilde{\phi}_{n-1}\widetilde{P}^{n-1}_{n-2}Y{\esm}Y \ar@{ >->}[r]^{i} \ar[d]
                             & \tilde{\phi}_{n-1}\widetilde{P}^{n-1}Y{\esm}Y \ar[d]^{\iso}
\\
X 
   & \ar[l]_{\mu} X^n 
           & \ar[l]_{r^n} \widetilde{P}^{n}_{n-1}Y \ar@{ >->}[r]^{i''} 
                             & \widetilde{P}^{n}Y }
\endxy
\end{equation}
To explain the entry at the top of the third column,
we identify the conditions on 
\begin{equation*}
(z_1, \ldots, z_n)\in  (\tilde{\phi}_{n-1}\widetilde{P}_{n-2}^{n-1}Y) \esm Y 
\quad \text{and} \quad 
(y_1, \ldots y_{n-1}) \in \Delta_{n-1}^n\tilde{\phi}_{n-1}\widetilde{P}^{n-1}Y
\end{equation*}
such that
$i(z_1, \ldots, z_n) = \Delta_{n-1}^n(y_1, \ldots, y_{n-1})$. 
We see that $z_j = y_j$ for  $ 1 \leq j \leq n{-}1$ and that there is $k$ between $1$ and $n{-}1$ such
that $z_n = y_k$.  Moreover, since no more than $n{-}2$ of the first $n{-}1$ simplices $z_j$ are distinct,
no more than $n{-}2$ of the simplices $y_j$ are distinct.  Hence, we obtain the description of 
the term at the top of the third column.  Additionally we obtain the fact that the induced map
\begin{equation*}
\xy \UCMT \xymatrix{
  \bigl(\phi_{n-1} \widetilde{P}^{n-1}_{n-2}Y{\esm}Y \bigr) 
              \cup_{\bigl(\Delta_{n-1}^n \phi_{n-1}\widetilde{P}^{n-1}_{n-2}Y \bigr)} 
                       \bigl( \Delta_{n-1}^n\phi_{n-1}\widetilde{P}^{n-1}Y \bigr) 
                           \ar@{ >->}[r] &
                                           \phi_{n-1}\widetilde{P}^{n-1}Y{\esm}Y }       
\endxy
\end{equation*}
is a cofibration, so lemma \ref{iteratedcolimlemma} applies to diagram \eqref{diagram2recursion1}. 

One takes the row-wise pushout of the three 
columns on the right and obtains the following cofibration sequence in $\cR_f(X^n)$.
\begin{equation*}
\xy \UCMT \xymatrix@C+=7ex{
 \widetilde{\Delta}_{n-1}^n \tilde{\phi}_{n-1}\widetilde{\theta}^{n-1} Y  \ar@{ >->}[r]^{\widetilde{\Delta}_{n-1}^n}
                    &   \tilde{\phi}_{n-1}\widetilde{\theta}^{n-1} Y \esm \theta^1 Y  \ar@{->>}[r] 
                           & \tilde{\phi}_n\widetilde{\theta}^n Y, }
\endxy
\end{equation*}
which is \eqref{unreducedcofibseq} from the statement. 

One also compose the arrows pointing to the left in each row and take the row-wise pushout of the resulting  diagram,
which consists of columns one, three, and four of the diagram 
\ref{diagram2recursion1}, 
obtaining
\begin{equation*}
  \xy \UCMT \xymatrix@C=10ex{
   \Delta_{n-1}^n\tilde{\phi}_{n-1}\tilde{\theta}^{n-1} Y  \ar@{ >->}[r]^{\mu_*\widetilde{\Delta}_{n-1}^n}
                    &   \phi_{n-1}{\theta}^{n-1} Y \sma \theta^1 Y  \ar@{->>}[r] 
                           & \phi_n{\theta}^n Y, }
\endxy
\end{equation*}
which is the second cofibration sequence \eqref{reducedcofibseq}
 in the statement. \end{proof}

We want to apply the cofibration sequence \eqref{reducedcofibseq} to evaluate the composite $\phi_n\theta^n$ on 
a homotopy class in $\pi_jA(X)$, where the basepoint is taken in the zero component.  Two features of algebraic $K$-theory
make this possible.  The first feature is essentially a consequence of the additivity theorem and says that cofibration
sequences imply additive relations. 
\begin{lemma} \label{additivity}
Let  $Z$ be a space. The two composite maps
\begin{equation*}
\xy  \UCMT \xymatrix{
|hS_2 \cR_f(Z)| \ar@<0.5ex>[r]^{t} \ar@<-0.5ex>[r]_{s \wed q}  
                                        &  |h \cR_f(Z) |  \ar[r] 
                                                                &  \Loops|h S_{\bullet}\cR(Z)|}
\endxy
\end{equation*}
are homotopic, where the right hand arrow is the canonical map
\begin{equation*}
|h \cR_f(Z) | \ra \Loops |hS_{\bullet} \cR_f(Z)|.  \qed
\end{equation*}
\end{lemma}
The second feature is the triviality of products in higher homotopy groups, explained as follows.
Since $X$ is a simplicial abelian group, 
the homotopy functor $Y \mapsto [Y, A(X)]$ 
has a ring structure induced from the bi-exact pairing
\begin{equation*}
  \cR(X)  \times \cR(X) \stackrel{\esm}{\lra} \cR(X{\times}X) \stackrel{\mu_*}{\lra} \cR(X).
\end{equation*}
Now suppose $Y=\Sigma Y'$ is a suspension.  Under this ring structure the product of two elements 
$[f_1]$ and $[f_2]$ in $[Y, A(X)]$ is zero, 
because $[f_1]$ may be represented by a map taking the upper cone $C_+Y'$ in $\Sigma Y'$ 
to the point in $A(X)$ represented by the zero element in $\cR_f(X)$,
while $[f_2]$ is represented by a map taking the lower cone $C_-Y'$ in $\Sigma Y'$ to the zero element.
In a similar manner, there are pairings
\begin{equation*}
  \cR(X^{n-1})  \times \cR(X) \stackrel{\esm}{\lra} \cR(X^{n-1}{\times}X) = \cR(X^n)
\end{equation*}
and these are also zero on higher homotopy groups.
Combining these observations means we have a chance to compute 
by induction the action of $\phi_n\theta^{n}$ on higher homotopy groups,
because at each stage of the induction the middle term of the relevant cofibration 
contributes nothing to the final answer. 

To start the induction,  we  compute $(\phi_2\theta^2)_*[f]$ for 
$f \colon S^j \ra A(X)$.  Applying the additivity theorem to the 
cofibration sequence \eqref{reducedcofibseq}, we can write
\begin{equation*}
  (\phi_2\theta^2)_*[f] = (\theta^1_*[f] \sma \theta^1_*[f]) - ({\Delta}^2_1\theta^1)_*[f]
\end{equation*}
For the first term on the right side of the equation, we have  observed that this product
is zero.  So we first obtain
\begin{equation} \label{firststep}
  (\phi_2\theta^2)_*[f] = - ({\Delta}^2_1\theta^1)_*[f]
  \end{equation}
 We analyse this expression as follows. First, $\phi_1$ and $\theta^1$ are identity functors. 
For $n=2$, there is one diagonal map $\delta_{1}^{2,1} \colon Z \ra \delta_1^{2,1} Z$ 
so $\widetilde{\Delta}_1^2\phi_1{\theta}^1Y = \delta_{1*}^{2,1}\phi_1{\theta}^1(Y)= \delta_{1*}^{2,1}Y$.
Then $\Delta_1^2\phi_1\theta^1 = \mu_* \com \widetilde{\Delta}_1^2\phi_1{\theta}^1 = \mu_* \com \delta_{1*}^{2,1}$,
and the point is to see what is happening with the retraction $r \colon Y \ra X$. 
Applying formula \eqref{diagonalformula2},  the composition
\begin{equation*}
  \mu \com (\widetilde{\Delta}^2_1r) \com i_1^{2,1}(y) = \mu( r(y), r(y) ) = \bigl( r(y) \bigr)^2 = (\tau_2 \com r)(y),   
\end{equation*}
where $\tau^2 \colon X \ra X$ is the squaring homomorphism.
That is, the action of $\Delta^2_1 = \mu_*\widetilde{\Delta^2_1}$ on homotopy is
the same as the action on homotopy induced by the squaring homomorphism $\tau^2$. 
Consequently,
\begin{equation*}
  (\phi_2\theta^2)_*[f] = - \tau^2_*[f].
\end{equation*}
The general result is
\begin{theorem}
  \label{generalcalculation}
Let $\tau^n \colon X \ra X$ be the homomorphism that raises elements to the
$n$th power, thinking of the operation in $X$ as multiplication.  Then
\begin{equation*}
  \phi_n \theta^n_* = (-1)^{n-1} \cdot (n{-}1)! \cdot \tau^n_*
  \colon \pi_jA(X) \ra \pi_jA(X)
\end{equation*}
for $j >0$. 
\end{theorem}
\begin{proof}
First we observe that on higher homotopy groups
\begin{equation*}
    (\phi_n \theta^n)_* = (-1)^{n-1} \cdot (\Delta_{n-1}^n\widetilde{\Delta}_{n-2}^{n-1} \cdots \widetilde{\Delta}_{1}^2)_*
\end{equation*}
An application of the cofibration sequence \eqref{reducedcofibseq} and the vanishing product principle gives 
$ (\phi_n \theta^n)_* = (-1) \cdot (\Delta_{n-1}^n\tilde{\phi}_{n-1}{\tilde{\theta}}^{n-1})_*$.
Then one continues with applications of the cofibration sequence \eqref{unreducedcofibseq} and the vanishing pairing principle.
\begin{multline*}
   (\phi_n \theta^n)_* = 
              (-1)^{2} \cdot (\Delta_{n-1}^n\widetilde{\Delta}_{n-2}^{n-1}\tilde{\phi}_{n-2}{\tilde{\theta}}^{n-2})_* = \cdots 
\\
         =   (-1)^{n-1} \cdot (\Delta_{n-1}^n\widetilde{\Delta}_{n-2}^{n-1} \cdots \widetilde{\Delta}_{1}^2)_*
         =   (-1)^{n-1} \cdot (\mu_* \widetilde{\Delta}_{n-1}^n\widetilde{\Delta}_{n-2}^{n-1} \cdots \widetilde{\Delta}_{1}^2)_*
 ,
\end{multline*}
recalling that $\tilde{\phi}_1$ and $\tilde{\theta}^1$ are identity functors.

Since the functors $\widetilde{\Delta}_{p-1}^p$ are built by unions from functors $\tilde{\delta}^{p,k}_{p-1 *}$ we have to analyse
composites
\begin{equation*}
  \delta^{n, k_{n-1}}_{n-1 *} \com \delta^{n-1, k_{n-2}}_{n-2 *} \com \cdots \delta^{2,1}_{1*} \colon \cR_f(X) \ra \cR_f(X^n)
\end{equation*}
for all choices of indices $1 \leq k_{n-1} \leq n{-}1$, $1 \leq k_{n-2} \leq n{-}2$, \ldots, $1 \leq k_2 \leq 2$.
On $(Y,r,s)$ the value of the chain is $(Y \cup_X X^n, r^n, s)$, 
where the retraction $r^n \colon  Y \ra X^n$ is evaluated by repeated application of formula \eqref{diagonalformula1}.
When we apply $\mu_*$ to this object, the value on $(Y,r,s)$ is seen to be $(Y, \tau^n \com r, s)$.
Finally, we identify the numerical coefficient $(n{-}1)!$ by counting the number of terms in the composites
$\widetilde{\Delta}_{n-1}^n\widetilde{\Delta}_{n-2}^{n-1} \cdots \widetilde{\Delta}_{1}^2$ according to the description above.
\end{proof}
\bibliographystyle{plain}
\bibliography{kahnpriddylist}
\end{document}